\documentclass[12pt]{amsart}

\usepackage[
    margin=1in
]{geometry}

\usepackage[backend=biber, 
	giveninits=true,
	hyperref=true, 
	maxcitenames=3,
	maxbibnames=100,
	block=none,
    style=alphabetic, %
     url=false, 
	isbn=false]{biblatex}	
	\AtEveryBibitem{%
        \clearfield{note}%
		\clearfield{mrclass}%
		\clearfield{mrnumber}%
		\clearfield{mrreviewer}%
		\clearfield{doi}%
	}
\usepackage{stmaryrd}
\usepackage{amssymb,
amsthm}
\usepackage{multicol}
\usepackage{mathtools}
    \mathtoolsset{showonlyrefs}
\usepackage{enumitem}
    \makeatletter
	   \def\namedlabel#1#2{\begingroup#2%
			\def\@currentlabel{#2}%
			\phantomsection\label{#1}\endgroup
		}
		\makeatother
  
\usepackage{tensor}
\usepackage{tikz-cd}
    \usetikzlibrary{calc,decorations.markings,arrows.meta}
\usepackage{upref}  
\usepackage[
    notextcomp
]{stix} 
\usepackage{hyperref}

    \newcommand\mvisiblespace[1][.7em]{%
        	\makebox[#1]{%
        		\kern.07em
        		\vrule height.3ex
        		\hrulefill
        		\vrule height.3ex
        		\kern.07em
        	}
        }

    \newcommand{\HleftG}{
        \mathchoice
              {\displaystyle\mathbin\smalltriangleright}
              {\textstyle\mathbin\smalltriangleright}
              {\scriptstyle\mathbin\smalltriangleright}
              {\scriptscriptstyle\mathbin\smalltriangleright}
        }
    
    \newcommand{\HrightG}{
        \mathchoice
              {\displaystyle\mathbin\smalltriangleleft}
              {\textstyle\mathbin\smalltriangleleft}
              {\scriptstyle\mathbin\smalltriangleleft}
              {\scriptscriptstyle\mathbin\smalltriangleleft}
        }

\newcommand{\HleftB}{
\mathchoice
{\mathbin{\raisebox{-.9pt}{$\includegraphics[scale=1.3]{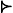}$}}}
{\mathbin{\raisebox{-.9pt}{$\includegraphics[scale=1.3]{semithickHleftB.pdf}$}}}
{\mathbin{\raisebox{-0.6pt}{$\includegraphics[scale=.8]{semithickHleftB.pdf}$}}}
{\mathbin{\raisebox{-.2pt}{$\includegraphics[scale=.6]{semithickHleftB.pdf}$}}}
}

\newcommand{\dom} {\operatorname{dom}}
\newcommand{\lspan}{\operatorname{span}}\newcommand{\ran} {\operatorname{ran}}
\newcommand{\cB}{\mathscr{B}}
    \newcommand{\qB}{q_{\cB}}
\newcommand{\cG}{\mathcal{G}}
\newcommand{\cU}{\mathcal{U}}

\newcommand{\ZS}{Zappa--Sz\'{e}p}
\newcommand{\etale}{\'{e}tale}
\newcommand{\usc}{upper semi-continuous}
\newcommand{\LCH}{locally compact Hausdorff}
\newcommand{\wordforPhi}{factorisation rule}
\newcommand{\wordforphi}{cocycle connector} 
\newcommand{\adjforBlend}{austere}

\newtheorem{theorem}{Theorem}[section]
\newtheorem{corollary}[theorem]{Corollary}
\newtheorem{proposition}[theorem]{Proposition}
\newtheorem{lemma}[theorem]{Lemma}
\newtheorem{claim}{Claim} 
\newtheorem*{claim*}{Claim} 
\newtheorem*{conjecture*}{Conjecture} 

\theoremstyle{definition}
    \newtheorem{definition}[theorem]{Definition}
    
    \newtheorem{example}[theorem]{Example}
    \newtheorem{remark}[theorem]{Remark}
    \newtheorem*{notation*}{Notation} 

\numberwithin{equation}{section}

\let\ipscriptstyle=\scriptscriptstyle
\def\lipsqueeze{{\mskip -2.0mu}}
\def\ripsqueeze{{\mskip -2.0mu}}
\newcommand{\bfp}[2]{\lipsqueeze\tensor*[_{\ipscriptstyle #1}]{\ast}{_{\ipscriptstyle #2}}\ripsqueeze} 
\newcommand{\norm}[1]{\left\| #1 \right\|}

\newcommand{\inv}{^{-1}}
\newcommand{\z}{^{(0)}}
\DeclareMathOperator{\supp}{supp}

\DeclareMathOperator{\Iso}{Iso}

\newcommand{\bfpsr}{\mathbin{\tensor[_{s}]{\times}{_{r}}}}
\newcommand{\astonetwo}{\bfp{1}{2}}
\newcommand{\asttwoone}{\bfp{2}{1}}
\newcommand{\astij}{\bfp{i}{j}}
\newcommand{\suppo}{\supp^{\circ}}
\newcommand{\red}{\mathrm{r}}
\newcommand{\full}{\mathrm{u}}
\newcommand{\ZST}[1][\Phi]{\mathbin{\bowtie_{#1}}}

\title{%
The
Zappa--Sz\'{e}p 
product of twisted groupoids}

\author{Anna Duwenig}
\address{KU Leuven, Department of Mathematics, Leuven (Belgium)}
\email{anna.duwenig@kuleuven.be}

\author{Boyu Li}
\address{Department of Mathematical Sciences, New Mexico State University, Las Cruces, New Mexico, 88003, USA}
\email{boyuli@nmsu.edu}
\date{\today}


\begin{document}

\begin{abstract}
We define and study the external and the internal \ZS\ product of 
twists  over groupoids. We determine when a pair $(\Sigma_{1},\Sigma_{2})$ of twists over a matched pair $(\cG_{1},\cG_{2})$ of groupoids gives rise to a \emph{\ZS\ twist} $\Sigma$ over the \ZS\ product $\cG_{1}\bowtie\cG_{2}$. We prove that the resulting (reduced and full) twisted groupoid C*-algebra of the \ZS\ twist $\Sigma\to \cG_{1}\bowtie\cG_{2}$  is a C*-blend of its subalgebras corresponding to the subtwists $\Sigma_{i}\to \cG_{i}$. Using Kumjian--Renault theory, we then prove a converse: Any C*-blend in which the intersection of the three algebras is a Cartan subalgebra in all of them, arises as the reduced twisted groupoid C*-algebras from such a \ZS\ twist $\Sigma\to \cG_{1}\bowtie\cG_{2}$ of two twists $\Sigma_{1}\to \cG_{1}$ and $\Sigma_{2}\to \cG_{2}$.  
\end{abstract}

\thanks{Anna Duwenig was supported by an FWO Senior Postdoctoral Fellowship (project number 1206124N), and she thanks Jonathan Taylor and Ying-Fen Lin for helpful discussions}

\maketitle

\section{Introduction}

The \ZS\ product originated from the study of composition and decomposition of groups \cite{Szep1950}. A group $K$ is called an \emph{internal \ZS\ product} of its two subgroups $G$ and $H$ if
 
    \begin{center}
    \hypertarget{item:intro:K product}{(1)}~$K=G\cdot H$, and
    \qquad
    \hypertarget{item:intro:G cap H}{(2)}~$G\cap H$ is the trivial group $\{e\}$.
    \end{center}
In this case, every element in $K$ can be written uniquely as a product $gh$ for some $g\in G$ and $h\in H$.
In particular, if we take any pair $(h',g')\in H\times G$, then there must exist a unique pair $(g,h)\in G\times H$ such that $h'g'=gh$ in $K$. This induces a right group action by $G$ on the space $H$ by $(h',g')\mapsto 
{%
h'\HrightG g'\coloneq
}%
h$
and a left group action of $H$ on the space $G$ by $(h',g')\mapsto
{%
h'\HleftG g'\coloneq
}%
g$. 
Since the group multiplication of $K$ is associative, these two actions of its subgroups $G$ and $H$ on each other interact with one-another `in a compatible way'. For example, if $h\in H$ acts on a product $g_{1}g_{2}$ of elements of $G$, then
\begin{equation}\label{eq:intro:ZS compatibility}
    h\HleftG (g_{1}g_{2}) = (h\HleftG g_{1}) \bigl([h\HrightG g_{1}]\HleftG g_{2}\bigr),
\end{equation}
and likewise for the action $(h_{1}h_{2})\HrightG g$ of $g\in G$ on a product $h_{1}h_{2}\in H$.
Conversely, any pair of such compatible actions of two groups $G$ and $H$ on one-another, allows us to equip the Cartesian product $G\times H$ with a group multiplication that encodes these actions; this group is denoted $G\bowtie H$ and is called the \emph{external \ZS\ product}.
A semi-direct product
    $G\rtimes H$
is a special case of a \ZS\ product: in this instance, the subgroup $G$ is normal in the product, it acts trivially on the subgroup $H$, and the action of $H$ on $G$ is by group homomorphisms.

We are interested in extending the notion of \ZS\ products to operator algebras. There is a long history in the study of operator algebras 
of encoding group actions and their dynamics as linear operators on Hilbert spaces. 
For example, to the continuous action of a topological group $G$ on a space $X$, 
one can associate a crossed product C*-algebra $C_0(X)\rtimes G$ which can be understood as a non-commutative replacement of the
quotient space $X/G$ which, as a topological space, might be rather ill-behaved.
Properties of the group dynamics are often reflected in the properties of the corresponding operator algebra and vice versa.
\ZS\ products encode two-way actions between two structures, and they have been studied in various algebraic contexts
\cite{ABRW2019, BKQS18, BRRW, BPRRW:ZS, EP2017, LRRW14, LRRW18, LY2022, LY2019, LY2019b, Nekrashevych2009, OP23, Starling2015}.
Our recent study of
the
\ZS\ product of Fell bundles by a groupoid \cite{DuLi:ZS} has resulted in 
vastly general imprimitivity theorems \cite{DuLi:ImprimThms-pp} that extend
many classical results arising from group dynamics. 

This paper continues our quest to study \ZS\ products in C*-algebras. Our main goal is to define the \ZS\ product of twisted groupoids.
Such
groupoids arise naturally in the study of C*-algebras. 
For instance,
Renault \cite{Renault:Cartan} showed that 
any
C*-algebra with a Cartan subalgebra 
can be realized as 
a twisted groupoid C*-algebra, and recent
work of X.\ Li \cite{XLi2020_Cartan} showed that every simple classifiable C*-algebra in Elliot's classification program has 
such
a Cartan subalgebra
and hence a twisted groupoid model.

Let us now explain the content of this paper in more detail.
Suppose we are
given two groupoids $\cG_{1}$ and  $\cG_{2}$ that share the same unit space $\cU$,
and suppose these groupoids act on one another in a `compatible way' akin to the compatibility assumption at~\eqref{eq:intro:ZS compatibility} in the group case;
to be precise,
they satisfy Conditions~\ref{item:ZS1}--\ref{item:ZS9} in our preliminary Section~\ref{ssec:ZS of gpds}.
{%
The pair $(\cG_{1},\cG_{2})$ is called \emph{a matched pair} in this situation.
}%
Then it is well known that their fibred product
\[
\cG_{1}\bfpsr \cG_{2}
=
\{(x,g)\in \cG_{1}\times\cG_{2}: s(x)=r(g)\}
\]
can be given the structure of a groupoid
whose multiplication encodes both actions and
with respect to which the canonical maps $\cG_{j}\to \cG_{1}\bfpsr\cG_{2}$ ($j=1,2$) are injective groupoid homomorphisms. The resulting groupoid is called the \emph{\ZS\ product} of $\cG_{1}$ and $\cG_{2}$ and is denoted by $\cG_{1}\bowtie\cG_{2}$. The conditions that need to be satisfied for this to work
boil down to the existence of a certain type of homeomorphism
\begin{equation}\label{eq:Psi,intro}   
\cG_{2}\bfpsr \cG_{1} \to 
\cG_{1}\bfpsr \cG_{2},
\quad
(g,x)\mapsto (g\HleftG x,g\HrightG x).
\end{equation}
The question we are concerned with in Section~\ref{sec:EZS} is the following. Suppose $\Sigma_{1}$ and $\Sigma_{2}$ are twists over $\cG_{1}$ and $\cG_{2}$, respectively; that is, 
we have central groupoid extensions
  \[\mathbb{T}\times \cU    \stackrel{\jmath_{j}}{\longrightarrow} \Sigma_{j}  \stackrel{\pi_{j}}{\longrightarrow} \cG_{j}, \]
which, by abuse of notation, we often  simply write as $\Sigma_{j}\to\cG_{j}$.
Under what conditions can one build a twist over $\cG_{1}\bowtie\cG_{2}$ out of them?

Clearly, the space $\Sigma_{1}\bfpsr \Sigma_{2}$ will not be the right candidate:
since $\Sigma_{i}$ 
is locally homeomorphic to $\mathbb{T}\times\cG_{i}$, the fibred product $\Sigma_{1}\bfpsr \Sigma_{2}$ is locally homeomorphic to $(\mathbb{T}\times\cG_{1})\bfpsr (\mathbb{T}\times\cG_{2}) = \mathbb{T}^2\times (\cG_{1}\bowtie\cG_{2})$ 
instead of $\mathbb{T}\times (\cG_{1}\bowtie\cG_{2}) $.
This crude argument already shows that
one has to quotient out a copy of $\mathbb{T}$ 
of the fibred product of $\Sigma_{1}$ and $\Sigma_{2}$ in order to have hopes of obtaining
a twist over $\cG_{1}\bowtie\cG_{2}$. 
So let us denote the quotient of $\Sigma_{1}\bfpsr \Sigma_{2}$ by the canonical
{%
diagonal
}%
$\mathbb{T}$-action by $\Sigma_{1}\ast_{\mathbb{T}}\Sigma_{2}$.

To be able to
give 
$\Sigma_{1}\ast_{\mathbb{T}}\Sigma_{2}$
the structure of a groupoid, we must assume the existence of a so-called \emph{\wordforPhi}: a certain type of homeomorphism
\[
\Phi\colon \Sigma_{2}\ast_{\mathbb{T}}\Sigma_{1}\to \Sigma_{1}\ast_{\mathbb{T}}\Sigma_{2}.
\]
As is visible from~\eqref{eq:Psi,intro}, $\Phi$ is, in its essence, a replacement for two-way actions of $\Sigma_{1}$ and $\Sigma_{2}$ on each other.
In the presence of 
such
a \wordforPhi, the space $\Sigma_{1}\ast_{\mathbb{T}}\Sigma_{2}$ can be given the structure of a topological groupoid 
which depends on $\Phi$ (Proposition~\ref{prop.external.twist}); we denote the groupoid by $\Sigma_{1} \ZST \Sigma_{2}$.
It is 
a twist over $\cG_{1}\bowtie\cG_{2}$ (Theorem~\ref{thm:external ZS-product}) and 
we hence call it the
\emph{external \ZS\ twist}.
We prove that, in the case that one of the two twists is trivial, then $\Sigma_{1}\ZST\Sigma_{2}$ is canonically isomorphic to the \ZS\ product of a Fell line bundle by a groupoid as constructed in \cite{DuLi:ZS} (Example~\ref{ex:MT if one is trivial,pt2}),
which shows that our construction behaves as expected with regard to known constructions in the literature.

In Section~\ref{sec:IZS}, we introduce a notion of \emph{internal} \ZS\ product of twists.
Suppose  we are given a twist $\Sigma\to \cG$, and let $\Sigma_{1}, \Sigma_{2}$ be subgroupoids of $\Sigma$.
 The obvious twist analogue of Condition~\hyperlink{item:intro:K product}{(1)} in the setting of group \ZS\ products, is the condition that $\Sigma=\Sigma_{1}\cdot\Sigma_{2}$. If we understand Condition~\hyperlink{item:intro:G cap H}{(2)} as the intersection of $G$ and $H$ being as small as possible, then
in the case of twists, 
it
translates to $\Sigma_{1}\cap\Sigma_{2}=\mathbb{T}\times \cU$.
Given both of these assumptions
and that
the subgroupoids $\Sigma_{1}$ and $\Sigma_{2}$ are closed,
then we say that $\Sigma$ is an \emph{internal \ZS\ product} of $\Sigma_{1}$ and $\Sigma_{2}$. 
It turns out that every such internal \ZS\ twist $\Sigma$ gives rise to a 
\wordforPhi~$\Phi$,
meaning that it is also
an external \ZS\ twist. Moreover, the converse is also true: every external \ZS\ twist 
{%
is
}
an internal \ZS\ twist (Theorem~\ref{thm:IZS=MT}).

The \ZS\ product of twisted groupoids occurs naturally in many contexts.  
In Section~\ref{sec:cocycles}, we study the setting of twists induced by continuous, normalized, $\mathbb{T}$-valued $2$-cocycles. Any such map $c$ on a groupoid $\cG$ gives rise to a twist $\Sigma_{c}$ over $\cG$ in a canonical
way. Given
$2$-cocycles $c_{1},c_{2}$ on the matched
pair $(\cG_{1},\cG_{2})$ of groupoids,
the pair  $(\Sigma_{c_{1}},\Sigma_{c_{2}})$ 
of twists
allows a \wordforPhi\  if and only if the pair $(c_{1},c_{2})$
of $2$-cocycles
allows what we call \emph{\wordforphi}, and in this situation, the external 
\ZS\ twist
is also induced by a $2$-cocycle (Proposition \ref{prop:EZS twist:2-cocycles}). We then briefly study the internal \ZS\ product of $2$-cocycles (Proposition \ref{prop:IZS twist:2-cocycles}). 

In our final Section~\ref{sec:Cartan}, we move our results into the realm of C*-algebras. We first show that the (reduced or full) groupoid C*-algebra of the \ZS\ groupoid $\cG_{1}\bowtie\cG_{2}$ twisted by the \ZS\ twist $\Sigma_{1}\ZST \Sigma_{2}$, is a C*-blend of the twisted C*-algebras of its components 
 $\Sigma_{1}\to\cG_{1}$ and  $\Sigma_{2}\to\cG_{2}$
(Theorem~\ref{thm:MT give C*-blends}). This result, applied to effective groupoids, then motivates the following converse:
suppose $A$ is a C*-algebra that has two subalgebras $A_{1}, A_{2}$ such that
    \begin{center}
    \hypertarget{item:intro:A product}{(a)}~$A_{1}\cdot A_{2}$ is dense in $A$, and
    \qquad
    \hypertarget{item:intro:A_{1} cap A_{2}}{(b)}~$A_{1}\cap A_{2}$ is a Cartan subalgebra in   $A_{1},A_{2},$ and $A$.
    \end{center}
    
Kumjian--Renault theory then tells us that each of the three algebras $A$, $A_{j}$ ($j=1,2$) corresponds to a twisted groupoid 
$\Sigma\to\cG$ and  $\Sigma_{j}\to\cG_{j}$,
respectively, with $\cG,\cG_{j}$ effective.
We prove that  
$\Sigma\to\cG$ is precisely an internal \ZS\ product of
$\Sigma_{1}\to\cG_{1}$ and  $\Sigma_{2}\to\cG_{2}$
(Theorem~\ref{thm:from C*-blend}).
 We would like to point out that Conditions~\hyperlink{item:intro:A product}{(a)} and~\hyperlink{item:intro:A_{1} cap A_{2}}{(b)} closely resemble Conditions~\hyperlink{item:intro:K product}{(1)} and~\hyperlink{item:intro:G cap H}{(2)}, respectively, in an internal \ZS\ product. Therefore, this can be viewed as an intrinsic notion of internal \ZS\ product of C*-algebras.

Our work extends the \ZS\ product construction to the much broader twisted groupoid context, which is applicable to a wider range of C*-algebras. Our notion of \wordforPhi\ brings a key insight that allows us to define the external \ZS\ product without using actions. This is a crucial step towards our ongoing research on the \ZS\ product of Fell bundles. 

\section{Preliminaries}\subsection{Groupoid and Twists}

Throughout this paper, we assume
that our groupoids, denoted by letters such as $\cG$, are
locally compact and Hausdorff. We often denote its unit space $\cG\z$ by $\cU$. A groupoid is called \emph{$r$-discrete} if the unit space $\cU$ is open in $\cG$. The range and source maps of $\cG$ are denoted by $r, s\colon \cG\to\cU$, respectively,
and $\cG$ is called \emph{\etale} if these
maps are local homeomorphisms. 

Given $u\in\cU$, denote $\cG u = \{g\in \cG: s(g)=u\}$ and $u\cG=\{g\in \cG: r(g)=u\}$. The isotropy group at $u$ is defined by $\cG_u^u = u\cG\cap \cG u$, and the isotropy bundle by $
\Iso(G)
=\bigcup_{u\in\cU} \cG_u^u$. A groupoid is called \emph{principal} if $
\Iso(G)
=\cU$, and \emph{effective} if the interior of $
\Iso(G)
$ is $\cU$.

Twists over groupoids were
first introduced in \cite{Kum:Diags}. 
Here, we will give the definition as stated in \cite[Definition 2.1]{CDGaHV:2024:Nuclear}.
 \begin{definition}\label{def:twist} Let $\cG $ be a (\LCH) groupoid with unit space $\cU$, and regard $\mathbb{T} \times \cU $ as a trivial group bundle with fibers $\mathbb{T}$.  A \emph{twist} $(\Sigma , \jmath, \pi)$ over $\cG $ consists of a \LCH\ groupoid $\Sigma $ and groupoid homomorphisms $\jmath$, $\pi$ such that 
  \[\mathbb{T}\times \cU    \stackrel{\jmath}{\longrightarrow} \Sigma  \stackrel{\pi}{\longrightarrow} \cG \]
is a central groupoid extension, which means that
\begin{enumerate}[label=\textup{(T\arabic*)}]
    \item\label{it:twist:jmath}  $\jmath\colon \mathbb{T} \times\cU\to \pi\inv (\cU)$ is a homeomorphism, where $\pi\inv (\cU)$ has the subspace topology from~$\Sigma $, and it satisfies
    $\jmath(1,\pi(u)) = u$ for all $u\in \Sigma\z$;
    \item\label{it:twist:pi} $\pi$ is a continuous, open surjection; and
    \item\label{it:twist:T central}
    $z\cdot e \coloneq  \jmath(z,\pi(r(e)))e$ equals $e\jmath( z,\pi(s(e)))$
    for all $e\in \Sigma $ and $z\in\mathbb{T}$.
\end{enumerate}
\end{definition}
As explained in \cite[p.\ 5]{CDGaHV:2024:Nuclear}, it follows that $\Sigma$ is locally trivial, that $\pi$ is proper, and that we can identify $\Sigma\z$ with $\cU$ via $\pi$.

\subsection{\ZS\ products of groupoids}\label{ssec:ZS of gpds}

The \ZS\ product of groupoids was first introduced in \cite{AA2005} and its application in groupoid C*-algebras was further explored in \cite{BPRRW:ZS}.
Here, we will recall its construction.

Given two groupoids $\cG_{1}, \cG_{2}$ with the same unit space $\cG_{1}\z=\cG_{2}\z=\cU$, define 
\[\cG_{2}\bfpsr \cG_{1} =\{(g,x)
{%
\in \cG_{2}\times\cG_{1} :
}
s(g)=r(x)\}.\]
{%
We say that $(\cG_{1},\cG_{2})$ is
}
a \emph{pair of matched groupoids} if there are two continuous maps
\begin{align*}
    \mvisiblespace\HleftG\mvisiblespace\colon &\cG_{2}\bfpsr \cG_{1} \to \cG_{1},&& (g,x) \mapsto g\HleftG x, \\
    \mvisiblespace\HrightG\mvisiblespace\colon&\cG_{2}\bfpsr \cG_{1} \to \cG_{2},&& (g,x) \mapsto g\HrightG x, 
\end{align*}
satisfying the following Properties \ref{item:ZS1}--\ref{item:ZS9}\footnote{%
We point out that we are following the numbering in \cite{DuLi:ZS}, which does not agree with that in \cite{BPRRW:ZS}.
},
where $(h,g)\in \cG_{2}^{(2)}$ and $(x,y)\in \cG_{1}^{(2)}$ are
such that $s(g)=r(x)$. 
\begin{multicols}{2}
\begin{enumerate}[label=(ZS\arabic*), leftmargin=2.5cm]
\item\label{item:ZS1} $(hg)\HleftG x=h\HleftG (g\HleftG x)$
\item\label{item:ZS2} $r(g\HleftG x)=r(g)$
\item\label{item:ZS3} $r(x)\HleftG x=x$
\item\label{item:ZS4} $g\HrightG (xy)=(g\HrightG x)\HrightG y$
\item\label{item:ZS5} $s(g\HrightG x)=s(x)$
\item\label{item:ZS6} $g\HrightG s(g)=g$
\end{enumerate}
\end{multicols}
\begin{enumerate}[label=(ZS\arabic*), leftmargin=2.5cm]\setcounter{enumi}{6}
\item\label{item:ZS7} $s(g\HleftG x)=r(g\HrightG x)$
\item\label{item:ZS8} $g\HleftG (xy)=(g\HleftG x)([g\HrightG x] \HleftG y)$
\item\label{item:ZS9} $(hg)\HrightG x=(h \HrightG [g \HleftG x]) (g\HrightG x)$
\end{enumerate}
Properties \ref{item:ZS1}--\ref{item:ZS3} mean
that
$\HleftG$ is a left $\cG_{2}$ action on the space $\cG_{1}$ 
with momentum map $r\colon \cG_{1}\to \cU=\cG_{2}\z$; Properties \ref{item:ZS4}--\ref{item:ZS6}  mean that
$\HrightG$ is a right $\cG_{1}$ action on the space $\cG_{2}$
with momentum map $s\colon \cG_{2}\to \cU=\cG_{1}\z$; and the remaining Properties \ref{item:ZS7}--\ref{item:ZS9} are needed to turn $\cG_{1}\bfpsr\cG_{2}$ into a groupoid, explained in more detail below.
One can show that the following additional properties hold:
\begin{enumerate}[label=(ZS\arabic*), leftmargin=2.5cm]\setcounter{enumi}{9}
\item\label{item:ZS10}  $g\HleftG s(g)=r(g)$
\item\label{item:ZS11}  $r(x)\HrightG x = s(x)$
\item\label{item:ZS12}  $(g\HleftG x)\inv = (g\HrightG x)\HleftG x\inv$
\item\label{item:ZS13}  $(g\HrightG x)\inv = g\inv \HrightG (g\HleftG x)$
\end{enumerate}

 
In \cite{BPRRW:ZS}, the maps $\HleftG$ and $\HrightG$ are called the \emph{action map} and the \emph{restriction map}, respectively.
We have adopted 
the notation
$g\HleftG x$ and $g\HrightG x$ from \cite{AA2005} 
in place of writing $g\cdot x$ and $g|_x$, respectively,
to avoid confusion in many computations.

With these assumptions, one can put a groupoid structure on the
fibred
product space $\cG_{1}\bfpsr \cG_{2}$ by 
defining
the multiplication
\begin{equation}\label{eq:EZS gpd:multiplication}(y,g)(x,h) = \bigl(y(g\HleftG x), (g\HrightG x) h\bigr),\quad\text{if } s(g)=r(x),\end{equation}
and inverse
\[(y,g)\inv = (g\inv \HleftG y\inv, g\inv\HrightG y\inv).\]
This is called the \emph{\ZS\ product groupoid}, which is denoted by $\cG_{1}\bowtie \cG_{2}$.

Recall from \cite[Corollary 8]{BPRRW:ZS} that
the pair $(\cG_{2}, \cG_{1})$ is also matched and that the map
\begin{align}
\label{eq:Psi}
    \Psi\colon& \cG_{2}\bfpsr \cG_{1} \to \cG_{1}\bfpsr \cG_{2},
    &\quad&
    (g,x)
    \mapsto(g\HleftG x, g\HrightG x)
\end{align}
is a groupoid isomorphism $\cG_{2}\bowtie \cG_{1} \cong \cG_{1}\bowtie \cG_{2}$.
The map $\Psi$ is called a factorisation rule in \cite[Definition 3.1]{MS:2023:ZS-pp} 
where such maps are
used to define the \ZS\ product of matched categories 
since one
can phrase Conditions~\ref{item:ZS1}--\ref{item:ZS9} entirely in terms of this map 
(Lemma~\ref{lem:ZS gpd from Psi} below). 
While this rephrasing is
a bit unwieldy,
it serves as our motivation for Definition~\ref{def:MT} later. In the following, we let $\bullet$ denote the left-action of $\cG_{1}$ (resp.\ right-action of $\cG_{2}$) on the space $\cG_{1}\bfpsr\cG_{2}$ given by translation.
\begin{lemma}[{%
cf.\ \cite[Lemma 3.4]{MS:2023:ZS-pp}%
}]\label{lem:ZS gpd from Psi}
    Suppose $\cG_{1}$ and $\cG_{2}$ are two groupoids with the same unit space $\cU$. Then $(\cG_{1},\cG_{2})$ is a pair of matched groupoids if and only if there exists a homeomorphism
    \[
    \Psi\colon \cG_{2}\bfpsr \cG_{1} \to \cG_{1}\bfpsr \cG_{2}
    \]
    satisfying the following properties, where $(g,x)\in \cG_{2}\bfpsr \cG_{1}$ is arbitrary and we write the element $\Psi(g,x)$ as $(y,h)$.
    \begin{enumerate}[label=\textup{(FR\arabic*)},leftmargin = 1.5cm] 
        \item\label{item:Psi:2+5} $r(g)=r(y)$ and $s(x)=s(h)$;
     \item\label{item:Psi:3+11,6+10} $\Psi(r(x),x)=(x,s(x))$ and $\Psi(g,s(g))=(r(g),g)$;
        \item\label{item:Psi:4+8} if $x'\in s(x)\cG_{1}$, then $\Psi(g,xx')=y\bullet\Psi(h,x')$; and
        \item\label{item:Psi:1+9} if $g'\in \cG_{2}r(g)$, then $\Psi(g'g,x)=\Psi(g',y)\bullet h$.
    \end{enumerate}
    In this setting, the relationship between the groupoid multiplication of $\cG_{1}\bowtie\cG_{2}$ and the map  $\Psi$ is determined by
    \begin{align}\label{eq:EZS gpd:multiplication via Psi}
        (y,g)(x,h) 
        =
        y\bullet \Psi(g,x) \bullet h.
    \end{align}
\end{lemma}
\begin{proof}
    If $(\cG_{1},\cG_{2})$ is a matched pair and we let $\Psi$ be the map in~\eqref{eq:Psi}, then 
    \begin{itemize}
        \item \ref{item:ZS2} and \ref{item:ZS5} imply  \ref{item:Psi:2+5};
        \item \ref{item:ZS3} and \ref{item:ZS11} imply the first equality in \ref{item:Psi:3+11,6+10};
        \item \ref{item:ZS6} and \ref{item:ZS10} imply the second equality in \ref{item:Psi:3+11,6+10};
        \item \ref{item:ZS4} and \ref{item:ZS8} imply  \ref{item:Psi:4+8}; and
        \item \ref{item:ZS1} and \ref{item:ZS9} imply \ref{item:Psi:1+9}.
    \end{itemize}
    Conversely, assume we have a map $\Psi$ satisfying the factorisation rules above. We then let $g\HleftG x \in \cG_{1}$ and $g\HrightG x \in \cG_{2}$ be the unique elements such that $\Psi(g,x)=(g\HleftG x, g\HrightG x)$. Since $\Psi$ is a homeomorphism, the maps $\HleftG$ and $\HrightG$ are well defined and continuous, and the fact that $\cG_{1}\bfpsr \cG_{2}$ is the codomain of $\Psi$  implies \ref{item:ZS7}. The other properties \ref{item:ZS1}--\ref{item:ZS6}, \ref{item:ZS8}, \ref{item:ZS9} follow as in the list above, just with reverse implication. Equation~\eqref{eq:EZS gpd:multiplication via Psi} now follows immediately from Equation~\eqref{eq:EZS gpd:multiplication}.
\end{proof}

Furthermore, we conclude with \cite[Proposition 7]{BPRRW:ZS}
(see also \cite[Proposition 3.9]{MS:2023:ZS-pp}).

\begin{lemma}\label{lem:2 unique decomp in ZS}
 If $(\cG_{1}, \cG_{2})$ is a matched pair of groupoids, then for $i\neq j$, the multiplication maps  $\cG_{i}\bfpsr \cG_{j}\to \cG_{1}\bowtie\cG_{2}$ are bijections, and if $(g,x)\in \cG_{2}\bfpsr \cG_{1}$, then
 \begin{equation}\label{eq:2 decompositions in ZS product}
     gx = (g\HleftG x) (g\HrightG x).
 \end{equation}
\end{lemma}

The groupoids $\cG_i$ embed naturally inside  their \ZS\ product $\cG_{1}\bowtie\cG_{2}$. 
    
\begin{lemma}[{%
cf.\ \cite[Lemma 3.5]{MS:2023:ZS-pp}%
}]\label{lem:clopen in ZS-product if r-discrete}
   Suppose $(\cG_{1}, \cG_{2})$ is a matched pair of groupoids. 
   The maps
    \begin{align}
    \label{eq:incl of G_{i} in G}
        \cG_{1}&\to \cG_{1}\bowtie\cG_{2}
        &&\text{and}&&&\cG_{2}&\to\cG_{1}\bowtie\cG_{2}
        \\
        \notag
        x&\mapsto (x,s(x))
        &&   &&&
        g&\mapsto (r(g),g)
    \end{align}
   are injective, continuous, closed groupoid homomorphisms; in particular, their images are closed subgroupoids of
   $\cG_{1}\bowtie\cG_{2}$. If one 
   of $\cG_{1}$, $\cG_{2}$
   is $r$-discrete, then the 
   image of the 
   other is open in $\cG_{1}\bowtie\cG_{2}$. The \ZS\ product $\cG_{1}\bowtie\cG_{2}$ is $r$-discrete (respectively even \etale) if and only if both 
   $\cG_{1}$ and $\cG_{2}$
   are $r$-discrete (respectively even \etale),
    in which case
    both of
    their images
    are open in $\cG_{1}\bowtie\cG_{2}$.
\end{lemma}

\begin{proof}
    It is clear that the maps are injective, continuous, and homomorphic. To see that the maps are closed, note that the image of $X\subseteq \cG_{1}$ is exactly $X\bfpsr\cU$ and that of $Y\subseteq \cG_{2}$ is $\cU\bfpsr Y$. Since $\cU$ is closed in $\cG_{2}$ and $\cG_{1}$, it follows that the homomorphisms map closed sets to closed sets. In particular, the maps are homeomorphisms onto their images.

    Let us prove that $\cG_{1}$ is open if $\cG_{2}$ is $r$-discrete; one proves the other case analogously. Assume that $\{k_\lambda\}$ is a net in $\cG_{1}\bowtie\cG_{2}$ that converges to 
    $(x,s(x))$.
    By continuity of the projection map $p_{2}\colon \cG_{1}\bowtie\cG_{2}\to\cG_{2}$, this implies that $p_{2}(k_\lambda)\to 
    s(x)
    $. Since $\cG_{2}$ is $r$-discrete, there exists $\lambda_0$ such that $p_{2}(k_\lambda)\in\cG_{2}\z$ for all $\lambda\geq\lambda_0$; in particular, $k_\lambda = 
    (x_\lambda, s(x_\lambda))
    $ for some 
    $x_\lambda\in\cG_{1}$.
    
    The claim about \etale ness was already proven in \cite[Proposition~9]{BPRRW:ZS}. For $r$-discreteness, note that $(\cG_{1}\bowtie\cG_{2})^{(0)} = (\cU\times\cU)\cap \cG_{1}\bowtie\cG_{2}$, so if both $\cG_{1}$ and $\cG_{2}$ are $r$-discrete, then so is their \ZS\ product. Conversely, suppose $\cG_{1}\bowtie\cG_{2}$ is $r$-discrete. If 
    $\{x_\lambda\}$
    is a net in, say, $\cG_{1}$ that converges to $u\in \cU$, then 
    $\{(x_\lambda,s(x_\lambda))\}$
    converges to the unit $(u,s(u))=(u,u)$ in $\cG_{1}\bowtie\cG_{2}$. Since $(\cG_{1}\bowtie\cG_{2})^{(0)}$ is open, the net 
    $\{(x_\lambda,s(x_\lambda))\}$
    must then eventually be contained in it, meaning that
    $x_\lambda\in \cU$,
    proving that $\cU$ is open in $\cG_{1}$. The proof for $\cG_{2}$ is identical.
\end{proof}

We therefore may identify $\cG_{i}$ with its image in $\cG_{1}\bowtie\cG_{2}$, and usually, we simply write $x$ and $g$ instead of $(x,s(x))$ and $(r(g),g)$, respectively.

\begin{corollary}\label{cor:ZS and effectiveness}
    Suppose $(\cG_{1}, \cG_{2})$ is a matched pair of groupoids. If  $\cG_{1}\bowtie\cG_{2}$ is effective and $r$-discrete, then so are $\cG_{1}$ and $\cG_{2}$. If $\cG_{1}\bowtie\cG_{2}$ is principal, then so are $\cG_{1}$ and $\cG_{2}$.
\end{corollary}

\begin{proof}
    Principality is inherited by any subgroupoid, and effectiveness by any \emph{open} subgroupoid. 
\end{proof}

Note that a converse of Corollary~\ref{cor:ZS and effectiveness} is clearly false: if, say, both $\cG_{1}$ and $\cG_{2}$ are the full equivalence relation on a set $X$ with at least two distinct elements $x,y$ (so the groupoids are principal), then the \ZS\ product for the trivial actions, i.e., $\cG_{1}\bfpsr\cG_{2}$, contains $((x,y),(y,x))$, which is a non-unit isotropy point.

\section{The external \ZS\ product of two twists}\label{sec:EZS}

For this section, we assume that $\cG_{1}$ and $\cG_{2}$ are two groupoids with the same unit space $\cU$ and that, over each of these groupoids, we have a twist
	\[
	\mathbb{T}\times \cU
	\overset{\jmath_{i}}{\to} \Sigma_{i}
	\overset{\pi_{i}}{\to}
	\cG_{i}.
	\]
	We identify $\Sigma_{i}\z$ with $\cU$ via $\pi_{i}$.
	
We will prove that certain conditions on the pair $(\Sigma_1,\Sigma_2)$ of twists  imply that $(\cG_{1},\cG_{2})$ is a matched pair of groupoids as in Section~\ref{ssec:ZS of gpds} and that allow us to define a new twist $\Sigma_{1}\ZST\Sigma_2$ over the \ZS\ product groupoid $\cG_{1}\bowtie\cG_{2}$.
Notice that each $\Sigma_i$ carries a copy of $\mathbb{T}$, so their product $\Sigma_{i}\bfpsr \Sigma_{j}$ has one too many copies of $\mathbb{T}$ in order
to be a twist over $\cG_{1}\bowtie\cG_{2}$. We thus first need to quotient out a copy of $\mathbb{T}$ from $\Sigma_{i}\bfpsr \Sigma_{j}$: We define a $\mathbb{T}$-action on $\Sigma_{i}\bfpsr \Sigma_{j}$ by $z\cdot (e,f)\coloneq  (z\cdot e, \overline{z}\cdot f)$, and we denote the resulting quotient space by $\Sigma_{i}\ast_{\mathbb{T}}\Sigma_{j}$ whose elements are written as $e\astij f$.

A very convenient tool in the proofs to come is the following:
\begin{lemma}\label{lem:external:top on quotient}
    For $i,j\in\{1,2\}$ with $i\neq j$, the quotient map $q_{i,j}\colon \Sigma_{i}\bfpsr \Sigma_{j}\to \Sigma_{i}\ast_{\mathbb{T}}\Sigma_{j}$ is open and the space $\Sigma_{i}\ast_{\mathbb{T}}\Sigma_{j}$ is locally compact Hausdorff with a $\mathbb{T}$-action given by $z\cdot (e\astij f)\coloneq  (z\cdot e)\astij f$.
\end{lemma}

\begin{proof}
    A basic open set of $\Sigma_{i}\bfpsr \Sigma_{j}$ is of the form 
    \[
        U \bfpsr  V 
        = \{(e,f)\in \Sigma_{i}\bfpsr \Sigma_{j}: e\in  U , f\in  V \},
    \]
    for open sets $ U \subseteq \Sigma_{i}, V\subseteq \Sigma_{j}$. We have
    \begin{align*}
        q_{i,j}\inv(q_{i,j}( U \bfpsr  V ))
        &=
        \{(z\cdot e,\overline{z}\cdot f)\in \Sigma_{i}\bfpsr \Sigma_{j}: e\in  U , f\in  V ,z\in\mathbb{T}\}
        \\
        &=
        \cup_{z\in\mathbb{T}}
        (z\cdot  U) \bfpsr  (\overline{z}\cdot  V ).
    \end{align*}
    Since $\mathbb{T}$ acts by homeomorphisms on $\Sigma_{i}$ and $\Sigma_{j}$, each of the sets $z\cdot  U$ and $\overline{z}\cdot  V $ is open, proving that $q_{i,j}\inv(q_{i,j}( U \bfpsr  V ))$ is the union of basic open sets and hence itself open.
    Since $\Sigma_{i}\ast_{\mathbb{T}}\Sigma_{j}$ carries the quotient space topology induced by the quotient map $q_{i,j}$, this proves that $q_{i,j}( U \bfpsr  V )$ is open in $\Sigma_{i}\ast_{\mathbb{T}}\Sigma_{j}$. As $U,V$ were arbitrary, this proves that $q_{i,j}$ is open.

   {%
   Since $\mathbb{T}$ is compact and $\Sigma_{i}\bfpsr \Sigma_{j}$ is \LCH, the quotient space is \LCH\ 
   (\cite[Chapter 4, \S 31, Ex. 8]{Munkres}).
   }%

   Since $\mathbb{T}$-acts continuously on $\Sigma_{1}$ and since the quotient map is continuous and open, it is immediate that $\mathbb{T}$ acts continuously on $\Sigma_{1}\ast_{\mathbb{T}}\Sigma_{2}$. 
\end{proof}

Next, we show that the twists $\Sigma_{1}$ and $\Sigma_{2}$ that we started with naturally embed into the quotient space $\Sigma_{1}\ast_{\mathbb{T}}\Sigma_{2}$, so that we can think of them as subspaces.

\begin{lemma}\label{lem:iotas}
    The maps
\begin{align*}
    \iota^{i}_{1,2}\colon \Sigma_{i} &\to \Sigma_{1}\ast_{\mathbb{T}}\Sigma_{2}
    &&\text{ and }&&&
    \iota^{i}_{2,1}\colon \Sigma_{i} &\to \Sigma_{2}\ast_{\mathbb{T}}\Sigma_{1}
\intertext{given by}
    \iota^{1}_{1,2} (e )&=
    e \astonetwo s(e )
    &&
    &&&
    \iota^{1}_{2,1} (e )&=
    r(e )\asttwoone e 
    \\
    \iota^{2}_{1,2} (f )&=
    r(f )\astonetwo f 
    &&
    &&& 
    \iota^{2}_{2,1} (f )&=
    f \asttwoone s(f )
\end{align*}
are 
embeddings whose images are closed.
\end{lemma} 

\begin{proof}
We will do the proof for $\iota^{1}_{1,2}$ only.
If $\iota^{1}_{1,2}(e' )= \iota^{1}_{1,2} (e )$, then there exists $z\in\mathbb{T}$ such that
\[
    (e' , s(e' ))
    = (z\cdot e , \overline{z}\cdot s(e )).
\]
But $e' =z\cdot e $ implies $s(e' )=s(z\cdot e )=
s(e)
$, so the fact that we also have $s(e' )=\overline{z}\cdot s(e )$ implies $z=1$ and hence $e' =e $, which proves that $\iota^{1}_{1,2}$ is injective.
Continuity is obvious, seeing that the map is a concatenation of continuous maps.

To see that the image is closed, assume that a net $\{\iota^{1}_{1,2}(e_\lambda)\}$ converges in $\Sigma_{1}\ast_{\mathbb{T}}\Sigma_{2}$ to, say, $q_{1,2}(e,f)=e\astonetwo f$. Since the quotient map $q_{1,2}$ is open (Lemma~\ref{lem:external:top on quotient}) and since we can without loss of generality pass to a subnet, there exists a net $\{(e_\lambda',f_\lambda)\}$ in $\Sigma_{1}\bfpsr \Sigma_{2}$ such that $e_\lambda' \astonetwo f_\lambda = \iota^{1}_{1,2}(e_\lambda)$, $e_\lambda'\to e$, and $f_\lambda\to f$. From $e_\lambda' \astonetwo f_\lambda = \iota^{1}_{1,2}(e_\lambda)$ we conclude that there exists $z_\lambda\in\mathbb{T}$ such that $(z_\lambda\cdot e_\lambda',\overline{z_\lambda}\cdot f_\lambda)=( e_\lambda,s(e_\lambda))$
in $\Sigma_{1}\bfpsr\Sigma_{2}$.
Since $\mathbb{T}$ is compact and since we can pass to a subnet, we can without loss of generality assume that $\{z_\lambda\}$ converges to some $z\in\mathbb{T}$. Since $e_\lambda'\to e$, this means that 
\[
f_\lambda
= z_\lambda \cdot s(e_\lambda)
= z_\lambda \cdot s(z_\lambda \cdot e_\lambda')\to z\cdot s(z \cdot e) =  z\cdot s(e).
\]
Since limits are unique and $f_\lambda\to f$, we conclude that $f=z\cdot s(e)$, and so
\[
    e\astonetwo f 
    =
    (z\cdot e)\astonetwo (\overline{z}\cdot f)
    =
    (z\cdot e)\astonetwo s(e)
    =
    \iota^{1}_{1,2}(z\cdot e)
    \in \iota^{1}_{1,2}(\Sigma_{1}).
\]
We have shown that the limit $e\astonetwo f $ of the net $\{\iota^{1}_{1,2}(e_\lambda)\}$ lies in $\iota^{1}_{1,2}(\Sigma_{1})$, i.e., the image of $\iota^{1}_{1,2}$ is closed.

To see that $\iota^{1}_{1,2}$ is open onto its image, we can adjust the above proof:  this time, the limit of $\{\iota^{1}_{1,2}(e_\lambda)\}$ is allowed to be chosen of the form $e\astonetwo s(e)$, i.e., the above scalar $z$ can be set to $1$, so that  $e_\lambda = z_\lambda \cdot e_\lambda' \to z\cdot e = e$, proving that $\iota^{1}_{1,2}$ is open onto its image by Fell's criterion \cite[Proposition 1.1]{Wil2019}. 
\end{proof}

If $(\cG_{1},\cG_{2})$ is a matched pair, so that the set $\cG_{1}\bfpsr\cG_{2}$ can be regarded as the \ZS\ product groupoid $\cG_{1}\bowtie\cG_{2}$, then in order for $\Sigma_{1}\ast_{\mathbb{T}}\Sigma_{2}$ to have any hopes of becoming a twist over 
$\cG_{1}\bowtie\cG_{2}$, we need there to be a projection map. That is the content of the following lemma.
\begin{lemma}\label{lem:Pi12 and its counterpart}
    The maps
    \begin{align*}
    	\Pi_{1,2}\colon \Sigma_{1}\ast_{\mathbb{T}}\Sigma_{2} &\to \cG_{1}\bfpsr \cG_{2} &\text{and}&&
     \Sigma_{2}\ast_{\mathbb{T}}\Sigma_{1} &\to \cG_{2}\bfpsr \cG_{1}
    	\\
    	e \astonetwo f &\mapsto \bigl(\pi_{1}(e ),\pi_{2}(f )\bigr)
    	&&&f \asttwoone e &\mapsto \bigl(\pi_{2}(f ),\pi_{1}(e )\bigr),
    \end{align*}
   	are well-defined, continuous,
    open
    surjections.
\end{lemma}

\begin{proof}
	We will do the proof for $\Pi_{1,2}$; by symmetry, the claim for the other map will then follow.
	
Since $\pi_{1},\pi_{2}$ are $\mathbb{T}$-invariant
and since $s(e)=r(f)$, the map $\Pi_{1,2}$ is
 well-defined.
Surjectivity follows from surjectivity of $\pi_{1}$ and $\pi_{2}$.

{%
Fell's criterion \cite[Proposition 1.1]{Wil2019} can be used to show that the map $\Pi_{1,2}$ is continuous, using that the quotient map is open by Lemma~\ref{lem:external:top on quotient} and that the maps $\pi_{1},\pi_{2}$ are continuous. Likewise, $\Pi_{1,2}$ is open since the quotient map is continuous and the maps $\pi_{1},\pi_{2}$ are open. 
}%
\end{proof}

Note that each $\Sigma_{i}$ acts continuously on itself by multiplication on both the left and right. Consequently, $\Sigma_{1}$ acts continuously on the left of
the space
$\Sigma_{1}\bfpsr \Sigma_{2}$ and $\Sigma_{2}$ acts on the right. These actions factor through the quotient:
\begin{lemma}\label{lem:anchor}
    The maps $\rho,\sigma\colon \Sigma_{1}\ast_{\mathbb{T}}\Sigma_{2}\to \cU$  given by
    \[
        \rho(e \astonetwo f )=r(e )
        \quad\text{ and }\quad
        \sigma(e \astonetwo f )=s(f )
    \]
    are  continuous surjections; if the groupoids $\cG_{i}$ are
    \etale,
    then these maps are open. On $\Sigma_{1}\ast_{\mathbb{T}}\Sigma_{2}$, we define a left $\Sigma_{1}$-action with anchor map $\rho$ and a right $\Sigma_{2}$-action with anchor map $\sigma$ by 
    \[
    e'
    \bullet (e\astonetwo f) = (
    e'
    e) \astonetwo f
    \quad\text{respectively}\quad
    (e\astonetwo f) \bullet 
    f'
     = e \astonetwo (f
    f'
     ).\]
This makes $\Sigma_{1}\ast_{\mathbb{T}}\Sigma_{2}$ a left $\Sigma_{1}$- and right $\Sigma_{2}$-space.
\end{lemma}

Similarly, $\Sigma_{2}\ast_{\mathbb{T}}\Sigma_{1}$ becomes a left $\Sigma_{2}$- and right $\Sigma_{1}$-space.

\begin{proof}
    We will do the proof for the left action; the proof for the right action is analogous. 
    Since the $\mathbb{T}$-actions on $\Sigma_{1},\Sigma_{2}$ leave  range and source invariant, $\rho$ is well defined. It is surjective as $\rho(u\astonetwo u)=u$ for any $u\in \cU$. Since the quotient map $\Sigma_{1}\bfpsr\Sigma_{2}\to\Sigma_{1}\ast_{\mathbb{T}}\Sigma_2$ is open and the map $\Sigma_{1}\bfpsr\Sigma_{2} \to \cU, (e,f)\mapsto r(e)$, is continuous, the map $\rho$ is  likewise continuous.
    If $\cG_{1}$ and $\cG_{2}$ are \etale, then their range maps are open; since $\pi_{1},\pi_{2}$ are open, an easy application Fell's criterion \cite[Proposition 1.1]{Wil2019} then proves that $\rho$ is open.

    It is clear that $\Sigma_{1}\ast_{\mathbb{T}}\Sigma_{2}$ is algebraically a left $\Sigma_{1}$-space in the sense of \cite[Definition 2.1]{Wil2019}. The action is continuous because the action of $\Sigma_{1}$ on $\Sigma_{1}\bfpsr \Sigma_{2}$ is continuous and the quotient map $\Sigma_{1}\bfpsr \Sigma_{2}\to \Sigma_{1}\ast_{\mathbb{T}}\Sigma_{2}$ is open by Lemma~\ref{lem:external:top on quotient}.
\end{proof}

%

We are now ready to make a key definition: if the pair $(\Sigma_1,\Sigma_2)$ of twists allows a \emph{\wordforPhi} as in Definition~\ref{def:MT}, then the pair $(\cG_{1},\cG_{2})$ of groupoids is matched in the sense of Section~\ref{ssec:ZS of gpds} (Lemma~\ref{lem:defn of Psi from Phi}). Moreover, we can equip $\Sigma_{i}\ast_{\mathbb{T}}\Sigma_{j}$ with the structure of a groupoid (Proposition~\ref{prop.external.twist}) with respect to which it is a twist over the \ZS\ product groupoid $\cG_{1}\bowtie\cG_{2}$ (Theorem~\ref{thm:external ZS-product}).

\begin{definition}
\label{def:MT}
        A \emph{\wordforPhi} for the pair $(\Sigma_1,\Sigma_2)$ of twists is
    a homeomorphism
    \begin{equation*}
    \Phi\colon \Sigma_{2}\ast_{\mathbb{T}}\Sigma_{1}\to \Sigma_{1}\ast_{\mathbb{T}}\Sigma_{2}
    \end{equation*}
    satisfying the following, where we write the arbitrary element  $\Phi(f\asttwoone e)$ as $e'\astonetwo f'$.

     \begin{enumerate}[label=\textup{(MT\arabic*)}, start=0, series=MTlist] 
         \item\label{it:MT:T equivariant}         
         $\Phi$ is $\mathbb{T}$-equivariant:
         {%
         $\Phi((z\cdot e)\astonetwo f)= (z\cdot e')\astonetwo f'$ for all $z\in\mathbb{T}$.  
         }%
         \item\label{it:MT:source and range}
         	We have $r(f)=r(e')$ and $s(e)=s(f')$.
         \item\label{it:MT:diagram,top}
         We have $\Phi(r(e)\asttwoone e)=e\astonetwo s(e)$ for all $e\in\Sigma_{1}$ and  $\Phi(f\asttwoone s(f))=r(f)\astonetwo f$ for all $f\in \Sigma_{2}$. In other words,
         this diagram commutes
         for $i=1,2$:
         \[
         \begin{tikzcd}
             &\Sigma_{i}
             \ar[dl,"\iota^{i}_{2,1}"']
             \ar[dr,"\iota^{i}_{1,2}"]&\\
             \Sigma_{2}\ast_{\mathbb{T}}\Sigma_{1}
             \ar[rr,"\Phi"]&&\Sigma_{1}\ast_{\mathbb{T}}\Sigma_{2}
         \end{tikzcd}
         \]       
         \item\label{it:MT:associativity, right} 
         For $\tilde{e}\in s(e)\Sigma_1$, we have
         \[
         \Phi\bigl(f\asttwoone (e\tilde{e})\bigr) = e'\bullet \Phi(f'\asttwoone \tilde{e})
         .\]
         \item\label{it:MT:associativity, left} 
         For $\tilde{f}\in\Sigma_2 r(f)$, we have
         \[
         	\Phi\bigl( (\tilde{f}f)\asttwoone e\bigr) = \Phi (\tilde{f}\asttwoone e')\bullet f'
         .\]
     \end{enumerate}
     {%
     If such a \wordforPhi\ exists, we will call $(\Sigma_1,\Sigma_2)$ a \emph{pair of matched twists}.
     }%
\end{definition}

\begin{remark}\label{rmk:Phi for units}
    \begin{enumerate}[label=\textup{(\roman*)}]
        \item  One can verify that, courtesy of Condition~\ref{it:MT:T equivariant}, Conditions~\ref{it:MT:associativity, right} and~\ref{it:MT:associativity, left} are well defined.
        \item Condition~\ref{it:MT:diagram,top} in particular implies for $u\in\cU=\cG_{i}\z$ that $\Phi(u\asttwoone u)=u\astonetwo u$.
        \item 
        In 
        contrast
        to the factorisation rule in \cite[Definition 3.3]{MS:2023:ZS-pp}, our map is defined on the quotient space $\Sigma_{2}\ast_{\mathbb{T}}\Sigma_{1}$ instead of the fibred product
        $\Sigma_{1}\bfpsr\Sigma_{2}$.
        \item
        We will see that twist \wordforPhi s are not unique; see Remark~\ref{rmk:Phi never unique}.
    \end{enumerate}        
\end{remark}

\begin{lemma}\label{lem:defn of Psi from Phi}
	If $\Phi$ is a twist~\wordforPhi\ for $(\Sigma_1,\Sigma_2)$, then there is a unique map
	$\Psi\colon \cG_{2}\bfpsr\cG_{1}\to \cG_{1}\bfpsr\cG_{2}$ that makes the diagram
	\begin{equation}\label{diag::defn of Psi from Phi}
	\begin{tikzcd}[ampersand replacement=\&, row sep = small]
		\Sigma_{2}\ast_{\mathbb{T}}\Sigma_1 \ar[rrr,"\Phi"]\ar[ddd] \& \& \&  \Sigma_1\ast_{\mathbb{T}}\Sigma_2 \ar[ddd, "\Pi_{1,2}"]
		\\
		\&  f\asttwoone e\ar[d,mapsto]\ar[ul, "\in", phantom] \&   e'\astonetwo f'\ar[d,mapsto]\ar[ur, "\in", phantom]\& 
		\\
		\& \bigl(\pi_2(f),\pi_1(e)\bigr) \&  \bigl(\pi_1(e'),\pi_2(f')\bigr) \& 
		\\
		\cG_{2}\bfpsr \cG_{1}\ar[rrr,"\Psi", dashed] \& \& \&  \cG_{1}\bfpsr \cG_{2}
	\end{tikzcd}
	\end{equation}
	commute. That is,
	\begin{align}\label{eq:defn of Psi from Phi}
		\Psi\bigl(\pi_2(f),\pi_1(e)\bigr) = \bigl(\Pi_{1,2}\circ\Phi\bigr)(f\asttwoone e).
	\end{align}
	Moreover, $\Psi$ is a factorisation rule for $(\cG_{1},\cG_{2})$.
\end{lemma}

\begin{proof}
	First, we must check that Equation~\eqref{eq:defn of Psi from Phi} determines a well-defined map $\Psi$. If $\tilde{f}\in \Sigma_2$ and $\tilde{e}\in \Sigma_1$ are such that $\pi_2(\tilde{f})=\pi_2(f)$ and $\pi_1(\tilde{e})=\pi_1(e)$, then there exist $z,w\in\mathbb{T}$ such that $\tilde{f}=z\cdot f$ and $\tilde{e}=w\cdot e$. By Condition~\ref{it:MT:T equivariant}, this means that
	\[
		\Phi(\tilde{f}\asttwoone \tilde{e}) = zw \cdot \Phi(f\asttwoone e).
	\]
	For any $a\in\Sigma_1,b\in\Sigma_2$, we have
	\[
	\Pi_{1,2}((z\cdot a)\astonetwo (w\cdot b))=\bigl(\pi_1(z\cdot a), \pi_2(w\cdot b)\bigr)=\bigl(\pi_1( a), \pi_2( b)\bigr) = \Pi_{1,2}(a\astonetwo b).
	\]
	so that
	\[
		 \Pi_{1,2}(\Phi(\tilde{f}\asttwoone \tilde{e})) = \Pi_{1,2}(\Phi(f\asttwoone e)).
	\]
	Thus, $\Psi$ only depends on $\bigl(\pi_2(f),\pi_1(e)\bigr)$, not on $f$ and $e$. 
	
	Since $\Phi$ is a homeomorphism, we can exchange the roles of $\Sigma_1$ and $\Sigma_2$ and conclude that the above argument yields a map $\cG_{1}\bfpsr\cG_{2}\to \cG_{2}\bfpsr\cG_{1}$ that is clearly inverse to $\Psi$. Since all maps of Diagram~\eqref{diag::defn of Psi from Phi} are open and continuous (see Lemma~\ref{lem:Pi12 and its counterpart}), we conclude that $\Psi$ is continuous and open (by an application of Fell's criterion \cite[Proposition 1.1]{Wil2019}). 
	
	Now fix $\bigl(\pi_2(f),\pi_1(e)\bigr)=(g,x) \in \cG_{2}\bfpsr\cG_{1}$. If we let $e'\astonetwo f'\in \Sigma_1\ast_{\mathbb{T}}\Sigma_2$ be such that $\Phi(f\asttwoone e)=e'\astonetwo f'$, then $\Psi(g,x)= \bigl(\pi_1(e'),\pi_2(f') \bigr)$. 
	Since \[r(g)=r(f)\overset{\ref{it:MT:source and range}}{=}r(e')\text{ and }s(x)=s(e)\overset{\ref{it:MT:source and range}}{=}s(f'),\] Condition~\ref{item:Psi:2+5} holds. 
	We have
	\[
		\Psi\bigl(r(e),\pi_1(e)\bigr) = \bigl(\Pi_{1,2}\circ\Phi\bigr)(r(e)\asttwoone e) \overset{\ref{it:MT:diagram,top}}{=} \Pi_{1,2}(e\astonetwo s(e)) = (\pi_1(e),s(e))
	\]
	and
	\[
		\Psi\bigl(\pi_2(f),s(f)\bigr) = \bigl(\Pi_{1,2}\circ\Phi\bigr)(f\asttwoone s(f)) \overset{\ref{it:MT:diagram,top}}{=} \Pi_{1,2}(r(f)\astonetwo f) = (r(f), \pi_2(f)),
	\]
	so $\Psi$ satisfies Condition~\ref{item:Psi:3+11,6+10}.
	If $\pi_1(\tilde{e})\in s(e)\cG_{1}$, then
	\begin{align*}
		\Psi(\pi_2(f),\pi_1(e \tilde{e}))
		&=
		\bigl(\Pi_{1,2}\circ\Phi\bigr)(f\asttwoone (e\tilde{e}))
		\overset{\ref{it:MT:associativity, right}}{=}
		\Pi_{1,2}\bigl(e'\bullet \Phi(f'\asttwoone \tilde{e})\bigr)
		\\
		&=
		\pi_{1}(e')\bullet 
		\bigl(\Pi_{1,2}\circ\Phi\bigr)(f'\asttwoone \tilde{e})
		=
		\pi_{1}(e')\bullet \Psi(\pi_2(f'),\pi_1(\tilde{e})),
	\end{align*}
	proving that Condition~\ref{item:Psi:4+8} holds. Likewise, Condition~\ref{it:MT:associativity, left} implies Condition~\ref{item:Psi:1+9}. Therefore, $\Psi$ is a factorisation rule for $(\cG_{1},\cG_{2})$.
\end{proof}

\begin{remark}\label{rmk:Psi, Pi, and Phi}
    Since $\Psi$ is a factorisation rule, it follows from Lemma~\ref{lem:ZS gpd from Psi} that $(\cG_{1},\cG_{2})$ is a matched pair and from Equation~\eqref{eq:Psi} that the two-way actions $\HleftG$ and $\HrightG$  are encoded in the formula
    \begin{equation}\label{eq:Psi,1}
        \Psi(g,x)= (g\HleftG x, g\HrightG x).
    \end{equation}
    If we think of $\cG_{1}\bowtie\cG_{2}$ as an internal \ZS\ product, then the map $\Pi_{1,2}$ from Lemma~\ref{lem:Pi12 and its counterpart} can now be written as
   \begin{align*}
   \Pi_{1,2}\colon \Sigma_{1}\ast_{\mathbb{T}}\Sigma_{2} \to \cG_{1}\bowtie \cG_{2},
   &&
   e \astonetwo f \mapsto \pi_{1}(e )\pi_{2}(f ),
   \end{align*}
   and we furthermore get a second map
   \begin{align*}
       \Pi_{2,1}\colon \Sigma_{2}\ast_{\mathbb{T}}\Sigma_{1} \to \cG_{1}\bowtie \cG_{2}
       ,
   &&f \asttwoone e \mapsto \pi_{2}(f )\pi_{1}(e ).
   \end{align*}
   Considering how $\Psi$ was defined (Equation~\eqref{eq:defn of Psi from Phi}), we can thus rewrite Equation~\eqref{eq:Psi,1} as
    \begin{align}\label{eq:Pi and Phi}
        \bigl(\Pi_{1,2}\circ\Phi\bigr)(f\asttwoone e)
      &=
      \bigl(\pi_{2}(f)\HleftG \pi_{1}(e)\bigr)\bigl( \pi_{2}(f)\HrightG \pi_{1}(e)\bigr)
      =
      \Pi_{2,1} 
   (f \asttwoone e),
    \end{align}
   so that $\Pi_{2,1}$ is a continuous, open  surjection.
\end{remark}

    Since multiplication in $\cG_{1}\bowtie\cG_{2}$ entirely captures the factorization rule $\Psi$ (see Equation~\eqref{eq:EZS gpd:multiplication via Psi}), we  can now deduce the following:
\begin{corollary}\label{cor:Phi induces a given Psi}
Suppose we are given a matched pair $(\cG_{1},\cG_{2})$. A \wordforPhi\ $\Phi$ for the pair $(\Sigma_1,\Sigma_2)$ of twists induces (via Lemma~\ref{lem:defn of Psi from Phi}) the \emph{given} factorization rule of $\cG_{1}\bowtie\cG_{2}$ if and only if $\Phi$ satisfies Equation~\eqref{eq:Pi and Phi}.
\end{corollary}

The above leads to the following definition:
\begin{definition}
    Given a matched pair $(\cG_{1},\cG_{2})$ of groupoids and two twists $\Sigma_i\to \cG_{i}$, we say that a twist \wordforPhi\ $\Phi$ \emph{covers} $\cG_{1}\bowtie\cG_{2}$ if $\Phi$ satisfies Equation~\eqref{eq:Pi and Phi}; in other words, the given factorisation rule for $(\cG_{1},\cG_{2})$ coincides with the rule constructed from $\Phi$ in Lemma~\ref{lem:defn of Psi from Phi}. 
\end{definition}

Even though it is not defined as a two-way action as in the case of the external \ZS\ product of groupoids, 
a twist \wordforPhi~$\Phi$ captures
the essence of a \ZS\ product structure,
in full analogy to how
the map $\Psi$ defined at~\eqref{eq:Psi} and its properties 
entirely capture the data needed for the \ZS\ product $\cG_{1}\bowtie\cG_{2}$ (see Lemma~\ref{lem:ZS gpd from Psi}).
The following example makes this more precise.

\begin{example}\label{ex:MT in the trivial case}
Consider the trivial twists $\Sigma_i=\mathbb{T}\times\cG_{i}$
{and assume that $(\cG_{1},\cG_{2})$ is a matched pair}.  We want to find a \wordforPhi\
   $\Phi\colon \Sigma_{2}\ast_{\mathbb{T}}\Sigma_{1}\to \Sigma_{1}\ast_{\mathbb{T}}\Sigma_{2}$ for 
   the pair $(\Sigma_1,\Sigma_2)$ that covers $\cG_{1}\bowtie\cG_{2}$. Let us denote the given factorisation rule of $\cG_{1}\bowtie\cG_{2}$ by $\Psi$. 
   Since $\Phi$
   must satisfy \ref{it:MT:T equivariant} ($\mathbb{T}$-equivariance) and 
   Equation~\eqref{eq:Pi and Phi} (see Corollary~\ref{cor:Phi induces a given Psi} and Remark~\ref{rmk:Psi, Pi, and Phi}), there is not that much choice: $\Phi$ must be of the form
    \begin{equation}\label{eq:Phi for trivial twists}
        \Phi\left(
        (z,g)\asttwoone (w,x)
        \right)
        =
        (w,y) \astonetwo \bigl(z\varphi(g,x),h\bigr)
        \text{ where }
        (y,h)=\Psi(g,x)
    \end{equation}
    for some map $\varphi\colon \cG_{2}\bfpsr\cG_{1}\to\mathbb{T}$. We will quickly argue here that $\varphi\equiv 1$ does the trick. (In Section~\ref{sec:cocycles}, we will study the more general but still special case where the twists $\Sigma_{i}$ are induced by $2$-cocycles, and we will see in Proposition~\ref{prop:EZS twist:2-cocycles} that 
    the choice $\varphi\equiv 1$
    is not the only one.)

    Recall that $\Psi$ satisfies \ref{item:Psi:3+11,6+10}, i.e.,
    $\Psi(r(x),x)=(x,s(x))$ and $\Psi(g,s(g))=(r(g),g)$
    for all $x\in \cG_{1}$ and $g\in\cG_{2}$; this means exactly that $\Phi$ satisfies Condition~\ref{it:MT:diagram,top}.
Clearly, Condition~\ref{item:Psi:4+8} of $\Psi$ corresponds to~\ref{it:MT:associativity, right} of $\Phi$,  and likewise~\ref{item:Psi:1+9} to~\ref{it:MT:associativity, left}.

\end{example}

\begin{example}[One Fell line bundle]\label{ex:MT if one is trivial}
Suppose $(\cG_{1},\cG_{2})$ is a pair of matched groupoids with unit space $\cU$ and that $\Sigma_{1}$ is a twist over $\cG_{1}$; let $L_{1}=\mathbb{C}\times_{\mathbb{T}}\Sigma_{1}$ be the associated line bundle whose elements we write as 
\begin{equation}\label{eq:[lambda,e]}
    [\lambda,e] \coloneq
    \{
        (z\lambda, \overline{z}\cdot e) : z\in \mathbb{T}
    \}.
\end{equation}
Recall that $L_{1}$ is a Fell bundle over the groupoid $\cG_{1}$ with projection map $q_{1}\colon [\lambda,e]\mapsto \pi(e)$; we let $r_{L_{1}}=r_{\cG_{1}}\circ q_{1}$ and $p\colon L_{1}^\times \to \Sigma_{1}$ the map that sends $[\lambda,e]$ for $\lambda>0$ to $e$ or, more generally, $p([\lambda,e])=\mathrm{Ph}(\lambda)\cdot e$.

If we equip $\cG_{2}$ with the trivial twist $\Sigma_{2}=\mathbb{T}\times \cG_{2}$, then a \wordforPhi\ $\Phi\colon \Sigma_{2}\ast_{\mathbb{T}}\Sigma_{1}\to \Sigma_{1}\ast_{\mathbb{T}}\Sigma_{2}$ as in Definition~\ref{def:MT} contains the same data as a \emph{$(\cG_{1},\cG_{2})$-compatible $\cG_{2}$-action} on $L_{1}$ 
\[
    \mvisiblespace \HleftB \mvisiblespace \colon \cG_{2}\bfp{s}{\rho} L_{1} \to L_{1},
    \quad
    (g,[\lambda,e]) \mapsto g\HleftB [\lambda,e],
\]
 in the sense of \cite[Definition 3.1]{DuLi:ZS} (which also can be defined for non-\etale\ groupoids; see \cite[Definition 4.1]{DuLi:ImprimThms-pp}).
Indeed, given 
$\Phi$, we let
\[
g\HleftB [\lambda,e] \coloneq  [\lambda ,f]
\text{ where $f$ is such that }f\astonetwo (1,g\HrightG\pi_{1}(e)) =
\Phi((1,g)\asttwoone e),
\]
and conversely, given a $(\cG_{1},\cG_{2})$-compatible $\cG_{2}$-action $\HleftB$ on $L_{1}$, we let
\begin{align*}
\Phi\colon &&\Sigma_{2}\ast_{\mathbb{T}} \Sigma_{1}&\to \Sigma_{1}\ast_{\mathbb{T}} \Sigma_{2},
&
(z,g) \asttwoone e
&\mapsto
        p(g\HleftB [1,e]) \astonetwo
        (z, g\HrightG \pi_{1}(e)).
\end{align*}
A detailed proof of these claims can be found in Appendix~\ref{appendix:FellBdl} .
\end{example}

We now put a groupoid structure on the quotient space $\Sigma_{1}\ast_{\mathbb{T}} \Sigma_{2}$ using the \wordforPhi\
$\Phi$. 

\begin{proposition}\label{prop.external.twist}
Given a pair $(\Sigma_{1},\Sigma_{2})$ of matched twists as in Definition~\ref{def:MT}
with \wordforPhi~$\Phi$,
 the following  makes $\Sigma_{1}\ast_{\mathbb{T}}\Sigma_{2}$ a locally compact Hausdorff groupoid
 which we will denote by $\Sigma_{1}\ZST\Sigma_{2}$:
 with $\rho$ and $\sigma$ as in Lemma~\ref{lem:anchor}, 
we declare the composable pairs to be
\[
(\Sigma_{1}
\ZST
\Sigma_{2})^{(2)}
=
\{
(\xi,\eta): \sigma(\xi)=\rho(\eta)
\}.
\]
For an element $e\astonetwo f$, we define its inverse as
    \begin{equation}\label{eq:EZS:inversion}(e\astonetwo f)\inv=\Phi(f\inv\asttwoone e\inv).
    \end{equation}
    Given a second element $e'\astonetwo f'$ 
    such that $\sigma(e\astonetwo f)=\rho(e'\astonetwo f')$,
    define
    \begin{equation}\label{eq:EZS:multiplication}
    (e\astonetwo f)(e'\astonetwo f')=e\bullet \Phi(f\asttwoone e')\bullet f'.
    \end{equation}
\end{proposition}

The reader should compare Equation~\eqref{eq:EZS:multiplication} with Equation~\eqref{eq:EZS gpd:multiplication via Psi}, which shows the relationship between the product in $\cG_{1}\bowtie\cG_{2}$ and the factorisation rule $\Psi$.

\begin{remark}
We chose the notation $\Sigma_{1}\ZST \Sigma_{2}$ in place of $\Sigma_{1}\bowtie \Sigma_{2}$ since the latter notation could be confused with a \ZS\ product of the \emph{groupoids} $\Sigma_{1}$ and $\Sigma_{2}$, forgetting their additional structure as twists over $\cG_{1},\cG_{2}$, 
and since the groupoid structure highly depends on $\Phi$ (see Remark~\ref{rmk:Phi never unique}).
\end{remark}

Before we prove Proposition~\ref{prop.external.twist}, we first show that the multiplication map in Equation~\eqref{eq:EZS:multiplication} can be written in some other forms.

\begin{lemma}\label{lm.external.multiplication} The multiplication map in Equation~\eqref{eq:EZS:multiplication} can be also written as
\begin{align*}
    (e\astonetwo f)(e'\astonetwo f') 
    &= e\bullet \Phi(f\bullet \Phi\inv(e'\astonetwo f')) \\
    &= \Phi(\Phi\inv(e\astonetwo f) \bullet e') \bullet f'
\end{align*}   
\end{lemma}
\begin{proof} Suppose $\Phi\inv(e'\astonetwo f')=a\asttwoone b$. We have by Condition~\ref{it:MT:associativity, right} that
 \[\Phi(f\asttwoone e')\bullet f' = \Phi(fa \asttwoone b).\]
Therefore, 
 \[e\bullet \Phi(f\asttwoone e')\bullet f' = e\bullet \Phi(f\bullet(a \asttwoone b)) = e\bullet \Phi(f\bullet \Phi\inv(e'\asttwoone f')).\]
The other equality follows similarly. 
\end{proof}

\begin{proof}[Proof of Proposition~\ref{prop.external.twist}]
It is clear that the inversion formula~\eqref{eq:EZS:inversion} is well defined. To see that the multiplication formula~\eqref{eq:EZS:multiplication} is well defined, just note that Condition~\ref{it:MT:T equivariant} implies 
that
for all $z,w\in\mathbb{T}$,
\begin{align*}
    [z\cdot e]\bullet \Phi\bigl( [\overline{z}\cdot f]\asttwoone e'\bigr)\bullet f'
    &=
    e\bullet\Phi(f\asttwoone e')\bullet f'
    \\&=
    e \bullet \Phi\bigl( f\asttwoone [w\cdot e']\bigr)\bullet [\overline{w}\cdot f']
    .
\end{align*}
Moreover, we can recover the range of an element $e\astonetwo f\in\Sigma_{1}\ZST  \Sigma_{2}$ as 
    \begin{align}
        (e\astonetwo f)(e\astonetwo f)\inv
        &= (e\astonetwo f)\Phi(f\inv \asttwoone e\inv) 
        \notag
        \\
        &= e \bullet \Phi\bigl(f\bullet \Phi\inv \Phi(f\inv \asttwoone e\inv)\bigr) 
        &&\text{by Lemma~\ref{lm.external.multiplication}}
        \notag
        \\
        &= e \bullet \Phi(ff\inv \asttwoone e\inv)
        \notag
        \\
        &= e \bullet \Phi(s(e) \asttwoone e\inv)
        && \text{since $s(e)=r(f)$}
        \notag\\
        &= e \bullet  (e\inv \asttwoone r(e))
        &&\text{by Condition~\ref{it:MT:diagram,top}}
        \notag\\
        &= r(e) \asttwoone r(e),
        \label{eq:range}
\intertext{and its source as} 
        (e\astonetwo f)\inv (e\astonetwo f) &= \Phi(f\inv \asttwoone e\inv) (e\astonetwo f)
        \notag\\
        &= \Phi\bigl(\Phi\inv\Phi(f\inv \asttwoone e\inv) \bullet e\bigr) \bullet f && \text{by Lemma~\ref{lm.external.multiplication}}\notag\\
        &= \Phi(f\inv \asttwoone r(f)) \bullet f
        && \text{since $s(e)=r(f)$}
        \notag \\
        &= (s(f) \astonetwo f\inv)\bullet f && \text{by Condition~\ref{it:MT:diagram,top}} \notag\\
        &= s(f)\astonetwo s(f). 
        \label{eq:source}
    \end{align}
In other words, the unit space of $\Sigma_{1}\ZST \Sigma_{2}$ is given by 
\begin{align*}
    \left(\Sigma_{1}\ZST \Sigma_{2}\right)\z
    &=
    \{r(e) \asttwoone r(e) : e\astonetwo f\in\Sigma_{1}\ast_{\mathbb{T}} \Sigma_{2}\}
    \\&=\{s(f)\astonetwo s(f): e\astonetwo f\in\Sigma_{1}\ast_{\mathbb{T}} \Sigma_{2} \}.
\end{align*}
To prove that multiplication is associative, we take $e_{1}, e_{2}, e_3\in \Sigma_{1}$ and $f_{1}, f_{2}, f_3\in\Sigma_{2}$ with appropriate ranges and sources so that  the product $(e_{1}\astonetwo f_{1})(e_{2}\astonetwo f_{2})(e_3\astonetwo f_3)$ makes sense. Denote $\Phi(f_{1}\asttwoone e_{2})=a_{1}\astonetwo b_{1}$ and $\Phi(f_{2}\asttwoone e_3)=a_{2}\astonetwo b_{2}$. We have that
\begin{align*}
     \left((e_{1}\astonetwo f_{1})(e_{2}\astonetwo f_{2})\right)(e_3\astonetwo f_3) 
    &= (e_{1}a_{1}\astonetwo b_{1} f_{2}) (e_3\astonetwo f_3) \\
    &= e_{1}a_{1} \bullet \Phi(b_{1}f_{2} \asttwoone e_3) \bullet f_3 \\
    &= e_{1}a_{2} \bullet \left(\Phi(b_{1}\asttwoone a_{2}) \bullet b_{2} \right) \bullet f_3
    &&\text{%
    by \ref{it:MT:associativity, left}%
    },
\end{align*}
where the last equality follows from Condition~\ref{it:MT:associativity, left} using $\Phi(f_{2}\asttwoone e_3)=a_{2} \ast b_{2}$. On the other hand, a similar computation shows that
\begin{align*}
    (e_{1}\astonetwo f_{1})\left((e_{2}\astonetwo f_{2})(e_3\astonetwo f_3) \right)
    &= (e_{1}\astonetwo f_{1}) (e_{2} a_{2}\astonetwo b_{2}f_3) \\
    &= e_{1} \bullet \Phi(f_{1} \asttwoone e_{2}a_{2}) \bullet b_{2}f_3 \\
    &= e_{1} \bullet \left(a_{2}\bullet \Phi(b_{1}\asttwoone a_{2}) \right) \bullet b_{2}f_3.
\end{align*}
Both of these equal $(e_{1}a_{2}) \bullet \Phi(b_{1}\asttwoone a_{2}) \bullet (b_{2}f_3)$, and thus the multiplication is associative. 

We have already seen in Lemma~\ref{lem:external:top on quotient} that the space is locally compact Hausdorff. Continuity of the inversion and multiplication maps follows directly from openness and continuity of the quotient map, from continuity of the inversion in $\Sigma_{1}$ and $\Sigma_{2}$ respectively of the actions in Lemma~\ref{lem:anchor}, and from continuity of~$\Phi$.
\end{proof}

Now that we have a groupoid structure on $\Sigma_{1}\ZST \Sigma_{2}$, we must check that our inclusion maps of Lemma~\ref{lem:iotas} are homomorphisms.

\begin{lemma}\label{lem:external:r and s}
In the setting of Proposition~\ref{prop.external.twist}, 
    the maps $\iota^{1}_{1,2}$ and $\iota^{2}_{1,2}$ of Lemma~\ref{lem:iotas} are groupoid homomorphisms and
    the unit space of the groupoid $\Sigma_{1}\ZST \Sigma_{2}$ is homeomorphic to $\cU$. With this identification,  range and source map are given by the maps from Lemma~\ref{lem:anchor}, i.e.,
\begin{equation}
    \label{eq:r and s for Lambda}
r(e\astonetwo f)=r(e)
\quad\text{ and }\quad
s(e\astonetwo f)=s(f).
\end{equation}
\end{lemma}
\begin{proof}
 If $(e, e')\in\Sigma_{1}^{(2)}$, then
    \begin{align*}
        \iota^{1}_{1,2}(e)\iota^{1}_{1,2}( e')
        &=
        (e\astonetwo s(e))
        ( e'\astonetwo s( e'))
        \overset{\eqref{eq:EZS:multiplication}}{=}
        e\bullet \Phi(s(e) \asttwoone  e')\bullet s(e)
        \\
       &=
        e\bullet \Phi(\iota^{1}_{2,1} ( e'))\bullet s(e)
        \overset{\ref{it:MT:diagram,top}}{=}
        e\bullet ( e'\astonetwo s( e'))\bullet s(e)
        =
        \iota^{1}_{1,2}(e e')
        ,
        \intertext{and}
        \iota^{1}_{1,2}(e)\inv
        &=
        (e\astonetwo s(e))\inv 
        \overset{\eqref{eq:EZS:inversion}}{=}
        \Phi (s(e)\asttwoone e\inv)
        \overset{\ref{it:MT:diagram,top}}{=}
        \iota^{1}_{1,2}(e\inv).
    \end{align*}
    One likewise proves that 
    $\iota_{1,2}^2$ is a homomorphism.

    The maps $\iota^{i}_{1,2}\colon \Sigma_{i}\to \Sigma_{1}\ZST \Sigma_{2}$
clearly coincide on $\cU=\Sigma_{1}\z=\Sigma_{2}\z$. The resulting map $u\mapsto u\astonetwo u\in \Sigma_{1}\ZST \Sigma_{2}$ has inverse the map $\Pi_{1,2}$ from 
Lemma~\ref{lem:Pi12 and its counterpart} restricted to $(\Sigma_{1}\ZST \Sigma_{2})\z$, which is continuous. This proves that $u\mapsto u\astonetwo u\in \Sigma_{1}\ZST \Sigma_{2}$ is an embedding, i.e., a homeomorphism onto its image; we can henceforth identify the unit space of $\Sigma_{1}\ZST \Sigma_{2}$ with $\cU$. The claim about the range and source maps now follows from Equations~\eqref{eq:range} and~\eqref{eq:source}.
\end{proof}

We now prove that $\Sigma_{1}\ZST  \Sigma_{2}$ is a twist over the \ZS\ product $\cG_{1}\bowtie\cG_{2}$.  

\begin{theorem}\label{thm:external ZS-product}
     Given a pair $(\Sigma_{1},\Sigma_{2})$ of matched twists 
     with \wordforPhi\ $\Phi$
     as in Definition~\ref{def:MT} over a pair $(\cG_{1},\cG_{2})$, 
     the sequence
    \begin{align*}
        \mathbb{T}\times \cU \overset{\jmath}{\to} \Sigma_{1}\ZST \Sigma_{2}\overset{\Pi_{1,2}}{\to}\cG_{1}\bowtie\cG_{2}
        \end{align*}{where}\begin{align*}
            \jmath(z,u)
        \coloneq 
        z\cdot (u\astonetwo u)
        =\jmath_{1}(1,u)\astonetwo \jmath_{2}(z,u)
        =
        \jmath_{1}(z,u)\astonetwo \jmath_{2}(1,u),
    \end{align*}
    makes the groupoid $\Sigma_{1}\ZST \Sigma_{2}$ a twist over $\cG_{1}\bowtie\cG_{2}$. We call it the \emph{external \ZS\ twist}.
\end{theorem}

\begin{remark}\label{rmk:MT in the trivial case,p2}
    In the setting where $\Sigma_{1}$ and $\Sigma_{2}$ are both trivial and the \wordforPhi\ is the map at~\eqref{eq:Phi for trivial twists} (see Example~\ref{ex:MT in the trivial case}), the groupoid we recover is (canonically isomorphic to) the Cartesian product $\mathbb{T}\times (\cG_{1}\bowtie\cG_{2})$.
\end{remark}

\begin{remark}  By construction, the restriction $(\Sigma_{1}\ZST  \Sigma_{2})|_{\cG_{j}}\coloneq \Pi_{1,2}\inv(\cG_{j})$ of the external \ZS\ twist to one of the closed subgroupoids $\cG_{j}$ of $\cG_{1}\bowtie\cG_{2}$ is canonically isomorphic  to $\Sigma_{j}$.
To be precise,
in the following diagram, the dashed arrow is the obvious map onto the $j$-th component, it is a groupoid isomorphism, and the diagram commutes.
    \[
    \begin{tikzcd}[column sep = huge, row sep = small]
    &(\Sigma_{1}\ZST \Sigma_{2})|_{\cG_{j}}\ar[rd,"\Pi_{1,2 }|"]\ar[dd,dashed, "\cong"]&
    \\
    \mathbb{T}\times\cU 
    \ar[ru,"\jmath|"]
    \ar[rd,"\jmath_{j}"']&    
    &
    \cG_{j}
    \\    
   &
    \Sigma_{j}
    \ar[ru,"\pi_{j}"']&
    \end{tikzcd}
    \]
In particular, even though $\cG_{j}$ does not need to be open in $\cG_{1}\bowtie\cG_{2}$(see Lemma~\ref{lem:clopen in ZS-product if r-discrete}), the restriction of the \ZS\ twist to $\cG_{j}$ is a twist.
\end{remark}

\begin{proof}[Proof of Theorem~\ref{thm:external ZS-product}]
We have seen in Proposition~\ref{prop.external.twist} that $\Sigma_{1}\ZST \Sigma_{2}$ is a locally compact Hausdorff groupoid and in 
{Lemma~\ref{lem:Pi12 and its counterpart}} that $\Pi_{1,2}$ is a continuous, open surjection. To see that it is a homomorphism, note first that since $\pi_{1},\pi_{2}$ are homomorphisms, we have
\begin{align*}
    \Pi_{1,2}\bigl(a\bullet (e\astonetwo f)\bullet b\bigr)
    &=
    \Pi_{1,2}\bigl((a e)\astonetwo (f b)\bigr)
    =
    \pi_{1}(ae)\ \pi_{2}(f b)
    \\&=
    \pi_{1}(a)\pi_{1}(e)\ \pi_{2}(f)\pi_{2}( b)
    =
    \pi_{1}(a)\ \Pi_{1,2}(e\astonetwo f) \ \pi_{2}(b),
\end{align*}
which implies
\begin{align*}
    \Pi_{1,2} (e\astonetwo f)\Pi_{1,2}(e'\astonetwo f')
    &=
    [\pi_{1}(e)\pi_{2}(f)][\pi_{1}(e')\pi_{2}(f')]
    &&\text{by def'n of }\Pi_{1,2}
    \\
    &=
    \pi_{1}(e)\ \Pi_{2,1}(f \asttwoone e') \ \pi_{2}(f')
    &&\text{by def'n of }\Pi_{2,1}
    \\
    &=
    \pi_{1}(e)\ (\Pi_{1,2}\circ \Phi)(f \asttwoone e') \ \pi_{2}(f')
    &&\text{by~
    {Eq.~\eqref{eq:Pi and Phi}}}
    \\&=
    \Pi_{1,2}\bigl(e\bullet \Phi(f\asttwoone e')\bullet f'\bigr)&&\text{by the above}
    \\&=
    \Pi_{1,2} \bigl( (e\astonetwo f)(e'\astonetwo f')\bigr)
    &&\text{by \eqref{eq:EZS:multiplication},}
\end{align*}
as needed.

 Since the $\mathbb{T}$-action on $\Sigma_{1}\ZST \Sigma_{2}$ is continuous (Lemma~\ref{lem:external:top on quotient}), the injective map $\jmath$ is continuous. To see that  $\jmath$ is a homomorphism, we first note that for appropriate $u,v\in\cU$ and all $z\in\mathbb{T}$,
 \begin{align*}
    \jmath_{1}(z,u)\bullet (e\astonetwo f)
    &=
    (\jmath_{1}(z,u) e)\astonetwo f
    =
    (z\cdot e)\astonetwo f
    \\
    &=
    e\astonetwo (z\cdot f)
    =
    e\astonetwo (f \jmath_{2}(z,v) )
    =
    (e\astonetwo f)\bullet \jmath_{2}(z,v),
\end{align*}
which are all equal to $ z\cdot (e\astonetwo f)$ as defined in Lemma~\ref{lem:external:top on quotient}. Therefore,
 \begin{align*}
     ([z_{1}\cdot u]\astonetwo u)([z_{2}\cdot u]\astonetwo u)
     &=
     \jmath_{1}(z_{1}, u)\bullet \Phi(u \asttwoone [z_{2}\cdot u])\bullet \jmath_{2}(1,u)
     \\
     &=
     z_{1}\cdot \Phi(u \asttwoone [z_{2}\cdot u]) &&\text{by the above}
     \\
     &=
     z_{1}\cdot [z_{2}\cdot \Phi(u \asttwoone u)]
     &&\text{by~\ref{it:MT:T equivariant}}
     \\
     &=
     [z_{1}z_{2}]\cdot (u\astonetwo u)
     &&\text{by~\ref{it:MT:diagram,top}}
     .
 \end{align*}
 Next, we show that the range of $\jmath$ is $\Pi_{1,2}\inv (\cU)$. Since $\Pi_{1,2}(\jmath(z,u))=\Pi_{1,2}([z\cdot u]\astonetwo u) = \pi_{1}(z\cdot u)\pi_{2}(u)=u$, we see that $\jmath (\mathbb{T}\times\cU)\subseteq \Pi_{1,2}\inv(\cU)$. Conversely, if $e\astonetwo f\in \Pi_{1,2}\inv(\cU)$, then $\pi_{1}(e)\pi_{2}(f)\in \cU$, so $\pi_{1}(e)=\pi_{2}(f)\inv$. Since $\cG_{1}\cap\cG_{2}=\cU$ in $\cG_{1}\bowtie\cG_{2}$, this implies $\pi_{1}(e)=\pi_{2}(f)=u\in \cU$. Since $\Sigma_{1},\Sigma_{2}$ are twists, this implies that there exist $z_{1},z_{2}\in\mathbb{T}$ such that $e=\jmath_{1} (z_{1},u)$ and $f=\jmath_{2} (z_{2},u)$. Therefore, $e\astonetwo f= \jmath (z_{1}z_{2},u)$, which proves that $\jmath (\mathbb{T}\times\cU)\supseteq \Pi_{1,2}\inv(\cU)$.

    It remains to show that the injective map $\jmath$ is a homeomorphism onto its image. So suppose $\{\jmath(z_\lambda,u_\lambda)\}_{\lambda}$ is a net converging to $j(z,u)$ in $\Sigma_{1}\ZST \Sigma_{2}$;
    since it suffices for us to show that a subnet of $\{(z_\lambda,u_\lambda)\}$ converges to $(z,u)$ in $\mathbb{T}\times\cU$, we will frequently pass to subnets in the following argument (or, in other words, without loss of generality assume that the subnet we pass to is the net we started with). Since the quotient map is open (Lemma~\ref{lem:external:top on quotient}), 
    (a subnet of)
    $\{\jmath(z_{\lambda},u_{\lambda})\}$ allows a lift to a net $\{(e_{\lambda}, f_{\lambda})\}_{\lambda}$ in $\Sigma_{1}\bfpsr \Sigma_{2}$ such that 
    $e_{\lambda}\to \jmath_{1}(z,u)$ and $ f_{\lambda}\to \jmath_{2}(1,u)$.
    Being a lift means that
    \[
    e_{\lambda}\astonetwo f_{\lambda} = \jmath(z_{\lambda},u_{\lambda})
    =
    \jmath_{1}(z_{\lambda},u_{\lambda})\astonetwo \jmath_{2}(1,u_{\lambda})
    \]
    for all $\lambda$, so the definition of the equivalence relation says that there exists $w_\lambda\in\mathbb{T}$ such that
    \[
    e_{\lambda}
    =
    \jmath_{1}(\overline{w_\lambda}
    z_\lambda
    ,u_{\lambda})
    \quad\text{and}\quad
    f_{\lambda} = \jmath_{2}(
    w_\lambda
    ,u_{\lambda}).
    \]
    Since  $\jmath_{1}$ and $\jmath_{2}$ are embeddings, the convergences
    \(
    f_{\lambda} \to \jmath_{2}(
    1
    ,u)
    \)
    and
    \(
    e_\lambda \to  \jmath_{1}(z
    ,u)
    \)
    thus imply that $(
    w_\lambda
    ,u_\lambda)\to (1
    ,u)$ and $( 
    \overline{w_\lambda}z_{\lambda}
    ,u_\lambda)\to (z
    ,u)$ in $\mathbb{T}\times\cU$. We conclude that (a subnet of) $\{(z_\lambda,u_\lambda)\}$ converges to $(z,u)$, proving that $\jmath$ is an embedding.
\end{proof}

\begin{example}[One Fell line bundle; continuation of Example~\ref{ex:MT if one is trivial}]\label{ex:MT if one is trivial,pt2}
Suppose again that $(\cG_{1},\cG_{2})$ is a pair of matched groupoids with unit space $\cU$,  and that $\Sigma_{1}$ is a twist over $\cG_{1}$ with associated line Fell bundle $L_{1}=\mathbb{C}\times_{\mathbb{T}}\Sigma_{1}$. If we let $\Sigma_{2}=\mathbb{T}\times\cG_{2}$
be the trivial twist,
then we have seen that a \wordforPhi~$\Phi$ for $(\Sigma_{1}, \Sigma_{2})$ can be translated into a $(\cG_{1},\cG_{2})$-compatible $\cG_{2}$-action $\HleftB$ on $L_{1}$. By \cite[Proposition 3.5]{DuLi:ZS}, such an action gives rise to a \ZS\ product Fell bundle $\cB=(\qB\colon L_{1}\bowtie \cG_{2} \to \cG_{1}\bowtie\cG_{2})$ that encodes the action; we claim that this bundle $\cB$ is (canonically isomorphic to) the  Fell line bundle $L$ that is associated to the \ZS\ twist $\Sigma_{1}\ZST \Sigma_{2}$,
whose elements are of the form $[\lambda, e\astonetwo (z,g)]$ for $\lambda\in\mathbb{C}$, $e\in\Sigma_{1}$, $z\in\mathbb{T}$, and $g\in \cG_{2}$.

If we denote the elements of $L_{1}$ by $[\lambda,e]$, as before in Equation~\eqref{eq:[lambda,e]}, then the total space of $\cB$ is defined as
\[
L_{1}\bowtie \cG_{2}
\coloneq
\left\{
    ([\lambda,e],g) \in L_{1}\times \cG_{2}
    :
    s_{\Sigma_{1}}(e)=r_{\cG_{2}}(g)
\right\}
\]
with the subspace topology of $L_{1}\times \cG_{2}$; the projection map $\qB$ maps $([\lambda,e],g)$ to $(\pi_{1}(e),g)\in \cG_{1}\bowtie\cG_{2}$; the 
Fell bundle
multiplication is given 
by
\begin{align}\label{eq:Fbl multiplication}
    ([\lambda,e ],g )([\lambda',e'],g')
    &\coloneq
    \bigl(
        [\lambda,e ](g \HleftB [\lambda',e']), (g \HrightG \pi_{1}(e')) g'
    \bigr)
    ,
\end{align}
where $s_{\cG_{1}}(g)=r_{\Sigma_{1}}(e')$;
and the $*$-operation is given by
\[
    ([\lambda,e ],g )^*
    \coloneq
    \bigl(
        g\inv \HleftB [\lambda,e ]^*,
        g\inv \HrightG \pi_{1}(e)\inv
    \bigr)
    .
    \]
We show in Appendix~\ref{appendix:FellBdl} that the map
\begin{align}\label{eq:Omega}
    {
    \Omega
    }
    \colon 
    L=\mathbb{C}\times_{\mathbb{T}} (\Sigma_{1}\ZST \Sigma_{2}) \to 
    L_{1}\bowtie \cG_{2} 
    ,
    \quad
    \bigl[
        \lambda, e \astonetwo (z,g )
    \bigr]
    \mapsto 
    \bigl(
    [\lambda z,  e ], g 
    \bigr)
    ,
\end{align}
is an isomorphism of Fell bundles.
\end{example}

We would like to point out that
the paper at hand was
motivated by our pursuit of the \ZS\ product of two Fell bundles. 
Definition~\ref{def:MT} suggests that 
an external \ZS\ product should be thought of in terms of the \wordforPhi. An
analogous map
can be defined in the more general setting of Fell bundles. We shall leave the construction of the \ZS\ product of two Fell bundles to future work.

\section{The internal \ZS\ product of two twists}\label{sec:IZS}

We now turn our attention to the other 
face
of a \ZS\ product: the internal product.

\begin{definition}
\label{def:IZS} 
    An \emph{internal \ZS\ structure} for a twist $(\Sigma,\jmath,\pi)$  over a groupoid $\cG $ with unit space $\cU$ is a tuple $(\Sigma_{1},\Sigma_{2})$ of two closed subgroupoids of $\Sigma$ such that
     \begin{enumerate}[label=\textup{(I\arabic*)}]
     \item\label{it:IZS:product} 
     $\Sigma=\Sigma_{1}\cdot\Sigma_{2}$ and
     \item\label{it:IZS:cap} $\Sigma_{1}\cap \Sigma_{2}=\jmath(\mathbb{T}\times \cU)$.
     \end{enumerate}
\end{definition}

Compared with the internal \ZS\ product of groups (see also \cite{BPRRW:ZS} for groupoids), Condition~\ref{it:IZS:product} is a natural replacement of the condition that a group is a product of its two subgroups
(see Condition~\hyperlink{item:intro:K product}{(1)} in the introduction),
and Condition~\ref{it:IZS:cap} is the analogue of the intersection of the subgroups being trivial
(Condition~\hyperlink{item:intro:G cap H}{(2)} \emph{loc.\ cit.}).
Indeed, 
in order for each $\Sigma_i$ to be a twist itself, it
must contain a copy of $\jmath(\mathbb{T}\times\cU)$, 
so this 
is the smallest intersection one can impose on these two subgroupoids $\Sigma_i$. 

We first outline several properties of the internal \ZS\ structure.

\begin{lemma}\label{lem:decomposing e_{2}e_{1}}
    Suppose $(\Sigma_{1},\Sigma_{2})$ is an internal \ZS\ structure for a twist $(\Sigma,\jmath,\pi)$ over a groupoid~$\cG$. Then the following hold for $i=1,2$.
    \begin{enumerate}[label=\textup{(I\arabic*)}, start=3]
        \item\label{it:IZS:order} $(\Sigma_{2},\Sigma_{1})$ is also an internal \ZS\ structure for $(\Sigma,\jmath,\pi)$;
        \item\label{it:IZS:saturated} $\Sigma_{i} = \pi\inv(\pi(\Sigma_{i}))$ and
        $\Sigma_{i}\z = \cU$;
        \item\label{it:IZS:closed} $\cG_{i}\coloneq  \pi(\Sigma_{i})$ is a closed subgroupoid of $\cG$;
        \item\label{it:IZS:ZS} $\cG=\cG_{1}\bowtie\cG_{2}$;
        \item\label{it:IZS:twist} the triple $(\Sigma_{i} , \jmath, \pi|_{\Sigma_{i}})$ is a twist over $\cG_{i} $;
        \item\label{it:IZS:mult open} the map $\Sigma_{i}\bfpsr \Sigma_{j}\to\Sigma, (e,f)\mapsto ef,$ is open.
    \end{enumerate}
    For $(f, {e})\in \Sigma_{2}\bfpsr \Sigma_{1}$, we further have the following.
    \begin{enumerate}[label=\textup{(I\arabic*)}, start=9]
        \item\label{it:decomposing e2e1:existence} There exist $(e',f')\in\Sigma_{1}\bfpsr \Sigma_{2}$ such that $f{e}=e'f'$.
        \item\label{it:decomposing e2e1:uniqueness} The choice of $(e',f')$ 
        in \ref{it:decomposing e2e1:existence}
        has only one degree of freedom; to be more precise,
        \[
        \{( e'', f'')\in\Sigma_{1}\bfpsr \Sigma_{2}: f{e}= e'' f''\}=\{(z\cdot  e',\overline{z}\cdot f'):z\in\mathbb{T}\}.
        \]
    \end{enumerate}
\end{lemma}

\begin{proof}

    \ref{it:IZS:order} For $e\in \Sigma$, use \ref{it:IZS:product} to write $e\inv = e_{1}e_{2}$ for some $e_{i}\in\Sigma_{i}$, so that $e=e_{2}\inv e_{1}\inv$ and hence $\Sigma=\Sigma_{2}\cdot\Sigma_{1}$.
    
    \ref{it:IZS:saturated} The containment $\Sigma_{i} \subseteq \pi\inv(\pi(\Sigma_{i}))$ is obvious. For the reverse containment, note that \ref{it:IZS:cap} implies $\jmath(\mathbb{T}\times\cU)\subseteq \Sigma_{i}$, so in particular, $\Sigma_{i}\z=\Sigma\z$. Since $\Sigma_{i}$  is a subgroupoid, we further conclude for any $e\in\Sigma_{i}$ and $z\in\mathbb{T}$ that $z\cdot e=\jmath(z,r(e))e\in\Sigma_{i}$, which proves $\Sigma_{i} \supseteq \pi\inv(\pi(\Sigma_{i}))$.

    \ref{it:IZS:closed} Since $\pi$ is a groupoid homomorphism, $\cG_{i}$ is a subgroupoid. Now suppose $\{g_\lambda\}$ is a net in $\cG_{i}$ that converges to $\pi(e)=g$ in $\cG$. Since $\pi$ is open, Fell's criterion \cite[Proposition 1.1]{Wil2019} says that (a subnet of) $\{g_\lambda\}$ can be lifted under $\pi$ to a net $\{e_\lambda\}$ in $\Sigma$ such that $e_\lambda\to e$. As $\pi(e_\lambda)=g_\lambda\in \cG_{i}$, we have $e_\lambda\in\pi\inv (\cG_{i})=\Sigma_{i}$ by~\ref{it:IZS:saturated}. Since $\Sigma_{i}$ is closed by assumption, the limit $e$ of the net must be contained in $\Sigma_{i}$, which proves that $\pi(e)=g$ is contained in $\pi(\Sigma_{i})=\cG_{i}$, as claimed. (Coincidentally, this proof also shows that $\pi_{i}\coloneq  \pi|_{\Sigma_{i}}\colon \Sigma_{i}\to \cG_{i}$ is an open map, which we require for \ref{it:IZS:twist}.)
    
    \ref{it:IZS:ZS} 
     Since  $\cG_{i}$ is a closed subgroupoid of the \LCH\ groupoid  $\cG$ by \ref{it:IZS:closed}, it is itself \LCH.
    Since $\pi$ is a groupoid homomorphism, it follows from Condition~\ref{it:IZS:product} that
    \[
        \cG=\pi(\Sigma) = \pi(\Sigma_{1}\cdot\Sigma_{2}) = \pi(\Sigma_{1})\cdot\pi(\Sigma_{2})=\cG_{1}\cdot\cG_{2}.
    \]
    Note that
    \begin{equation}
        \label{eq:G1 cap G2}
        \cG_{1}\cap\cG_{2}
        =
        \pi(\Sigma_{1})\cap \pi(\Sigma_{2})
        =
        \pi(\Sigma_{1}\cap \Sigma_{2})
        \overset{\ref{it:IZS:cap} }{=}
        \pi(\jmath(\mathbb{T}\times \cU))
        \overset{\ref{it:twist:jmath}}{=}
        \cU.
    \end{equation}
    This implies that every element $k\in \cG$ has a \emph{unique} decomposition as a product $k=x g$ with $x\in\cG_{1},g\in\cG_{2}$. By \cite[Proposition 7]{BPRRW:ZS}, $\cG$ is therefore the internal \ZS\ product $\cG=\cG_{1}\bowtie\cG_{2}$.

    \ref{it:IZS:twist} By the same argument as in \ref{it:IZS:ZS}, $\Sigma_{i}$ is a \LCH\ groupoid. The maps  $\jmath_{i}\coloneq  \jmath \colon \mathbb{T}\times\cU\to \Sigma_{i}\subseteq\Sigma$ and $\pi_{i}$ are clearly continuous groupoid homomorphisms.
    Since
    \[
        \pi_{i}\inv (\cU)
        =
        \pi\inv(\cU)\cap \Sigma_{i}
        \overset{\ref{it:twist:jmath}}{=}
        \jmath(\mathbb{T}\times\cU)\cap \Sigma_{i}
        \overset{\ref{it:IZS:cap}}{=}
        \jmath(\mathbb{T}\times\cU)
        \overset{\ref{it:twist:jmath}}{=}
        \pi\inv(\cU),
    \]
    we conclude that the subspace topology that $\pi_{i}\inv(\cU)=\pi\inv(\cU)$ inherits from $\Sigma_{i}$ is identical to the topology that it inherits from $\Sigma$. In other words, $\jmath_{i}$
    satisfies \ref{it:twist:jmath} because $\jmath$ does. Furthermore, \ref{it:twist:T central} clearly holds as well by assumption on $(\Sigma,\jmath,\pi)$. Since we have seen in the proof of \ref{it:IZS:closed} that $\pi_{i}$ is open, we have shown all conditions in Definition~\ref{def:twist}.

    \ref{it:IZS:mult open} Suppose that $\{\sigma_\lambda\}$ is a net in $\Sigma$ that converges to $ef$ for $e\in\Sigma_{i},f\in\Sigma_{j}$; it suffices to lift a subnet of $\{\sigma_\lambda\}$
    under the multiplication map to nets $\{e_\lambda'\}$ in $\Sigma_{1}$ and $\{f_{\lambda}'\}$ in $\Sigma_{2}$ such that $e_\lambda'\to e$ and $f_{\lambda}'\to f$.
    By \ref{it:IZS:product} (applied to either $(\Sigma_{1},\Sigma_{2})$ or $(\Sigma_{2},\Sigma_{1})$, using \ref{it:IZS:order}), for every $\lambda$, there exists $e_\lambda\in\Sigma_{i}$ and $f_\lambda\in\Sigma_{j}$ such that $\sigma_\lambda=e_\lambda f_\lambda$. Since $\pi$ is a continuous groupoid homomorphism, we have
    \[
    \pi(e_\lambda)\pi(f_\lambda)=\pi(\sigma_\lambda)\to\pi(ef)=\pi(e)\pi(f)
    \]
    in $\cG$. As $\pi(e_\lambda)\in\cG_{i}$ and $\pi(f_\lambda)\in\cG_{i}$, this implies together with \ref{it:IZS:ZS} that
    \[
    \pi(e_\lambda)\to \pi(e)
    \quad\text{and}\quad
    \pi(f_\lambda)\to\pi(f).
    \]
    Since $\pi$ is open and since we can pass to a subnet, there exist without loss of generality  $z_\lambda,w_\lambda\in\mathbb{T}$ such that
    \[
    z_\lambda \cdot e_\lambda \to e
    \quad\text{and}\quad
    w_\lambda\cdot f_\lambda\to f.
    \]
    By compactness of $\mathbb{T}$ (and since we can pass to subnets), we can without loss of generality assume that $z_\lambda\to z$ and $w_\lambda\to w$ for some $z,w\in \mathbb{T}$; in particular, 
    \begin{align*}
        e_\lambda'\coloneq  z\cdot e_\lambda =(z \overline{z_\lambda})\cdot(z_\lambda \cdot e_\lambda)
        &\to (z \overline{z})\cdot e = e
        &&\text{and}
        \\
        f_\lambda'\coloneq  w\cdot f_\lambda =(w \overline{w_\lambda})\cdot(w_\lambda \cdot f_\lambda)&\to (w \overline{w})\cdot f = f
        ,
    \end{align*}
    so that
    \[
        ef \leftarrow (z\cdot e_\lambda) (w\cdot f_\lambda) = (zw)\cdot \sigma_\lambda \to (zw)\cdot ef.
    \]
    In particular, $zw=1$ as $\Sigma$ is Hausdorff. By \ref{it:IZS:saturated}, the net $(e_\lambda',f_\lambda')$ is in $\Sigma_{i}\bfpsr \Sigma_{j}$, converges to $(e,f)$, and since $zw=1$, it satisfies
    \[
    e_\lambda' f_\lambda' = (z\cdot e_\lambda)(w\cdot f_\lambda)= (zw)\cdot (e_\lambda f_\lambda) = \sigma_{\lambda}.
    \]
    This proves that the surjective map $\Sigma_{i}\bfpsr \Sigma_{j}\to\Sigma, (e,f)\mapsto ef,$ is open.

    \ref{it:decomposing e2e1:existence} Since $(f, {e})\in \Sigma_{2}\bfpsr \Sigma_{1}$, we see that    $\pi(f{e})$ is an element of $\cG_{2}\cdot \cG_{1}$. Since $\cG=\cG_{1}\bowtie\cG_{2}$, 
    there exist unique $g_{i}\in \cG_{i}$ such that $\pi(f{e})=g_{1}g_{2}$. We can let $f_{i}\in\Sigma_{i}$ be any lifts under $\pi_{i}$ of $g_{i}$. Since $\pi(f{e})=g_{1}g_{2}=\pi(f_{1}f_{2})$, there exists a unique $z\in \mathbb{T}$ such that $fe=(z\cdot f_{1})f_{2}$. Note that $ e'\coloneq  z\cdot f_{1}\in \Sigma_{1}$ and $ f'\coloneq  f_{2}\in\Sigma_{2}$, so we have written $fe$ as a product of an element in $\Sigma_{1}$ and $\Sigma_{2}$, as claimed.

    \ref{it:decomposing e2e1:uniqueness} We clearly have $\supseteq$. For the converse containment, assume $fe=e'' f''$.
    Then 
    $e'' f''=e'f'$, so that
    ${e'}\inv e'' = f'{f''}\inv$ is an element of $\Sigma_{1}\cap\Sigma_{2}$. By \ref{it:IZS:cap}, this means that there exists $z\in\mathbb{T}$ such that $z\cdot  e'= e''$ and $z\cdot  f''= f'$.
\end{proof}

As in the case of 
any construction of a \ZS\ product
(e.g., \cite[Proposition 7]{BPRRW:ZS} and \cite[Proposition 3.13]{MS:2023:ZS-pp}), the external and internal \ZS\ structures are, in fact,
two sides of the same coin.
The same is true for our \ZS\ twists, as
is summarized in the following main theorem of this section. 

\begin{theorem}\label{thm:IZS=MT}
There is a one-to-one correspondence between internal \ZS\ structures (Definition~\ref{def:IZS}) and matched twists (Definition~\ref{def:MT}). To be more precise:
\begin{enumerate}[label=\textup{(\arabic*)}]
    \item\label{it:thm:from MT to IZS} A pair $(\Sigma_{1},\Sigma_{2})$ of matched twists
    with \wordforPhi\ $\Phi$
    is an internal \ZS\ structure for the external \ZS\ twist $(\Sigma_{1} \ZST \Sigma_{2},\jmath,\Pi_{1,2})$ of Theorem~\ref{thm:external ZS-product}.
    \item\label{it:thm:from IZS to MT} 
    For any internal \ZS\ structure $(\Sigma_{1},\Sigma_{2})$ of a twist $(\Sigma,\jmath,\pi)$
    {%
    over a groupoid $\cG$,
    }
    {%
    the map $\Phi\colon \Sigma_{2}\ast_{\mathbb{T}}\Sigma_{1}\to \Sigma_{1}\ast_{\mathbb{T}}\Sigma_{2}$ determined by
    \[
        \Phi( f \asttwoone e )
        =
        e' \astonetwo f'
        \text{ whenever }
        fe=e'f'
        \text{ in } \Sigma,
    \]
    is well-defined. It is furthermore the unique \wordforPhi\
    }
    for $(\Sigma_{1},\Sigma_{2})$
    {%
    that covers $\cG=\pi(\Sigma_{1})\bowtie\pi(\Sigma_{2})$
    }%
    such that the map
    \[
        \Sigma_{1}\ZST \Sigma_{2}\to \Sigma,\quad
        e\astonetwo f \mapsto ef,
    \]
    is an isomorphism of twists, meaning that 
    the following diagram commutes.
    {%
    \[
    \begin{tikzcd}[column sep = huge, row sep = small]
    &\Sigma_{1}\ZST \Sigma_{2}\ar[r,"\Pi_{1,2}"]\ar[dd,dashed]& \pi(\Sigma_{1})\bowtie\pi(\Sigma_{2})
    \ar[dd,"\cong"]
    & \ar[l, phantom, "\ni"] (x,g)\ar[dd, mapsto]
    \\
    \mathbb{T}\times\cU 
    \ar[ru,"\jmath"]
    \ar[rd,"\jmath"']&    
    &&
    \\    
   &
    \Sigma
    \ar[r,"\pi"']& \cG & \ar[l, phantom, "\ni"]  xg
    \end{tikzcd}
    \]
    }%
\end{enumerate}  
\end{theorem}

In Part~\ref{it:thm:from MT to IZS}, one could be pedantic and say that $(\iota^{1}_{1,2}(\Sigma_{1}),\iota^{2}_{1,2}(\Sigma_{2}))$ is the internal \ZS\ structure for $(\Sigma_{1} \ZST \Sigma_{2},\jmath,\Pi_{1,2})$, but courtesy of Lemmas~\ref{lem:iotas} and~\ref{lem:external:r and s}, we can identify $\Sigma_{i}$ with its homeomorphic, closed image $\iota^{i}_{1,2}(\Sigma_{i})$.

\begin{proof}   
    \ref{it:thm:from MT to IZS} Suppose $(\Sigma_{1},\Sigma_{2})$ is a pair of matched twists
    {%
    with \wordforPhi\ $\Phi$.
    }%
    By Lemmas~\ref{lem:iotas} and~\ref{lem:external:r and s}, $\Sigma_{1},\Sigma_{2}$ are (homeomorphic to) closed subgroupoids of  $\Lambda\coloneq  \Sigma_{1}\ZST \Sigma_{2}$. Since
    \begin{align*}
        e\astonetwo f
        &=        
        e\bullet (s(e)\astonetwo r(f))\bullet f
        =
        e\bullet \Phi(s(e) \asttwoone r(f))\bullet f 
        &&\text{by Remark~\ref{rmk:Phi for units}}
        \\
        &=
        (e\astonetwo s(e))(r(f)\astonetwo f)
        =
        \iota^{1}_{1,2}(e)
        \iota^{2}_{1,2}(f),
    \end{align*}    we see that $\Lambda=\iota^{1}_{1,2}(\Sigma_{1})\iota^{2}_{1,2}(\Sigma_{2})$, i.e., Condition~\ref{it:IZS:product} is satisfied. Moreover,
    an element of $\Lambda$ is in $    \iota^{1}_{1,2}(\Sigma_{1})\cap \iota^{2}_{1,2}(\Sigma_{2})
    $ if it can be written both as $e\astonetwo s(e)$ for some $e\in\Sigma_{1} $ and as $r(e)\astonetwo f$ for some $f\in \Sigma_{2}$; in particular, there exists $z\in\mathbb{T}$ such that $(e, s(e))= (z\cdot r(e), \overline{z}\cdot f)$. We conclude that the ranges and sources of $e$ and $f$ are all the same element, say $u$, so that we can write $e=z\cdot u=f$ and conclude that the element we started with can be written as $e\astonetwo s(e)=z\cdot (u\astonetwo u)=\jmath (z,u)$. This proves $ \iota^{1}_{1,2}(\Sigma_{1})\cap \iota^{2}_{1,2}(\Sigma_{2})
    \subseteq \jmath(\mathbb{T}\times\cU)$, i.e., Condition~\ref{it:IZS:cap}.

    \ref{it:thm:from IZS to MT} Suppose $(\Sigma_{1},\Sigma_{2})$ is an internal \ZS\ structure of a twist $(\Sigma,\jmath,\pi)$. Because of \ref{it:decomposing e2e1:existence}, we may use the Axiom of Choice to define a map $\Phi_0\colon \Sigma_{2}\bfpsr \Sigma_{1}\to \Sigma_{1}\bfpsr \Sigma_{2}$ by letting it send a tuple $(f,e)$ to any element $(e',f')$ for which $fe=e'f'$ in $\Sigma$. Now, if $\Phi(z\cdot f,\overline{z}\cdot e)=(\tilde{\sigma}_{1}, \tilde{\sigma}_{2})$, then
    \begin{align*}
        \tilde{\sigma}_{1}\tilde{\sigma}_{2}
        =
        (z\cdot {f})(\overline{z}\cdot e)
        =
        fe
        =
        e'f',        
    \end{align*}
    so that ${e'}\inv \tilde{\sigma}_{1}=f'\tilde{\sigma}_{2}\inv \in \Sigma_{1}\cap\Sigma_{2}$. By \ref{it:IZS:cap}, this means there exists $w\in\mathbb{T}$ such that $w\cdot e'=\tilde{\sigma}_{1}$ and $w\cdot\tilde{\sigma}_{2}= f'$. In other words, $\tilde{\sigma}_{1}\astonetwo \tilde{\sigma}_{2} =  e'\astonetwo  f'$ in $\Sigma_{1}\ast_{\mathbb{T}}\Sigma_{2}$.
    This proves that $\Phi_0$ induces a map
    \[
    \Phi\colon \Sigma_{2}\ast_{\mathbb{T}}\Sigma_{1}\to \Sigma_{1}\ast_{\mathbb{T}}\Sigma_{2}.
    \]
    Thanks to \ref{it:decomposing e2e1:existence}, this map is injective.
    And while $\Phi_0$ was not unique, this map is.
    
    To see that $\Phi$ is continuous, assume that $\{\xi_\lambda\}$ is a net in $\Sigma_{2}\ast_{\mathbb{T}} \Sigma_{1}$ that converges to $\xi$; it suffices to show that a subnet of $\{\Phi(\xi_\lambda)\}$ converges to $\Phi(\xi)$. Pick any $(e,f)\in \Sigma_{2}\bfpsr \Sigma_{1}$ such that $\xi=e
    \asttwoone
     f$. Since the quotient map $q_{2,1}$ is open (Lemma~\ref{lem:external:top on quotient}) and since we can pass to subnets, we can without loss of generality assume that there exists a net $\{(e_\lambda, f_\lambda)\}$ in $\Sigma_{2}\bfpsr \Sigma_{1}$ such that $\xi_\lambda=e_\lambda
    \asttwoone
    f_\lambda$ and $(e_\lambda,f_\lambda)\to (e,f)$; in particular
    \(
    e_\lambda f_\lambda \to ef
    \). 
    Now pick any $(\sigma,\mu)\in \Sigma_{1}\bfpsr \Sigma_{2}$ such that $ef=\sigma\mu$.
    Since the multiplication map $\Sigma_{1}\bfpsr\Sigma_{2}\to \Sigma$ is surjective by \ref{it:IZS:product} and open by
    \ref{it:IZS:mult open} and since we can, again, pass to subnets, can without loss of generality lift the convergent net $\{e_\lambda f_\lambda\}$, i.e., there exists a net $\{(\sigma_\lambda,\mu_\lambda)\}$ in $\Sigma_{1}\bfpsr \Sigma_{2}$ such that $(\sigma_\lambda,\mu_\lambda) \to (\sigma,\mu)$ and $\sigma_\lambda \mu_\lambda = e_\lambda f_\lambda$. The latter means exactly that $\Phi(\xi_\lambda) = \sigma_\lambda \astonetwo \mu_\lambda$, and the former implies that $\Phi(\xi_\lambda)\to \sigma\astonetwo \mu = \Phi(\xi)$. This proves that $\Phi$ is continuous.

    Since $(\Sigma_{2},\Sigma_{1})$ is also a matched pair by~\ref{it:IZS:order}, the same argument yields an injective continuous map $\Sigma_{1}\ast_{\mathbb{T}}\Sigma_{2}\to \Sigma_{2}\ast_{\mathbb{T}}\Sigma_{1}$ which, by construction, is inverse to $\Phi$, so $\Phi$ is a homeomorphism. We claim that $\Phi$ is a \wordforPhi, i.e., satisfies the conditions in Definition~\ref{def:MT}.

    If $\Phi(e_{2}\asttwoone e_{1})=\sigma_{1}\astonetwo\sigma_{2}$, then
    $e_{2}e_{1}=\sigma_{1}\sigma_{2}$ 
    by definition of $\Phi$. This implies
    $(z\cdot e_{2})e_{1}=(z\cdot \sigma_{1})\sigma_{2}$,
    so that 
    $\Phi((z\cdot e_{2})\asttwoone e_{1})=(z\cdot \sigma_{1})\astonetwo\sigma_{2}=z\cdot\Phi(e_{2}\asttwoone e_{1})$, proving Condition~\ref{it:MT:T equivariant}.
    Condition~\ref{it:MT:source and range} follows from $r(f)=r(fe)=r(e'f')=r(e')$ and $s(e)=s(fe)=s(e'f')=s(f')$.
    Since $es(e)=e=r(e)e$, we further get Condition~\ref{it:MT:diagram,top}. Lastly, if $f_{i}\in\Sigma_{i}$ with appropriate range and source, then since $(e_{2}e_{1})f_{1}=\sigma_{1}(\sigma_{2}f_{1})$ and $f_{2}(e_{2}e_{1})=(f_{2}\sigma_{1})\sigma_{2}$, we conclude that
    \begin{align*}
        \Phi(e_{2}\asttwoone e_{1}f_{1})=\sigma_{1}\bullet \Phi(\sigma_{2}\asttwoone f_{1})
        \quad\text{and}\quad
        \Phi(f_{2}e_{2}\asttwoone e_{1})=\Phi(f_{2}\asttwoone \sigma_{1})\bullet \sigma_{2},
    \end{align*}
    which proves Conditions~\ref{it:MT:associativity, right} and~\ref{it:MT:associativity, left}, so that
    $\Phi$ is indeed a \wordforPhi\ and $(\Sigma_{1},\Sigma_{2})$ is
    a pair of matched twists.

    It remains to prove the claim about the map $e\astonetwo f\mapsto ef$. It is well defined since centrality of the $\mathbb{T}$-action on $\Sigma$ implies
    \(
        e_{1}e_{2} = (z\cdot e_{1})(\overline{z}\cdot e_{2})
    \). It is now easy to check that the map is a groupoid homomorphism. It is a bijection by Lemma~\ref{lem:decomposing e_{2}e_{1}}\ref{it:decomposing e2e1:uniqueness} (the only degree of freedom was the choice of $\mathbb{T}$). 
    The isomorphism is continuous (respectively open) since the quotient map is open (respectively continuous) and the map $\Sigma_{1}\bfpsr \Sigma_{2}\to\Sigma, (e,f)\mapsto ef$, is continuous (respectively open by \ref{it:IZS:mult open}). It is obvious that the diagram commutes since $\pi$ is a homomorphism.
\end{proof}

As mentioned earlier, the \ZS\ product $\cG_{1}\bowtie\cG_{2}$ of two groupoids is canonically isomorphic to $\cG_{2}\bowtie\cG_{1}$. A similar results holds for twists:

\begin{corollary}[cf.\ {\cite[Corollary 9]{BPRRW:ZS}}]
Suppose $(\Sigma_{1}, \Sigma_{2})$ is a pair of matched twists with \wordforPhi\ $\Phi\colon \Sigma_{2} \ast_{\mathbb{T}} \Sigma_{1} \to \Sigma_{1} \ast_{\mathbb{T}} \Sigma_{2}$. Then $(\Sigma_{2}, \Sigma_{1})$ is a pair of matched twists with \wordforPhi\ $\Phi\inv$. Moreover, $\Phi$ and $\Phi\inv$
are groupoid isomorphisms that make the diagram
\[
    \begin{tikzcd}[column sep = huge, row sep = small]
    &\Sigma_{1}\ZST \Sigma_{2}\ar[rd,"\Pi_{1,2}"]
    \ar[dd,dashed, shift left = 1ex,"\Phi\inv"]&
    \\
    \mathbb{T}\times\cU 
    \ar[ru,"\jmath"]
    \ar[rd,"\jmath"']&    
    &
    \cG_{1}\bowtie\cG_{2}
    \\    
   &
    \Sigma_{2}\ZST[\Phi\inv] \Sigma_{1}
    \ar[ru,"\Pi_{2,1}"']\ar[uu,dashed,shift left = 1ex, "\Phi"]&
    \end{tikzcd}
    \]
commute. In other words, the associated external \ZS\ twists are isomorphic as twists.
\end{corollary}

\begin{proof}
    Let us write $\Sigma\coloneq  \Sigma_{1}\ZST \Sigma_{2}$.
    Since $(\Sigma_{1},\Sigma_{2})$ is a pair of matched twists, it follows from Theorem~\ref{thm:IZS=MT}\ref{it:thm:from MT to IZS} that $(\Sigma_{1},\Sigma_{2})$ is an internal \ZS\ structure for $(\Sigma,\jmath,\Pi_{1,2})$. 
    Thus by~\ref{it:IZS:order}, $(\Sigma_{2},\Sigma_{1})$ is also an internal \ZS\ structure for $(\Sigma,\jmath,\Pi_{1,2})$. By Theorem~\ref{thm:IZS=MT}\ref{it:thm:from IZS to MT}, 
    this implies that $(\Sigma_{2},\Sigma_{1})$ is a matched pair; it remains to find its \wordforPhi.
    
    Since the multiplication of $\Sigma$ given at~\eqref{eq:EZS:multiplication} is constructed exactly such that 
    \[\iota^{2}_{1,2}(f)\iota^{1}_{1,2}(e_{1})=(r(f)\astonetwo f)(e\astonetwo s(e)) = \Phi(f\asttwoone e),
    \]
    we can conclude two things: Since the \wordforPhi\ 
    $\Phi'\colon %
    \Sigma_{1}\ast_{\mathbb{T}}\Sigma_{2}\to \Sigma_{2}\ast_{\mathbb{T}}\Sigma_{1}$
    associated to the matched pair $(\Sigma_{2},\Sigma_{1})$ is  determined by sending  $\sigma \astonetwo \mu 
    \in \Sigma_{1}\ast_{\mathbb{T}}\Sigma_{2}
    $ to the unique element $f \astonetwo e 
    \in \Sigma_{2}\ast_{\mathbb{T}}\Sigma_{1}
    $  for which $\sigma \mu =f e $ in $\Sigma$, it follows from
    $\sigma \mu = fe = \Phi(f\asttwoone e)$ that 
    $\Phi'$
    must indeed be given by $\Phi\inv$. Secondly, the isomorphism from $\Sigma_{2}\ast_{\mathbb{T}}\Sigma_{1}$ to $\Sigma=\Sigma_{1}\ZST\Sigma_{2}$ that Theorem~\ref{thm:IZS=MT}\ref{it:thm:from IZS to MT} gives us, is exactly the map~$\Phi$.
\end{proof}

\section{\ZS\ Product of $2$-cocycles}\label{sec:cocycles}

One natural way of constructing 
twists over
groupoids is via $2$-cocycles. 
Recall that a \emph{(normalized) $2$-cocycle} on a groupoid $\cG$ is a continuous map $c\colon \cG^{(2)}\to\mathbb{T}$ such that $c(r(g), g)=1=c(g, s(g))$ and
such that
the \emph{cocycle identity} holds:
\begin{equation}\label{eq:cocycle identity}
    c(g, hk)c(h,k)=c(gh, k)c(g,h)\quad\text{ for all }  (g,h,k)\in\cG^{(3)}.  
\end{equation}
Such a $2$-cocycle defines a twist $\Sigma_{c}$ as follows. As a topological space, $\Sigma_{c}$ is just $\mathbb{T}\times \cG$, and its composition and inversion is defined for $(g,h)\in\cG^{(2)}$ and $z,w\in\mathbb{T}$ by 
\[(z,g)(w,h) = ( c(g,h) zw, gh)
\quad\text{and}\quad
(z,g)\inv = \left(\overline{z \, c(g,g\inv)}, g\inv\right).\]
If we let $\cU=\cG\z$, then the central groupoid extension is given by
\[
    \mathbb{T}\times \cU    \overset{\mathrm{incl}}{\longrightarrow} \Sigma_{c} \overset{\mathrm{pr}_{2}}{\longrightarrow} \cG
\]

In this section, we study the \ZS\ product of two twisted groupoids arising from $2$-cocycles. We start with a matched pair $(\cG_{1}, \cG_{2})$ of groupoids, each of which carries a continuous $2$-cocycle $c_{i}$.

\begin{definition}\label{def:connector}
We call a
continuous
map $\varphi\colon\cG_{2}\bfpsr \cG_{1}\to\mathbb{T}$ 
a \emph{\wordforphi\ for the pair $(c_{1},c_{2})$} if the following hold.
\begin{enumerate}[label=\textup{(CC\arabic*)}]
    \item\label{it:connector:normalized} For all $g\in \cG_{2}$ and $x\in \cG_{1}$, $\varphi(r(x), x)=1=\varphi(g, s(g))$.
    \item\label{it:connector:cocycle cond,c1} For all $g\in \cG_{2}$ and $(x,y)\in \cG_{1}^{(2)}$ such that $s(g)=r(x)$, we have
    \begin{align*}\varphi(g,xy)c_{1}(x,y)=\varphi(g,x)\varphi(g\HrightG x, y)c_{1}(g\HleftG x, (g\HrightG x)\HleftG y)
    .
    \end{align*}
    \item\label{it:connector:cocycle cond,c2} For all $(g,h)\in \cG_{2}^{(2)}$ and $x\in \cG_{1}$ such that $s(h)=r(x)$, we have
    \[\varphi(gh, x)c_{2}(g,h) = \varphi(h,x)\varphi(g, h\HleftG x) c_{2}(g\HrightG (h\HleftG x), h\HrightG x).
    \]
\end{enumerate}
\end{definition}

Definition~\ref{def:connector} should be compared to \cite[Definition~7.12]{MS:2023:ZS-pp}; a \wordforphi\ $\varphi$ corresponds in spirit exactly to their map $\varphi_{1,1}$, even if their set-up is rather different from ours.

\begin{remark}[Sanity checks]
    The reader is encouraged to check that the components in Conditions~\ref{it:connector:cocycle cond,c1} and~\ref{it:connector:cocycle cond,c2} make sense. For~\ref{it:connector:cocycle cond,c1}, for example, one requires 
    Conditions~\ref{item:ZS8}, \ref{item:ZS2}, and~\ref{item:ZS7}
    of the matched pair $(\cG_{1},\cG_{2})$ as listed in 
    Section~\ref{ssec:ZS of gpds}.
\end{remark}

\begin{example}\label{ex:connector:trivial}
    If $c_{1},c_{2}$ are the trivial $2$-cocycles, so constant $1$, then the constant map $\varphi(g,x)=1$ is always a \wordforphi.
\end{example}

\begin{example}\label{ex:connector:rotation 2-cocycles}
    Let $\cG_{1}=\cG_{2}=\mathbb{Z}$ act trivially on each other, and fix $\theta\in\mathbb{R}$.  The map $\varphi\colon \cG_{2}\times  \cG_{1} \to\mathbb{T}$ given by $(m,n)\mapsto \mathsf{e}^{2\pi i \theta mn}$ is a \wordforphi\ for the trivial $2$-cocycles on $\cG_{1}$ and $\cG_{2}$.
\end{example}

We now prove that matched twists from $2$-cocycles come precisely from \wordforphi s.
\begin{proposition}[External \ZS\ twist for $2$-cocycles]\label{prop:EZS twist:2-cocycles}
 Suppose $(\cG_{1}, \cG_{2})$ is a matched pair of groupoids with continuous $2$-cocycles $c_{i}$. Then $(\Sigma_{c_{1}}, \Sigma_{c_{2}})$ is a pair of matched twists (Definition~\ref{def:MT}) if and only if there exists a \wordforphi\ for $(c_{1},c_{2})$ (Definition~\ref{def:connector}). Indeed, the \wordforPhi\ $\Phi\colon \Sigma_{c_{2}}\ast_{\mathbb{T}}\Sigma_{c_{1}}\to \Sigma_{c_{1}}\ast_{\mathbb{T}}\Sigma_{c_{2}}$
 that covers $\cG_{1}\bowtie\cG_{2}$
 and the \wordforphi\ $\varphi\colon \cG_{2}\bfpsr \cG_{1}\to\mathbb{T}$ are related by the formula
 \begin{equation}     
 \label{eq:Phi and phi}
 \Phi\bigl((1,g)\asttwoone (1,x)\bigr) = (1,g\HleftG x) \astonetwo (\varphi(g,x),g\HrightG x)
 \end{equation}
 for all $(g,x)\in \cG_{2}\bfpsr \cG_{1}$. In this case, the map
 \[
    c\colon (\cG_{1}\bowtie\cG_{2})^{(2)}\to \mathbb{T},
    \;
    c\bigl((x_{1},g_{1}),(x_{2},g_{2})\bigr)=
    c_{1}(x_{1},g_{1}\HleftG x_{2})
    \
    \varphi(g_{1},x_{2})
    \
    c_{2}(g_{1}\HrightG x_{2}, g_{2}),
 \]
 is a $2$-cocycle on $\cG_{1}\bowtie\cG_{2}$, and the twist $\Sigma_c$ it induces is 
 canonically isomorphic
 to the
 external \ZS\ twist $\Sigma_{c_{2}}\ZST \Sigma_{c_{1}}$ from Theorem~\ref{thm:external ZS-product}.
\end{proposition}

\begin{remark}\label{rmk:Phi never unique}
Example~\ref{ex:connector:trivial} shows that \wordforphi s are very much not unique,
and that different choices of \wordforphi s can lead to non-isomorphic twists.
Consequently, Proposition~\ref{prop:EZS twist:2-cocycles} shows that there is likewise no uniqueness of twist \wordforPhi s, even if the twists $\Sigma_{1},\Sigma_{2}$ are trivial and even if the two-way actions of $\cG_{1}$ and $\cG_{2}$ on one another are fixed.
\end{remark}

\begin{proof}
    Write $\Sigma_{i}=\Sigma_{c_{i}}$, and suppose first that 
    {$\Phi\colon \Sigma_{2}\ast_{\mathbb{T}}\Sigma_{1}\to \Sigma_{1}\ast_{\mathbb{T}}\Sigma_{2}$ is a \wordforPhi\ for the pair $(\Sigma_{1},\Sigma_{2})$.}
    For any fixed $(g,x)\in\cG_{2}\bfpsr \cG_{1}$,
    {it follows from Equation~\eqref{eq:Pi and Phi} in Remark~\ref{rmk:Psi, Pi, and Phi} (see also Example~\ref{ex:MT in the trivial case}) that }
    there exist $z_{1},z_{1}\in\mathbb{T}$ 
    such that 
    \begin{equation}\label{eq:from Phi to varphi}
     \Phi\bigl((1,g)\asttwoone (1,x)\bigr) = (z_{1},
     g\HleftG x
     )\astonetwo (z_{2},
     g\HrightG x
     ).
    \end{equation}
    Since $(z_{1},g\HleftG x)=z_{1}\cdot(1,g\HleftG x)$, the balancing in $\Sigma_{1}\ast_{\mathbb{T}}\Sigma_{2}$ allows us to 
    turn Equation~\eqref{eq:Phi and phi} into Equation~\eqref{eq:from Phi to varphi} by defining
    $\varphi(g,x)\coloneq  z_{1}z_{2}$.
    Note that $\varphi$ is well defined: the elements $(1,x)$ and $(1,g\HleftG x)$ of $\cG_{1}$ have no other representatives with $1$ in the first component, so that $g,x$ and the product $z_{1}z_{2}$ are uniquely determined.

    Observe that, because of Condition~\ref{it:MT:T equivariant}, we can generalize Equation~\eqref{eq:Phi and phi} to
    \begin{equation}     
     \label{eq:Phi and phi, generalized}
     \Phi\bigl((z,g)\asttwoone (w,x)\bigr) = (z,g\HleftG x) \astonetwo (\varphi(g,x)w,g\HrightG x)
    \end{equation}

    To see that $\varphi$ is continuous, let
    \[
    u\coloneq 
    r_{\Sigma_{1}\ZST \Sigma_{2}}
    \bigl((1,g\HleftG x) \astonetwo (1,g\HrightG x)
    \bigr)
    .
    \]
    Using that $\Sigma_{1}\ZST \Sigma_{2}$ is a twist over $\cG_{1}\bowtie\cG_{2}$, we
    rewrite Equation~\eqref{eq:from Phi to varphi} to
    \begin{align*}
        \jmath(\varphi(g,x) ,u)
        =
        \Phi\bigl((1,g)\asttwoone (1,x)\bigr)
        \bigl[(1,g\HleftG x) \astonetwo (1,g\HrightG x)\bigr]\inv .
    \end{align*}
    In other words,
    \begin{align}
        \varphi(g,x)
        =
        (\mathrm{pr}_{\mathbb{T}}\circ \jmath\inv)
        \bigl(\Phi\bigl((1,g)\asttwoone (1,x)\bigr)
        \bigl[(1,g\HleftG x) \astonetwo (1,g\HrightG x)\bigr]\inv \bigr).
    \end{align}
    This shows that $\varphi$ is built from an array of continuous maps (multiplication and inversion of $\mathmbox{\Sigma_{1}\ZST \Sigma_{2}}$ are continuous and $\jmath$ is a homeomorphism onto its image; Proposition~\ref{prop.external.twist}), so that $\varphi$ is itself continuous.


    Now take $g\in (\cG_{2})^u_v$. Then
    \begin{align}
        \varphi(u,g)=1
        &\iff
        \Phi\bigl((1,g)\asttwoone (1,u)\bigr) = (1,g\HleftG u)\astonetwo (1,g\HrightG u).
        \notag
        \intertext{With \ref{item:ZS10} and \ref{item:ZS6}, we rewrite the right-hand condition to arrive at:}
        \varphi(u,g)=1
        &\iff
        \Phi\bigl((1,g)\asttwoone (1,u)\bigr) = (1,v)\astonetwo (1,g).
        \label{eq:normalized iff}
    \end{align}
    The right-hand equality is indeed satisfied by Condition~\ref{it:MT:diagram,top} for $i=2$, so the first half of \ref{it:connector:normalized} holds. The second half follows 
    analogously
    from Condition~\ref{it:MT:diagram,top} for $i=1$.

    To see that \ref{it:connector:cocycle cond,c1} holds, fix $g\in \cG_{2}$ and $(x,y)\in \cG_{1}^{(2)}$ such that $s(g)=r(x)$; if we let $z\coloneq  c_{1}(g\HleftG x,[g\HrightG x]\HleftG y)$ and  $w\coloneq  c_{1}(x,y)$, then we must verify that $\varphi(g,xy)w
    =
    \varphi(g,x)\varphi(g\HrightG x, y)z$.
    By Equation~\eqref{eq:Phi and phi, generalized}, we have
    \[
        \Phi\bigl((1,g)\asttwoone (1,x)\bigr)
        =
        (1,g\HleftG x) \astonetwo (\varphi(g,x),g\HrightG x),
    \]
    so it follows from Condition~\ref{it:MT:associativity, right} that 
    \begin{align}\label{eq:a long equation}
        \Phi\bigl((1,g)\asttwoone (1,x)(1,y)\bigr)
        &=
        (1,g\HleftG x)\bullet \Phi\bigl((\varphi(g,x),g\HrightG x) \asttwoone (1,y)\bigr).
    \end{align}
    We will compute both sides of this equation. For the left-hand side, we use
    in the first step that $w\in\mathbb{T}$ is such that
    $(1,x)(1,y)
        =
        (w,xy)$ 
    and 
    we use
    \ref{item:ZS8} in the last step:
    \begin{align}
    \Phi\bigl((1,g)\asttwoone (1,x)(1,y)\bigr) 
    &=\Phi\bigl((1,g)\asttwoone (w,xy)\bigr) 
    \notag
    \\
    &
    = (1,g\HleftG xy)\astonetwo (\varphi(g,xy)w,g\HrightG xy)
    &&
    \text{by \eqref{eq:Phi and phi}}
    \notag
    \\
    &= \bigl(1,(g\HleftG x)([g\HrightG x]\HleftG y)\bigr)\astonetwo \bigl(\varphi(g,xy)w,g\HrightG xy\bigr)
    \label{eq:for associativity}
    \end{align}   
    For right-hand side of Equation~\eqref{eq:a long equation}, we first write
    \begin{align}
        \Phi\bigl((1,g\HrightG x) \asttwoone (1,y)\bigr)
        &\overset{\eqref{eq:Phi and phi, generalized}}{=}
        \bigl(1,[g\HrightG x]\HleftG y\bigr)
        \astonetwo 
        \bigl(\varphi(g\HrightG x,y),[g\HrightG x]\HrightG y\bigr)
        \\
        &\overset{\ref{item:ZS4}}{=}
        \bigl(1,[g\HrightG x]\HleftG y\bigr)
        \astonetwo 
        \bigl(\varphi(g\HrightG x,y),g\HrightG xy\bigr),
    \intertext{%
    so that}
        \Phi\bigl((\varphi(g,x),g\HrightG x) \asttwoone (1,y)\bigr)
        &
    =
        \bigl(\varphi(g,x),[g\HrightG x]\HleftG y\bigr)
        \astonetwo 
        \bigl(\varphi(g\HrightG x,y),g\HrightG xy\bigr).
        \label{eq:RHS of long eq}
    \end{align}
    The definition of
    the multiplication in $\Sigma_{2}=\Sigma_{c_{2}}$ and the definition of $z$ tell us that
    \[
        \bigl(1,g\HleftG x\bigr)\bigl(
        \varphi(g,x)
        ,[g\HrightG x]\HleftG y\bigr)
        =
        \bigl(z
        \varphi(g,x)
        ,(g\HleftG x)([g\HrightG x]\HleftG y)\bigr).
    \]
    Combining this with Equation~\eqref{eq:RHS of long eq},
    the right-hand side of Equation~\eqref{eq:a long equation} becomes
    \begin{align*}
        &\bigl(z \varphi(g,x),(g\HleftG x)([g\HrightG x]\HleftG y)\bigr)
        \astonetwo 
        \bigl(\varphi(g\HrightG x,y),g\HrightG xy\bigr)
        \\
        &
        =
        \bigl(1,(g\HleftG x)([g\HrightG x]\HleftG y)\bigr)
        \astonetwo 
        \bigl(z\varphi(g,x)\varphi(g\HrightG x,y),g\HrightG xy\bigr)
        .
    \end{align*}
    Comparing this to the left-hand side of Equation~\eqref{eq:a long equation}, which we have computed in line~\eqref{eq:for associativity}, 
    %
    we conclude that we must have
    \[
        \varphi(g,xy)w=z\varphi(g,x)\varphi(g\HrightG x,y),
    \]
    which is exactly Condition~\ref{it:connector:cocycle cond,c1}. One shows Condition~\ref{it:connector:cocycle cond,c2} analogously, invoking Condition~\ref{it:MT:associativity, left}
    in place of ~\ref{it:MT:associativity, right}.
    
    Conversely, assume that we have a \wordforphi\ $\varphi$. We define $\Phi$ by the formula at~\eqref{eq:Phi and phi, generalized}; we get Condition~\ref{it:MT:T equivariant} 
    ($\mathbb{T}$-equivariance)
    for free. 
    The map
    \begin{align*}
        \cG_{2}\bfpsr \cG_{1}
        &\to \cG_{1}\bfpsr \cG_{2}        
        ,&
        (g,x) &\mapsto
        (
            g\HleftG x, g\HrightG x
        ),
    \intertext{has inverse given by}
        f\colon
        \cG_{1}\bfpsr \cG_{2}
        &\to \cG_{2}\bfpsr \cG_{1}        
        ,&
        (y,h) &\mapsto
        \bigl(
            [h\inv \HrightG y\inv]\inv, 
            [h\inv \HleftG y\inv]\inv
        \bigr).
    \end{align*}
    If we let $\mathrm{pr}_{\cG_{j}}$ denote the projection map $\cG_{1}\bowtie\cG_{2}\to\cG_{j}$, then
    it is easy to check that
     \begin{align}   
    \Phi\inv \bigl((z,y) \astonetwo (w,h)\bigr)
     =
    \bigl(z, \mathrm{pr}_{\cG_{2}}(f(y,h))\bigr)
    \asttwoone
    \bigl(\overline{\varphi(f(y,h))}w,  \mathrm{pr}_{\cG_{1}}(f(y,h))\bigr).
    \end{align}
    As a concatenation of continuous maps, we see that $\Phi$ and $\Phi\inv$ are continuous, so $\Phi$ is a homeomorphism.

    Condition~\ref{it:MT:source and range} follows directly from the definition of $\Phi$ from $\varphi$ (Equation~\eqref{eq:Phi and phi, generalized}) and from the \ZS\ product conditions on $\cG_{1}\bowtie\cG_{2}$:
    \begin{align*}
        r_{\Sigma_2}(z,g)=r_{\cG_{2}}(g)
        \overset{\ref{item:ZS2}}{=}
        r_{\cG_{1}}(g\HleftG x) = r_{\Sigma_{1}} (z,
     g\HleftG x)
    \end{align*}
    and 
    \begin{align*}
        s_{\Sigma_{1}} (w,x)
        =
        s_{\cG_{1}}(x)
        \overset{\ref{item:ZS5}}{=}
        s_{\cG_{2}}(g\HrightG x)
        =
        s_{\Sigma_2} \bigl(\varphi(g,x)w,g\HrightG x\bigr).
    \end{align*}    

    For \ref{it:MT:diagram,top}, we compute
    \begin{align*}
        (\Phi\circ \iota^{1}_{2,1}) (w,x)
        &=
        \Phi
        \bigl((1,r(x))\asttwoone (w,x)\bigr)
        \\
        &=
        (1,r(x)\HleftG x) \astonetwo (\varphi(r(x),x)w,r(x)\HrightG x)
        \\
        &\overset{(\dagger)}{=}
        (1,x) \astonetwo (w,s(x))
        =
        (w,x) \astonetwo (1,s(x))
        =
        \iota^{1}_{1,2} (w,x)
    \intertext{where $(\dagger)$ follows from  \ref{item:ZS3}, \ref{it:connector:normalized},\ref{item:ZS11}. Likewise}
        (\Phi\circ \iota^{2}_{2,1}) (z,g)
        &=
        \iota^{2}_{1,2} (z,g).
    \end{align*}
    
    For~\ref{it:MT:associativity, right},
    we must prove that Equation~\eqref{eq:a long equation} holds. We
    amend the computation we have done in Equation~\eqref{eq:for associativity}, where 
    as before
    $w\coloneq  c_{1}(x,y)$ and $z\coloneq  c_{1}(g\HleftG x,[g\HrightG x]\HleftG y)$:
    \begin{align*}
        \Phi\bigl((1,g)\asttwoone (1,x)(1,y)\bigr) 
        &
        \overset{\eqref{eq:for associativity}}{=} \bigl(1,(g\HleftG x)([g\HrightG x]\HleftG y)\bigr)\astonetwo \bigl(\varphi(g,xy)w,g\HrightG xy\bigr)
        \\
        &= (1,g\HleftG x)\bullet \bigl(\overline{z},[g\HrightG x]\HleftG y\bigr)\astonetwo \bigl(\varphi(g,xy)w,g\HrightG xy\bigr)
    \end{align*}  
    To conclude Condition~\ref{it:MT:associativity, right}, it now suffices to prove
    \[
    \Phi\bigl(
       (\varphi(g,x), g\HrightG x) \asttwoone (1,y)
    \bigr)
    =
    \bigl(\overline{z},[g\HrightG x]\HleftG y\bigr)\astonetwo \bigl(\varphi(g,xy)w,g\HrightG xy\bigr)
    .
    \]
    Since $\varphi(g,x)\varphi(g\HrightG x,y)=\varphi(g,xy)w\overline{z}$ by \ref{it:connector:cocycle cond,c1} and since $\Phi$ is $\mathbb{T}$-equivariant, 
    this is equivalent to showing
    \[
    \Phi\bigl(
       (1, g\HrightG x) \asttwoone (1,y)
    \bigr)
    =
    \bigl(1,[g\HrightG x]\HleftG y\bigr)\astonetwo \bigl(\varphi(g\HrightG x,y),g\HrightG xy\bigr).
    \]
    But this is exactly the definition of $\Phi$ at~\eqref{eq:Phi and phi, generalized} once one realizes that $g\HrightG xy=(g\HrightG x)\HrightG y$ by \ref{item:ZS4}. Condition~\ref{it:MT:associativity, left} is shown analogously, using \ref{it:connector:cocycle cond,c2}. This concludes our proof that $\Phi$ is a \wordforPhi.

    \medskip

    Since the continuous map
    \[
        \mathfrak{s}\colon \cG_{1}\bowtie\cG_{2}\to \Sigma_{1}\ZST \Sigma_{2},
        \quad
        (x,g)\mapsto (1,x)\astonetwo (1,g),
    \]
    satisfies $(\Pi_{1,2}\circ\mathfrak{s})(x,g)=(x,g)$, it is a continuous section of the twist $\Sigma_{1}\ZST \Sigma_{2}$, so  this means that the function $c'\colon (\cG_{1}\bowtie\cG_{2})^{(2)}\to \mathbb{T}$ given
for $(\xi_{1},\xi_{2})\in (\cG_{1}\bowtie\cG_{2})^{(2)}$ by
\[
c\bigl(\xi_{1},\xi_{2}\bigr) =
\mathfrak{s}(\xi_{1})
\
\mathfrak{s}(\xi_{2})
\
\mathfrak{s}(\xi_{2}\xi_{2})\inv
\]
is a $2$-cocycle on $\cG_{1}\bowtie\cG_{2}$ and that $\Sigma_{1}\ZST \Sigma_{2}\cong \Sigma_{c}$ (cf.\ \cite[Fact 4.1]{Kum:Diags}, 
    \cite[Proposition
I.1.14]{Renault:gpd-approach}, \cite[Remark 11.1.6]{Sims:gpds}). 
We compute for $\xi_{i}=(x_{i},g_{i})$:
\begin{align*}
    \mathfrak{s}(\xi_{1})
    \
    \mathfrak{s}(\xi_{2})
    &=
    (1,x_{1})\bullet 
    \Phi\bigl( (1,g_{1}) \asttwoone (1,x_{2})\bigr)
    \bullet (1,g_{2})
    \\
    &=
    (1,x_{1})\bullet 
    (1,g_{1}\HleftG x_{2}) \astonetwo (\varphi(g_{1},x_{2}),g_{1}\HrightG x_{2})
    \bullet (1,g_{2})
&& \text{by \eqref{eq:Phi and phi}}
    \\
    &=
    c_{1}(x_{1},g_{1}\HleftG x_{2})
    \
    \varphi(g_{1},x_{2})
    \
    c_{2}(g_{1}\HrightG x_{2}, g_{2})
    \cdot
    \\
    &\mathbin{\phantom{=}}
    (1,x_{1}[g_{1}\HleftG x_{2}]) \astonetwo (1,[g_{1}\HrightG x_{2}]g_{2})
    \\
    &=
    c_{1}(x_{1},g_{1}\HleftG x_{2})
    \
    \varphi(g_{1},x_{2})
    \
    c_{2}(g_{1}\HrightG x_{2}, g_{2})\cdot \mathfrak{s}(\xi_{1}\xi_{2}),
\end{align*}
which proves that $c$ has the claimed form.
\end{proof}

\begin{example}\label{ex:connector:rotation 2-cocycles,ctd}
Using Example~\ref{ex:connector:rotation 2-cocycles}, we recover a very famous twisted group:
    By Proposition~\ref{prop:EZS twist:2-cocycles}, it follows that
    \[
        c\colon \mathbb{Z}^2 \times \mathbb{Z}^2 \to \mathbb{T},
        \quad
        c\bigl(\vec{m},\vec{n}\bigr)=\mathsf{e}^{2\pi i \theta m_{2}n_{1}}
    \]
    is a $2$-cocycle on $\mathbb{Z}^2$. It is well known that the twisted group $C^*$-algebra associated to this $2$-cocycle is the rotation algebra, $A_{\theta}$.
\end{example}

In the case of an internal \ZS\ twist, the \wordforphi\ can be recovered from the $2$-cocycle directly. 

\begin{proposition}[Internal \ZS\ twist for $2$-cocycles]\label{prop:IZS twist:2-cocycles}
Suppose $c$ is a $2$-cocycle on a groupoid $\cG$ such that $(\Sigma_{c},\cG)$ has an internal \ZS\ product structure  $(\Sigma_{1}, \Sigma_{2})$ as in Definition~\ref{def:IZS}; let $\cG_{i}\coloneq \mathrm{pr}_{2}(\Sigma_{i})\leq\cG$. Then
\(
c_{i}\coloneq  c|_{\cG_{i}}
\)
is a $2$-cocycle on $\cG_{i}$, and the 
map $\varphi\colon \cG_{2}\bfpsr \cG_{1}\to \mathbb{T}$ given by
\[\varphi(g,x)=c(g,x)\overline{c(g\HleftG x, g\HrightG x)}\]
is a \wordforphi\ for the pair $(c_{1},c_{2})$.
\end{proposition}

If we want to think of $\cG_{1}\bowtie\cG_{2}$ as an external \ZS\ product, then
the above in-line equation 
should be
written as
\[
\varphi(g,x) = c\bigl( (r(g),g),(x, s(x))\bigr)\overline{c\bigl((g\HleftG x,s(g)), (r(x),g\HrightG x)\bigr)}
.
\]

\begin{proof}
    The fact that $c_{i}$ is a $2$-cocycle on $\cG_{i}$ is 
    obvious.
    The map $\varphi$ is continuous as concatenation of continuous maps; it remains to show the algebraic conditions in Definition~\ref{def:connector}.

    \ref{it:connector:normalized} Follows immediately since $c$ is normalized and since $r(x)\HrightG x = s(x)$
    by~\ref{item:ZS11} and $g\HleftG s(g)=r(g)$
    by~\ref{item:ZS10}.
    
    \ref{it:connector:cocycle cond,c1} Fix $g\in \cG_{2}$ and $(x,y)\in \cG_{1}^{(2)}$ such that $s(g)=r(x)$. We compute
    \begin{align*}
        c_{1}(x,y)\varphi(g,xy)
        &=
        c(x,y)c(g,xy)\overline{c(g\HleftG [xy], g\HrightG [xy])} 
        \intertext{which, by the cocycle identity~\eqref{eq:cocycle identity} (applied to the first two factors) and by \ref{item:ZS8} (applied to the third factor) equals}
        &=
        [c(g,x)c(gx,y)]
        \overline{c\bigl((g\HleftG x)([g\HrightG x]\HleftG y), g\HrightG [xy]\bigr)}.
    \end{align*}
    We compute the last factor, using the cocycle identity yet again:
    \begin{align*}
       &c\bigl((g\HleftG x)([g\HrightG x]\HleftG y), g\HrightG [xy]\bigr)
       \\
     &=
        c\bigl(g\HleftG x, ([g\HrightG x]\HleftG y)(g\HrightG [xy])\bigr)
        c\bigl([g\HrightG x]\HleftG y, g\HrightG [xy]\bigr)
        \overline{
            c\bigl(g\HleftG x,[g\HrightG x]\HleftG y\bigr)
        }
    \intertext{which, since $g\HrightG [xy]=(g\HrightG x)\HrightG y$ and thus $([g\HrightG x]\HleftG y)(g\HrightG [xy])=[g\HrightG x]y$ by Lemma~\ref{lem:2 unique decomp in ZS} equals}
     &=
        c\bigl(g\HleftG x, [g\HrightG x]y\bigr)
        c\bigl([g\HrightG x]\HleftG y, g\HrightG [xy]\bigr)
        \overline{
            c\bigl(g\HleftG x,[g\HrightG x]\HleftG y\bigr)
        }.
    \end{align*}
    Again, we take the very first factor and apply the cocycle identity:
    \begin{align*}
        c\bigl(g\HleftG x, [g\HrightG x]y\bigr)
        &=
        c\bigl((g\HleftG x)[g\HrightG x], y\bigr)
        c\bigl(g\HleftG x, g\HrightG x\bigr)
        \overline{c(g\HrightG x,y)}
        \\
        &=
        c\bigl(gx, y\bigr)
        c\bigl(g\HleftG x, g\HrightG x\bigr)
        \overline{c(g\HrightG x,y)}.
    \end{align*}
    Combining all of these computations, we arrive at:
    \begin{align*}
        c_{1}(x,y)\varphi(g,xy)
        &=
        c(g,x)c(gx,y)
        \overline{
        c\bigl(gx, y\bigr)
        c\bigl(g\HleftG x, g\HrightG x\bigr)
        }
        \\&\phantom{=}
        c(g\HrightG x,y)
        \overline{
        c\bigl([g\HrightG x]\HleftG y, g\HrightG [xy]\bigr)
        }
            c\bigl(g\HleftG x,[g\HrightG x]\HleftG y\bigr).
        \intertext{The first row collapses to $\varphi(g,x)$, so that}
        c_{1}(x,y)\varphi(g,xy)
        &=
        \varphi(g,x)
        c(g\HrightG x,y)
        \overline{
        c\bigl([g\HrightG x]\HleftG y, g\HrightG [xy]\bigr)
        }
            c\bigl(g\HleftG x,[g\HrightG x]\HleftG y\bigr)
        \\&=
        \varphi(g,x)\varphi(g\HrightG x, y)c_{1}(g\HleftG x, (g\HrightG x)\HleftG y),
    \end{align*}
    as needed. One proves \ref{it:connector:cocycle cond,c2} analogously.
\end{proof}

\section{Cartan pairs and C*-blends}\label{sec:Cartan}

Twisted groupoid algebras arise naturally from C*-algebras with Cartan subalgebras due to the work of Renault \cite{Renault:Cartan}. There is an abundance of C*-algebras with Cartan subalgebras thanks to the recent work of Xin Li \cite{XLi2020_Cartan}, where he proved that every simple classifiable C*-algebra has a Cartan subalgebra. In this section, we will apply our results to twisted groupoids arising from Cartan pairs. 

One of our original motivations 
for this paper
was to explore an internal \ZS\ product structure of C*-algebras, similar to the internal \ZS\ product of groups. 
The reader should recall the two ingredients 
\hyperlink{item:intro:K product}{(1)} and \hyperlink{item:intro:G cap H}{(2)} from the introduction for a group $K$ to be the internal \ZS\ product of two subgroups $G$, $H$.
Given a C*-algebra $A$ with two subalgebras $A_{1}, A_{2}$, 
a good
C*-analogue for $A$ being ``a product of $A_{1}$ and $A_{2}$'' is precisely the notion of a 
\emph{C*-blend} due to Exel \cite[Definition 2.1]{Exel:Blends}. 
Viewed this way, \cite[Theorem 13]{BPRRW:ZS} is not surprising:  the full groupoid C*-algebra $C^*(\cG_{1}\bowtie\cG_{2})$ of the \ZS\ product of a matched pair $(\cG_{1},\cG_{2})$ is a C*-blend of the individual full C*-algebras $C^*(\cG_{1})$ and $C^*(\cG_{2})$. 
In our first result of this section, Theorem~\ref{thm:MT give C*-blends}, we generalize this result to twisted groupoid C*-algebras, both full and reduced.

In the last part of this section, we are then concerned with a partial converse:  Given a C*-blend, what properties are needed to realize the large algebra $A$ as ``a \ZS\ product'' of its two subalgebras $A_{1}, A_{2}$? At least if we want the algebras in question to be (reduced twisted) groupoid C*-algebras, then in view of Kumjian--Renault theory,  a natural replacement of Condition~\hyperlink{item:intro:G cap H}{(2)} in this setting is
that $D=A_{1}\cap A_{2}$ be a Cartan subalgebra in
all of $A_{1},A_{2},A$. Consequently, in Theorem~\ref{thm:from C*-blend}, we prove that the twisted Weyl groupoid associated with the Cartan pair $(A,D)$ is exactly the \ZS\ product of the twisted Weyl groupoids associated with $(A_{1},D)$ and $(A_{2},D)$.
We shall start with a brief review of twisted groupoid C*-algebras and C*-blend. 
We refer the reader to \cite{Wil2019} for a detailed discussion on groupoid C*-algebras, and to \cite{Exel:Blends} for more information regarding C*-blends.

\subsection{Twisted groupoid C*-algebras}

Suppose $\cG$ is a (\LCH) \etale\ groupoid, and let $\Sigma$ be a twist over $\cG$ as in Definition~\ref{def:twist}. Then we can give $\Sigma$ a canonical  Haar system 
induced by the system of counting measures on $\cG$  and the Lebesgue measure on $\mathbb{T}$.
On
\[C_c(\cG;\Sigma) \coloneq \{f\in C_c(\Sigma): f(z\cdot e)=zf(g), \text{ for all }z\in\mathbb{T}, e\in \Sigma\},\]
this allows us to define a $*$-algebra structure by the convolution formula
\[f*g(e')=\sum_{\pi(e):s(e')=s(e)} f(e' e\inv) g(e) \]
and the involution $f^*(e)=\overline{f(e\inv)}$. 

Aside from the supremum norm $\norm{\cdot}_\infty$, there are two natural choices of norms  on $C_c(\cG;\Sigma)$. For the \emph{reduced} norm $\norm{\cdot}_\red$, we first define for $u\in \cG\z$ the representation $\pi_u$ of $ C_c(\cG;\Sigma)$ on $L^2(\cG u;\Sigma u)$ by
\[
    \bigl(\pi_u (f)\xi \bigr) (e')=
        \sum_{\pi(e):s(e')=s(e)} f(e' e\inv) \xi(e)
    ,
\]
and we then let
\[
    \norm{f}_\red \coloneq \sup_{u\in\cG\z} \norm{\pi_u(f)}.
\]
Completing $C_c(\cG;\Sigma)$ in $\norm{\cdot}_\red$ yields the \emph{reduced twisted groupoid C*-algebra $C^*_\red (\cG;\Sigma)$}. The other norm we consider is the \emph{universal} or \emph{full} norm $\norm{\cdot}_\full$ on $C_c(\cG;\Sigma)$:
\[
\norm{f}_\full \coloneq \sup\left\{ \norm{\pi(f)}: \pi \text{ is a $*$-representation of }C_c(\cG;\Sigma)\right\}.
\]
The completion of $C_c(\cG;\Sigma)$ in $\norm{\cdot}_\full$ yields the \emph{full twisted groupoid C*-algebra $C^* (\cG;\Sigma)$}.

Many of our arguments in this section will happen on the level of the dense $*$-subalgebra $C_{c}(G;\Sigma)$. The following lemma gives a (well-known) description of it.
\begin{lemma}\label{lem:C_c(G;Sigma)=span}
    Suppose $\Sigma$ is a twist over a groupoid $\cG$ and suppose that $\mathfrak{U}$ is a base for the topology of $\cG$. Then
    \[
        C_c(\cG;\Sigma) 
        =
        \lspan
        \{
        f\in C_c(\cG;\Sigma)
        : 
        \pi(\suppo(f))\subseteq U \text{ for some }U\in \mathfrak{U}
        \}
        .
    \]
    In particular, 
    the span of such functions $f$ is dense both in the reduced $C_{\red}^*(\cG; \Sigma)$ 
    and in the full $C^*(\cG; \Sigma)$ twisted groupoid C*-algebra.
\end{lemma}
\begin{proof}
    Our proof is similar to that in \cite[Lemma 9.1.3]{Sims:gpds}, with changes accounting for the twist. For any $f\in C_c(\cG;\Sigma)$, we can cover the compact set $\pi(\supp(f))$ with finitely many elements $U_{1},\ldots, U_{n}$ of $\mathfrak{U}$. Let $\{h_{j}\}_{j=1}^{n}\subseteq C_c(\cG)$ be a partition of unity subordinate to $\{U_{j}\}_{j=1}^{n}$ and define $f^{j}(e)\coloneq  h_{j}(\pi(e)) f(e)$. As a pointwise product of continuous, compactly supported functions, $f^j$ is continuous and compactly supported. Since $\pi(\suppo(f'))\subseteq  \suppo(h_{j})\subseteq U_{j}\in \mathfrak{U}$, and since $f$ is $\mathbb{T}$-equivariant and $h_{j}\circ \pi$ is $\mathbb{T}$-invariant, we conclude $f^{j}\in C_c(\cG;\Sigma)$. Lastly, since we picked a partition of unity, we have
    $\sum_{j} f_{j}
        =
        f
    $, which finishes our claim.
\end{proof}

One reason why the above is so helpful is 
when it is used in conjunction with
the following, which greatly simplifies some computations:
\begin{lemma}[{see \cite[Theorem 11.1.11]{Sims:gpds}}]\label{lem:BPRRW's Lemma 15}
    Suppose $\Sigma$ is a twist over an \etale\  groupoid~$\cG$. If $f\in C_c(\cG;\Sigma)$ is such that $\pi(\suppo(f))$ is a bisection, then \[\|f\|_{\infty}=\|f\|_{\red}=\|f\|_{\full}
    .\]
\end{lemma}

There is one last tool that we will need later
 to be able to turn elements of $C_{c}(\Sigma)$ into elements of $C_{c}(\cG;\Sigma)$;
we thank Dana P.\ Williams for pointing us to this lemma which is readily checked.
\begin{lemma}[cf.\ {\cite[Lemme 
3.3]{Renault:1987:rep}}]\label{lem:making equiv fcts}
  Suppose $\Sigma$ is a twist over a groupoid $\cG$. For $f\in C_c(\Sigma)$, define $T(f)\colon \Sigma\to \mathbb{C}$ by 
 \[
    T(f)(e)= \int_\mathbb{T} \overline{z} f(z\cdot e)\,\mathrm{d}z.
 \]
 Then $T$ is a $\norm{\cdot}_{\infty}$-decreasing, linear, idempotent, surjective map $C_c(\Sigma)\to C_c(\cG;\Sigma)$.
\end{lemma}

\subsection{C*-blend}

We begin with some basic background on 
C*-blends
and prove that the 
``\ZS-twisted'' C*-algebra
of 
$\Sigma_{1}\ZST \Sigma_{2}\to \cG_{1}\bowtie\cG_{2}$
can naturally be realized as a C*-blend of the twisted C*-algebras of the individual
twists $\Sigma_{j}\to\cG_{j}$.

\begin{definition}[%
{%
    \cite[Definition 3.1]{Exel:Blends}%
}%
]
    A quintuple $(B_{1},\iota_{1};B_{2}, \iota_{2}; A)$
    is \emph{a C*-blend} if
    \begin{enumerate}[label=\textup{(B\arabic*)}]
        \item $A,B_{1},B_{2}$ are C*-algebras;
        \item each $\iota_{j}\colon B_{j}\to \mathcal{M}(A)$ is a *-homomorphism; and
        \item\label{it:C*B:dense range} the range of the map 
        \begin{align*}
            \iota_{1}\odot \iota_{2}\colon B_{1}\odot B_{2}&\to \mathcal{M}(A),
            &&
            b_{1}\odot b_{2}\mapsto \iota_{1}(b_{1})\iota_{2}(b_{2}),
        \end{align*} 
        is
        contained and dense in $A$.
    \end{enumerate}
    If 
    $\iota_{1}\odot\iota_{2}$ is injective, the C*-blend is called a \emph{C*-alloy}.  
    We call the C*-blend \emph{\adjforBlend} if both maps $\iota_{1},\iota_{2}$ are injective with image contained in $A$.
    When the maps $\iota_{1},\iota_{2}$ are 
   understood, we may drop them and write the quintuple 
    simply as a triple $(B_{1};B_{2};A)$.
\end{definition}

If  $(\cG_{1},\cG_{2})$ is a pair of matched \emph{$r$-discrete} groupoids, then it follows from Lemma~\ref{lem:clopen in ZS-product if r-discrete} that the subgroupoid $\cG_{i}$  of $\cG_{1}\bowtie\cG_{2}$ is not only closed but also open; in particular, we have a well-defined map $C_c(\cG_{i})\to C_c(\cG_{1}\bowtie\cG_{2})$ that extends functions by $0$ outside of $\cG_{i}$. In the case of \etale\ groupoids, this inclusion can be extended to the associated reduced groupoid C*-algebras, and indeed even for 
the twisted algebras:
\begin{lemma}\label{lem:MT gives incl}
   Suppose $(\cG_{1},\cG_{2})$ is a pair of matched \emph{\etale} groupoids. Suppose further that $(\Sigma_{1},\Sigma_{2})$ is a pair of matched twists
   with \wordforPhi\ $\Phi$
   that covers $\cG=\cG_{1}\bowtie\cG_{2}$; let $\Sigma=\Sigma_{1}\ZST \Sigma_{2}$ be the external \ZS\ twist (Theorem~\ref{thm:external ZS-product}).
   For $i=1,2$, the map $C_c(\cG_{i};\Sigma_{i})\to C_c(\cG;\Sigma)$ that extends functions by $0$ extends to 
   an injective $*$-homomorphism
   \[
   \iota^{\red}_{i}
   \colon C_{\red}^*(\cG_{i};\Sigma_{i})\to C_{\red}^*(\cG;\Sigma)
   \]
   of reduced twisted groupoid C*-algebras 
   and to a $*$-homomorphism
   \[
   \iota^{\full}_{i}\colon C^*(\cG_{i};\Sigma_{i})\to C^*(\cG;\Sigma)
   \]
   of their full counterparts. 
\end{lemma}

\begin{proof}
   Lemma~\ref{lem:clopen in ZS-product if r-discrete} implies that $\cG_{i}$ is an open subgroupoid of $\cG$ since $\cG_{j}$ is $r$-discrete. Furthermore, by Condition~\ref{it:IZS:saturated}, we know that $\Sigma_{i}=\pi\inv(\cG_{i})$. The claim 
   about reduced C*-algebras
   now follows from an application of \cite[Lemma~2.7]{BFPR:GammaCartan}. 
   For the full C*-algebras, the claim follows directly from the universal property of $C^*(\cG_{i};\Sigma_{i})$ with respect to $*$-representations of $C_c(\cG_{i};\Sigma_{i})$.
\end{proof}

\begin{theorem}\label{thm:MT give C*-blends}
   Suppose $(\cG_{1},\cG_{2})$ is a pair of matched \emph{\etale} groupoids and that $(\Sigma_{1},\Sigma_{2})$ is a pair of matched twists
   with \wordforPhi\ $\Phi$
   that covers $\cG=\cG_{1}\bowtie\cG_{2}$.
    With 
    $\iota^{\red}_{1},\iota^{\red}_{2}$ and $\iota^{\full}_{1},\iota^{\full}_{2}$
    the maps from Lemma~\ref{lem:MT gives incl}, the 
    quintuples
    \begin{align*}
    \Bigl(
        C^*_{\red}(\cG_{1};\Sigma_{1}),
        \iota^{\red}_{1};
        C^*_{\red}(\cG_{2};\Sigma_{2}),
        \iota^{\red}_{2};
        C^*_{\red}\bigl(\cG_{1}\bowtie\cG_{2}; \Sigma_{1}\ZST \Sigma_{2}\bigr)
    \Bigr)
    \intertext{and}
    \Bigl(
        C^*(\cG_{1};\Sigma_{1}),
        \iota^{\full}_{1};
        C^*(\cG_{2};\Sigma_{2}),
        \iota^{\full}_{2};
        C^*\bigl(\cG_{1}\bowtie\cG_{2}; \Sigma_{1}\ZST \Sigma_{2}\bigr)
    \Bigr)
    \end{align*}
    are C*-blends.
    Moreover, the blend of reduced C*-algebras is \adjforBlend.
\end{theorem}

\begin{remark}
    Because of Theorem~\ref{thm:IZS=MT}, we could have likewise said that any \emph{internal} \ZS\ structure $(\Sigma_{1},\Sigma_{2})$ of a twist $\Sigma$ gives rise to a C*-blend.
\end{remark}

\begin{remark}
    Applying Theorem~\ref{thm:MT give C*-blends} to the situation of trivial twists (see Remark~\ref{rmk:MT in the trivial case,p2}), we recover \cite[Theorem 13]{BPRRW:ZS}: the quintuple
    \[
     \Bigl(
        C^*(\cG_{1}),
        \iota^{\full}_{1};
        C^*(\cG_{2}),
        \iota^{\full}_{2};
        C^*(\cG_{1}\bowtie\cG_{2})
    \Bigr)
    \]
    is a C*-blend. 
\end{remark}

\begin{remark}
    If the unit space $\cU$ is not a point, then the C*-blend of Theorem~\ref{thm:MT give C*-blends} is not a C*-alloy. Indeed, take two distinct points $u_{1}, u_{2}$ in $\cU$. Use Urysohn to get functions $f_{1},f_{2}\in  C_0(\cU)$ with $f_{i} (u_{i})=1$ and $f_{i}(u_j)=0$ for $j\neq i$. Then $f_{1} f_{2}=f_{2}f_{1}$ in $C^*_{\red}(\cG;\Sigma)$ even though $f_{1}\odot f_{2} \neq f_{2} \odot f_{1}$ in $C^*_{\red}(\cG_{1};\Sigma_{1})\odot C^*_{\red}(\cG_{2};\Sigma_{2})$.
\end{remark}

\begin{example}\label{ex:A_theta as a C*-blend}
    One of the motivating examples of a C*-blend is the crossed product: if $(A,\alpha,\Gamma)$ is  a C*-dynamical system, i.e., $\alpha$ denotes 
    an action of a \LCH\ group
    $\Gamma$ on a C*-algebra $A$, then
    \begin{equation}\label{eq:ex:CP are blends}
    \bigl(
        A
        ;
        C^*_{\red}(\Gamma);
        A\rtimes_{\alpha,r} \Gamma
    \bigr)
    \quad\text{and}\quad
    \bigl(
        A
        ;
        C^*(\Gamma);
        A\rtimes_{\alpha} \Gamma
    \bigr)
    \end{equation}
    are C*-blends.
    See \cite[Proposition 3.4]{Exel:Blends} for more details. (Exel proves it for the full C*-algebras but points out that it likewise holds for the reduced.)
    
    We can, of course, see some subexamples of this appear in our Theorem~\ref{thm:MT give C*-blends}. For instance,  for a fixed $\theta\in\mathbb{R}$, define the $2$-cocycle $c_{\theta}$ on the group $\cG=\mathbb{Z}^2$ by 
    \[
        c_{\theta}\bigl((m_{1},m_{2}),(n_{1},n_{2})\bigr)=\mathsf{e}^{2\pi i \theta m_{2}n_{1}}.
    \]
    Using Examples~\ref{ex:connector:rotation 2-cocycles} and~\ref{ex:connector:rotation 2-cocycles,ctd}, Theorem~\ref{thm:MT give C*-blends} yields that
    \[
    \Bigl(
        C^*\bigl(\mathbb{Z}\times\{0\}\bigr)
        ;
        C^*\bigl(\{0\}\times\mathbb{Z}\bigr);
        C^*\bigl(\mathbb{Z}^2, c_{\theta}\bigr)
    \Bigr)
    \]
    is a  C*-blend.
    To put this into the framework of crossed products, we identify
    $C^*(\mathbb{Z}^2, c_{\theta})$
    with the crossed product $ C(\mathbb{T})\rtimes_\theta \mathbb{Z}$
    of $\mathbb{Z}$ acting on $C(\mathbb{T})$ by rotation by $\theta$: the *-isomorphism is determined by mapping the generators $\delta_{(1,0)}$ and $\delta_{(0,1)}$ of 
    $C^*_{\red}(\mathbb{Z}^2, c_{\theta})=C^*(\mathbb{Z}^2, c_{\theta})$
    to the generators $z\delta_{0}$  and $\operatorname{const}_{1}\delta_{1}$ of $C(\mathbb{T})\rtimes_\theta \mathbb{Z}$, respectively. The above C*-blend then turns into
    \[
    \bigl(
        C(\mathbb{T})
        ;
        C^*(\mathbb{Z});
        C(\mathbb{T})\rtimes_\theta \mathbb{Z}
    \bigr),
    \]
    which is an example of a C*-blend as in~\eqref{eq:ex:CP are blends}.
\end{example}

\begin{proof}[Proof of Theorem~\ref{thm:MT give C*-blends}]
    Our proof will follow the ideas in \cite[Theorem 13]{BPRRW:ZS}, making the necessary (and quite subtle) adjustments to account for the twist. 

    Write $\cG=\cG_{1}\bowtie\cG_{2}$ and $\Sigma=\Sigma_{1}\ZST \Sigma_{2}$. 
   By definition, the maps $\iota^\red_i$ and $\iota^\full_i$ coincide when restricted to the subalgebra of compactly supported functions: they extend functions on the subgroupoid $\Sigma_{i}$ by $0$ to functions on $\Sigma$. Let us therefore simply write 
   \[
   \iota_{i}\coloneq\iota^\red_i\vert_{C_c(\cG_i;\Sigma_i)}=\iota^\full_i\vert_{C_c(\cG_i;\Sigma_i)}\colon C_c(\cG_i;\Sigma_i) \to C_c(\cG;\Sigma),
   \]
   independently of whether we regard its codomain $C_c(\cG;\Sigma)$ as a subalgebra of $C^*_{\red}(\cG; \Sigma)$ or of $C^*(\cG; \Sigma)$.
    We have seen in Lemma~\ref{lem:MT gives incl} that the maps $\iota^\red_{1}, \iota^\red_{2}$ are injective with codomain $C^*_{\red}(\cG; \Sigma)$, and that the codomain of $\iota^\full_{1}, \iota^\full_{2}$ is $C^*(\cG; \Sigma)$. Thus,
    it suffices to show that the range of $\iota_{1}\odot\iota_{2}$ is dense in $C^*_{\red}(\cG; \Sigma)$  and in $C^*(\cG; \Sigma)$. 
   
   By construction, the element $f_{1}\circledast f_{2} \coloneq  (\iota_{1}\odot\iota_{2})(f_{1}\odot f_{2})$ for $f_{i}\in C_c (\cG_{i};\Sigma_{i})$ is given by
   the element
    \[
        f_{1}\circledast f_{2}\colon e_{1}e_{2}\mapsto f_{1}(e_{1})f_{2}(e_{2})
    \]
    of $C_c(\cG;\Sigma)$.
    We will bootstrap this to more general functions than equivariant ones:

    For $n\in\mathbb{N}^\times$, let $C_c(\cG_{i};\Sigma_{i};n)$ be the subspace of $C_c(\Sigma_{i})$ consisting of elements $f_{i}$ for which
    \(
         \mathmbox{f_{i}(z\cdot e)}
         =
         \mathmbox{z^n f_{i}(e)}
    \) for all $z\in\mathbb{T}$ and all $e\in\Sigma_{i}$. Given $f_{i}\in C_c(\cG_{i};\Sigma_{i};n)$ (for $i=1,2$ and the same $n$), consider the continuous and compactly supported function $\Sigma_{1}\bfpsr \Sigma_{2}\to \mathbb{C}$ defined by $(e_{1},e_{2})\mapsto f_{1}(e_{1})f_{2}(e_{2})$. Since
    \[
        f_{1}(z\cdot e_{1})f_{2}(\overline{z}\cdot e_{2})
        =
        z^n f_{1}(e_{1})\ \overline{z}^n f_{2}(e_{2})
        =
        f_{1}(e_{1})f_{2}(e_{2}),
    \]
    the function factors through a map $\Sigma_{1}\ast_{\mathbb{T}} \Sigma_{2}\to \mathbb{C}$; it is continuous and compactly supported since the quotient map $\Sigma_{1}\bfpsr\Sigma_{2} \to \Sigma_{1}\ast_{\mathbb{T}} \Sigma_{2}$ is open and continuous. Since the space $\Sigma_{1}\ast_{\mathbb{T}} \Sigma_{2}$ is homeomorphic to $\Sigma$ by Theorem~\ref{thm:IZS=MT}\ref{it:thm:from MT to IZS}, we conclude that the map $f_{1}\circledast f_{2}\colon \Sigma\to\mathbb{T}$ given by $e_{1}e_{2}\mapsto  f_{1}(e_{1})f_{2}(e_{2})$ is an element of $C_c(\Sigma)$.
    We claim that these elements form a uniformly dense spanning set of 
    the commutative C*-algebra
    $C_0(\Sigma)$.
    
    Pick any $e\in \Sigma$. By Theorem~\ref{thm:IZS=MT}\ref{it:thm:from MT to IZS} and Condition~\ref{it:IZS:product}, we know that $e=e_{1}e_{2}$ for some $e_{i}\in\Sigma_{i}$. Thanks to Lemma~\ref{lem:making equiv fcts}, we may pick $f_{i}\in C_c(\cG_{i};\Sigma_{i})$ with $f_{i}(e_{i})\neq 0$. Then $f_{1}\circledast f_{2}$ satisfies $(f_{1}\circledast f_{2})(e)=f_{1}(e_{1})f_{2}(e_{2})\neq 0$, meaning that these functions separate the points of $\Sigma$. By Stone--Weierstrass, they thus 
    generate a uniformly dense $*$-subalgebra of $C_0(\Sigma)$.
    For $f_{i},f_{i}'\in C_c(\cG_{i};\Sigma_{i})$, we have
    \[
        (f_{1}\circledast f_{2}) \cdot (f_{1}'\circledast f_{2}')
        =
        (f_{1}\cdot f_{1}')\circledast (f_{2}\cdot f_{2}').
    \]
    As $f_{i}\cdot f_{i}'\in C_c(\cG_{i};\Sigma_{i};2)$, we conclude by induction that the set
    \begin{align*}
        X\coloneq  
        \bigcup_{n\in\mathbb{N}^\times}
        \left\{
        f_{1}\circledast f_{2}
        :
        f_{i}\in C_c(\cG_{i};\Sigma_{i};n)
        \right\}
    \end{align*}
    is closed not just under the involution but also under the (pointwise) multiplication of $C_0(\Sigma)$, so its span is a uniformly dense $*$-subalgebra of $C_0(\Sigma)$.
    To use $\lspan(X)$ to approximate elements of $C^*_{\red}(\cG;\Sigma)$ in the reduced norm
    and elements of $C^*(\cG;\Sigma)$ in the full norm,
    we next must argue that it suffices to consider elements supported in preimages of bisections.
    
Since $\cG_{i}$ is \LCH\ and \etale, its topology has a basis of precompact open bisections. In particular,
 \begin{equation}\label{eq:base for ZS}
 \mathfrak{U}=\{
 U_{1}U_{2}: \overline{U_{i}}\subseteq V_{i}\subseteq \cG_{i} \text{ and } U_{i}, V_{i} \text{ are precompact open bisections}\}
 \end{equation}
 is a base for the topology of $\cG$.  Thus, by Lemma~\ref{lem:C_c(G;Sigma)=span},
   \[
        \lspan
        \left\{
        f\in C_c(\cG;\Sigma)
        : 
        \pi(\suppo(f))\subseteq U 
        \text{ for some }U\in\mathfrak{U}
        \right\}
    \]
    is dense in $C_{\red}^*(\cG; \Sigma)$
    and in $C^*(\cG; \Sigma)$.
    To prove that  $\ran (\iota_{1}\odot\iota_{2})$ is dense in $C_{\red}^*(\cG; \Sigma)$
    and in $C^*(\cG; \Sigma)$,
    it therefore suffices to approximate such dense spanning elements in the 
    reduced,
    respectively full,
    norm.

    So suppose we are given $f\in C_c(\cG;\Sigma)$ with $\pi(\suppo(f))\subseteq U_{1} U_{2}$, where $U_{i}$ is an open bisection of $\cG_{i}$ whose closure is contained in another open bisection~$V_{i}$. Fix $\epsilon >0$. If an element $g\in C_c(\cG;\Sigma)$ satisfies $\pi(\suppo(g))\subseteq V_{1}
    V_{2}$ and  $\norm{f-g}_{\infty}<\epsilon$, then Lemma~\ref{lem:BPRRW's Lemma 15} implies 
    both $\norm{f-g}_{\red}<\epsilon$ and
    $\norm{f-g}_{\full}<\epsilon$;
    thus, if we can find such a $g$ in the range of $\iota_{1}\odot\iota_{2}$, we are done.

    Since $\lspan(X)$ is uniformly dense in $C_0(\Sigma)$, there exist finitely many $n_j\in\mathbb{N}^\times$ and finite collections $\{\tilde{f}_{i}^{j}\}_{j}\subseteq C_{c}(\cG_{i};\Sigma_{i};n_{j})$ for $i=1,2$ such that
    \begin{equation}\label{eq:approximating f}
        \norm{ f - \sum\nolimits_{j} \tilde{f}_{1}^{j}\circledast \tilde{f}_{2}^{j}}_{\infty}
        < 
        \epsilon.
    \end{equation}
    To construct the element $g \in \ran (\iota_{1}\odot\iota_{2})$, we will now cut down the  functions $\tilde{f}_{i}^{j}$ so that they  are supported in $\pi\inv(V_{1}V_{2})$.
    
    Let $h_{i}\in C_c(\cG_{i},[0,1])$ be identically $1$ on $U_{i}$ and zero outside of $V_{i}$, and consider the pointwise product $f_{i}^{j}\coloneq  (h_{i}\circ\pi_{i})\cdot\tilde{f}_{1}^{j}$.
    By construction,
    the open support of this continuous function satisfies 
    \[
        \suppo(f_{i}^{j}) \subseteq \pi\inv (V_{i})\cap \suppo(\tilde{f}_{i}^{j}),
    \]
    and is hence precompact by compactness of $\supp(\tilde{f}_{i}^{j})$. Since $h_{i}\circ\pi_{i}$ is $\mathbb{T}$-invariant, we conclude that $f_{i}^{j}$ is, like $\tilde{f}_{i}^{j}$, an element of $C_c(\cG_{i};\Sigma_{i};n_{j})$. Lastly,
    since $\pi(\suppo(f))\subseteq U_{1}U_{2}$,
    it is easy to see that Equation~\eqref{eq:approximating f} implies
    \begin{equation}\label{eq:approximating f,p2}
        \norm{ f - \sum\nolimits_{j} f_{1}^{j}\circledast f_{2}^{j}}_{\infty}
        < 
        \epsilon.
    \end{equation}
    Now that we have adjusted the supports, the next step in our hunt for $g$ in the range of $\iota_{1}\odot\iota_{2}$ is to make the functions $f_{i}^{j}$  $\mathbb{T}$-equivariant. To this end, we will use the map $T$ from Lemma~\ref{lem:making equiv fcts}.
    
    Since $f$ is $\mathbb{T}$-equivariant, we have $f=T(f)$, so that for any $e=e_{1}e_{2}\in \Sigma$,
    \begin{align*}
        &\left|
        \left(f - \sum\nolimits_{j}   T(f_{1}^{j})\circledast f_{2}^{j}
        \right)
        (e)\right|
        =
        \left| T(f)(e) - \sum\nolimits_{j}   T(f_{1}^{j})(e_{1})\, f_{2}^{j}(e_{2})\right|
        \\
        &=
        \left| \int_{\mathbb{T}}z f(z\cdot e)\,\mathrm{d}z
        -
        \sum\nolimits_{j}  
        \left(\int_{\mathbb{T}}z f_{1}^{j}(z\cdot e_{1})\,\mathrm{d}z\right)
        f_{2}^{j} (e_{2})
        \right|
        \\
        &\leq
        \int_{\mathbb{T}}
        \left| f(z\cdot e)
        -
        \sum\nolimits_{j}  
        f_{1}^{j}(z\cdot e_{1})
        f_{2}^{j} (e_{2})
        \right|\,\mathrm{d}z
        \\
        &=
        \int_{\mathbb{T}}
        \left| \left[ f
        -
        \sum\nolimits_{j}  
        f_{1}^{j}\circledast
        f_{2}^{j}\right] (z\cdot e)
        \right|\,\mathrm{d}z
        <
        \epsilon
        \quad\text{by~\eqref{eq:approximating f,p2}},
    \end{align*}
    and a second application of the same trick then shows that we have
    \[
        \left\| f - \sum\nolimits_{j}  T(f_{1}^{j})\circledast T(f_{2}^{j})\right\|_{\infty}
        <
        \epsilon.
    \]
    As $T(f_{i}^{j})\in C_c(\cG_{i};\Sigma_{i})$, the element 
    \[
    g\coloneq  \sum\nolimits_{j}  T(f_{1}^{j})\circledast T(f_{2}^{j})
    \]
    is in the range of $\iota_{1}\odot \iota_{2}$. Since 
    \[
        \suppo(T(f_{i}^{j}))
        \subseteq
        \mathbb{T}\cdot \suppo(f_{i}^{j})
        \subseteq 
        \pi\inv(V_{i}),
    \]
    we have
    \(
        \pi(\suppo(g))
        \subseteq 
        V_{1}V_{2}
    \), and so as argued earlier,
     it follows from $\norm{f-g}_{\infty}<\epsilon$ that $\norm{f-g}_{\red}<\epsilon$
     and $\norm{f-g}_{\full}<\epsilon$
     which finishes our proof that the range of $\iota_{1}\odot \iota_{2}$ is
      dense, both in $C^*_\red (\cG;\Sigma)$ and in $C^*(\cG;\Sigma)$.
\end{proof}

Consider
the statement in Theorem~\ref{thm:MT give C*-blends} about reduced twisted groupoid C*-algebras
 in the situation where $\cG_{1}\bowtie\cG_{2}$ is effective. By Corollary~\ref{cor:ZS and effectiveness}, $\cG_{1},\cG_{2}$ are effective as well, so by \cite[Theorem 5.2]{Renault:Cartan}, the three C*-algebras $C^*_{\red}(\cG_{1};\Sigma_{1})$, $C^*_{\red}(\cG_{2};\Sigma_{2})$, and $C^*_{\red}(\cG_{1}\bowtie\cG_{2};\Sigma_{1}\ZST \Sigma_{2})$ all share the same Cartan subalgebra, $C_0(\cU)$. Note that this Cartan is exactly the intersection of the two subalgebras $C^*_{\red}(\cG_{1};\Sigma_{1})$ and $C^*_{\red}(\cG_{2};\Sigma_{2})$. It turns out that this is no coincidence,
 as we will see in Theorem~\ref{thm:from C*-blend}.

\subsection{Cartan subalgebra and Weyl groupoids}

Let us start by recalling some basic definitions of Cartan subalgebras and the associated Weyl groupoid
and twist.

\begin{definition}[{\cite[Definition 5.1]{Renault:Cartan}, \cite[Definition 2.8]{pitts2022normalizers}}] Let $A$ be a C*-algebra. A subalgebra $D$ is called a Cartan subalgebra if 
\begin{enumerate}
    \item $D$ is maximal abelian inside $A$;
    \item there exists a faithful conditional expectation from $A$ to $D$; and 
    \item the normalizer $N_A(D)=\{n\in A: ndn^*, n^*dn\in D \text{ for all }d\in D\}$ generates $A$ as a C*-algebra.
\end{enumerate}
\end{definition}

By \cite[Theorem 5.2]{Renault:Cartan} (see also \cite[Theorem 1.2]{Raad:2022:Renault} for the non-separable case), any
Cartan pair $(A,D)$ can be realized as $(C_r^*(\cG; \Sigma), C_0(\cU))$, where 
$\Sigma\to\cG$
is a twisted groupoid with unit space $\cU$ and $\cG$ effective. Note that, in particular, the space $\cU$ is exactly
 the Gelfand spectrum $\cU=\widehat{D}$. The groupoids $\cG$ and $\Sigma$ are called the \emph{Weyl groupoid} and the \emph{Weyl twist}, respectively. We will recall 
 how to construct this Weyl pair 
$\Sigma\to\cG$
from $(A,D)$;
 one can refer to \cite{Renault:Cartan} for more detailed discussions.

Fix $n\in N_A(D)$. Since $D$ contains an approximate identity for $A$ \cite[Theorem 2.6]{pitts2022normalizers} and is closed,
$n^*n$ is an element of $D\cong C_0(\widehat{D})$. In particular, for $x\in \widehat{D}$, we may consider $n^*n(x)\in\mathbb{C}$ and
define 
\[\dom(n)\coloneq \left\{x\in\widehat{D}: n^*n(x) > 0\right\}.\]
By \cite{Kum:Diags}, there exists a unique homeomorphism $\alpha_n\colon \dom(n) \to \dom(n^*)$ such that for all $d\in D$ and $x\in\dom(n)$,
\[n^*dn (x) = d(\alpha_n(x)) n^*n(x).\]
Write 
\begin{equation}\label{eq:E}
    E\coloneq\{(n,x): n\in N_{A}(D), x\in \dom(n)\}
\end{equation}
and define the following equivalence relations\footnote{See \cite[Proposition 2.2]{DGN:Weyl}, which proves that definition of ${\sim}$ given here indeed coincides with that given in \cite{Renault:Cartan}.} $\sim$ and $\approx$ on $E$:
\begin{align*}
(n,x) \sim (m,y) \text{ if } &x=y\text{ and
there exist }d,d'\in \widehat{D}
\text{ such that }
\\& d(x),d'(x)\neq 0\text{ and } nd=md';
\\
(n,x) \approx (m,y) \text{ if } &x=y\text{ and
there exist }d,d'\in \widehat{D}
\text{ such that }
\\& d(x),d'(x)>0\text{ and } nd=md'.
\end{align*}
The Weyl twist $\Sigma$ 
of $(A,D)$ is defined as the quotient space $E/{\approx}$. We denote the equivalence class of $(n,x)\in E$ under $\approx$ by $\llbracket n,x\rrbracket$, and give $\Sigma$ the structure maps
\[
    \llbracket m,\alpha_n(x)\rrbracket\llbracket n,x\rrbracket=\llbracket mn,x\rrbracket
    \quad\text{and}\quad
    \llbracket n,x\rrbracket\inv = \llbracket n^*,\alpha_n(x)\rrbracket.
\]
The basic open sets for the topology are
the sets $W_\Sigma(n,U,V)$, indexed by $n \in N_A (D)$ and open sets $U \subseteq\mathbb{T}$ and $V\subseteq\widehat{D}$ and defined by
\[
W_\Sigma(n,U,V)=\{\llbracket zn, x\rrbracket : z \in U, x\in V\cap \dom(n) \}.
\]
Similarly, the Weyl groupoid $\cG$ is defined as the quotient space $E/{\sim}$; its elements are denoted $[n,x]$, and its structure maps are given by the same formulas as those for $\Sigma$, replacing double-brackets with regular brackets. We give $\cG$ the topology that makes the surjective map
\[
\Sigma\overset{\pi}{\to}\cG,  \quad\llbracket n, x\rrbracket\mapsto [n,x],
\]
continuous; it is automatically open. 
Note that
the source of an element $\llbracket n, x\rrbracket$ of $\Sigma$ is given by
\[
s(\llbracket n,x\rrbracket) = \llbracket n^*n,x\rrbracket.
\]
As $n^*n(x)>0$, we see that
the unit space $\cU$ of $\Sigma$ can be identified with $\widehat{D}$. More precisely, for $x\in\widehat{D}$, choose any $f_x\in D$ with $f_x(x)>0$. Then the map $x\mapsto \llbracket f_x,x\rrbracket$ is a homeomorphism $\widehat{D}\to \cU$.
To realize  $\Sigma$ as a twist
\[
\mathbb{T}\times \cU    \stackrel{\jmath}{\longrightarrow} \Sigma  \stackrel{\pi}{\longrightarrow} \cG
\]
over $\cG$, we let $\jmath(z,x)=\llbracket f, x\rrbracket$ where $f\in D$ is any element such that $f(x)=z$.

\begin{remark}\label{rmk:W_Sigma}
    When describing a topology base for $\Sigma$, we do not actually need the leeway of an open neighborhood $V$ around $x\in\dom(n)$, as long as $n$ is allowed to vary over all of $N_{A}(D)$. Indeed, for any $f\in D$ and $n\in N_{A}(D)$, the element $fn$ is also a normalizer. Since $\llbracket zn,x\rrbracket=\llbracket z(nf),x\rrbracket$ for any $f\in D$ with $f(x)>0$, any element $f\in C_0(\widehat{D},[0,1]) $ with $\suppo(f)=V$ yields
    \[
        W_\Sigma(n,U,V)=
        \{\llbracket z(nf), x\rrbracket : z \in U, x\in V\cap \dom(n)=\dom(nf) \}
        =
        W_\Sigma(nf,U,\widehat{D}).
    \]
\end{remark}

We will soon need the following easy corollary of \cite[Proposition 4.1]{DGNRW:Cartan}:
\begin{lemma}\label{lem:DGNRW-Prop4.1}
    If $(A,D)$ is a Cartan pair and a set $N\subseteq N_A(D)$ densely spans $A$, then every element of the Weyl twist can be written as $\llbracket z n, x\rrbracket$ for some $z\in\mathbb{T}$, $n\in N $, and $x\in\dom(n)$.
\end{lemma}

\begin{proof}
   Take any element $e$ of the Weyl twist. It was shown in \cite{BG:2022:DGNRW}  that the element $\pi(e)$ of the associated Weyl groupoid can be written as $[n,x]$ for some $n\in N$. As $\pi(e)=[n,x]=\pi(\llbracket n, x\rrbracket)$, there exists $z\in\mathbb{T}$ such that $e=z\cdot \llbracket n,x\rrbracket= \llbracket zn,x\rrbracket$.
\end{proof}

The main goal of this section is to prove the following theorem
which can be understood as a partial converse to Theorem~\ref{thm:MT give C*-blends}.

\begin{theorem}\label{thm:from C*-blend}
    Suppose that $(A_{1}; A_{2}; A)$ is an \adjforBlend\ C*-blend.
    Assume that $D\coloneq  A_{1}\cap A_{2}$ is a Cartan subalgebra of $A_{1}, A_{2}$, and $A$; let 
        $\Sigma_{i}\to\cG_{i}$
    denote the Weyl pair of the Cartan pair  $(A_{i}, D)$,
    and $\Sigma\to\cG$ that of $(A,D)$.
    Then $(\Sigma_{1},\Sigma_{2})$ is an internal \ZS\ structure for $\Sigma$. 
    Moreover, the map $\Phi\colon \Sigma_{2} \ast_{\mathbb{T}} \Sigma_{1} \to \Sigma_{1} \ast_{\mathbb{T}} \Sigma_{2}$ given for $m_{j}\in N_{j}$ and $x\in \dom(m_{2}m_{1})$ by
    \[
        \Phi
        \left(
        \llbracket m_{2},\alpha_{m_{1}}(x)
        \rrbracket
        \asttwoone
        \llbracket m_{1},x
        \rrbracket
        \right)
        =
        \llbracket n_{1},\alpha_{n_{2}}(x)
        \rrbracket
        \asttwoone
        \llbracket n_{2},x
        \rrbracket
        \quad\text{ where }m_{2}m_{1}g=n_{1}n_{2}
    \]
    for some $n_{j}\in N_{j}$ and $g\in D$ with $g(x)>0$, is the unique \wordforPhi\ 
    such that $\Sigma\cong \Sigma_{1}\ZST\Sigma_{2}$ as twists.
\end{theorem}

\begin{notation*}
In the following, we will work in the setting of Theorem~\ref{thm:from C*-blend}. We will let $\cU=\widehat{D}$ denote the Gelfand spectrum of $D$.
As in~\eqref{eq:E}, we let $E$ be 
the set of representatives for the elements of the Weyl pair
\[\mathbb{T}\times\cU \overset{\jmath}{\to} \Sigma \overset{\pi}{\to}\cG\]
associated to the specific Cartan pair $(A,D)$
from Theorem~\ref{thm:from C*-blend}.

If we add a subscript-$j$ ($j=1,2$) to any  item, then we mean the corresponding item for the Cartan pair $(A_{j},D)$. Lastly, we let
   \(
   N_{j}= N_{A_{j}}(D)
   \) 
   be the normalizers in $A_{j}$. Note that for all six groupoids that we are considering, the unit space can be identified with $\cU$.
\end{notation*}
  
Since the C*-blend is assumed to be \adjforBlend, we can without loss of generality assume $A_{j}$ with a sub-C*-algebra of $A$. Thus
$N_{j} \subseteq N_A(D)$, and so $E_{j}\subseteq E$. We will prove that this implies that the Weyl groupoid and twist of $(A_{j},D)$ are canonically isomorphic to clopen subgroupoids of the Weyl groupoid $\cG$ respectively the Weyl twist $\Sigma$ of $(A,D)$.

\begin{lemma}\label{lem:open subgroupoids}
     In the setting of Theorem~\ref{thm:from C*-blend}, 
    the inclusion $E_{j}\subseteq E$ ($j=1,2$) factors through injective, continuous, open groupoid homomorphisms $
    \iota_{{j}}^{\Sigma}\colon 
    \Sigma_{j}\to \Sigma$ and $\iota_{{j}}^{\cG}\colon
    \cG_{j}\to \cG$, and we have $\iota_{{j}}^{\Sigma}(\Sigma_j)=\pi\inv(\iota_{{j}}^{\cG}(\cG_{j}))$.
\end{lemma}
\begin{proof}
    From the definition of 
   the equivalence relation that gives rise to $\Sigma$ and $\Sigma_{j}$, it is
    clear that two elements of $E_{j}$ are equivalent in $E$ if and only if they are equivalent in $E_{j}$.      In other words, the inclusion $E_{j}\subseteq E$ factors through a map 
    \[
    \iota\coloneq  \iota_{{j}}^{\Sigma} \colon \Sigma_{j}\to \Sigma,
    \quad
    \llbracket n,x\rrbracket_{j} \mapsto \llbracket n,x\rrbracket.
    \]
    Since $N_{j} \subseteq N_A(D)$, we
    have $\llbracket n,x\rrbracket_{j} \subseteq \llbracket n,x\rrbracket$; the map $\iota $ is nevertheless injective.
    Since the product and inversion of $\Sigma_{j}$ is exactly the product resp.\ inversion of $\Sigma$ when restricted to $\Sigma_{j}$, we conclude that $\iota$ is a homomorphism. 
    
To see continuity, fix a basic open set $W
= W_{\Sigma}(n,U,\widehat{D})
$ of $\Sigma$
(see Remark~\ref{rmk:W_Sigma}),
i.e.,
    \[
    W= \left\{\llbracket zn, x\rrbracket : z \in U, x\in 
    \dom(n) \right\}
    \]
    for some fixed $n\in N_A(D)$, $U\subseteq \mathbb{T}$ open.
    Suppose that $\llbracket m , x_0\rrbracket_{j}\in\iota\inv (W)$, so $m\in N_{j}$ and $x_0\in 
    \dom(n)\cap \dom(m )$, and there exist $z_0 \in U$ and $f,g\in D$ with 
    \begin{equation}\label{eq:z_0nf=mg}
        f(x_0), g(x_0) >0
        \quad\text{and}\quad
        (z_0 n)f = m  g.
    \end{equation}
    By replacing $g$ with $gg^*$ and $f$ with $fg^*$, we can without loss of generality assume that $g\geq 0$
    everywhere.

    Since the function 
    \[\mathbb{T}\times\suppo(f)\to \mathbb{T},\quad (z,x)\mapsto z z_0 f(x)/|f(x)|,
    \]
    is continuous, the set
    \[
        \{ (z,x) \in \mathbb{T}\times\suppo(f) : z z_0 f(x)/|f(x)| \in U\} \cap
        \mathbb{T}\times \bigl(\suppo(g)
        \cap \dom(n)\bigr)
    \]
    is open. It furthermore contains $(1,x_0)$, as $f(x_0)>0$ and $z_0\in U$; in particular, there exist open neighborhoods $U'\subseteq \mathbb{T}$ of $1$ and 
    \[
    V'\subseteq \suppo(f)\cap \suppo(g) 
    \cap \dom(n)\subseteq \cU
    \]
    of $x_0$ such that, whenever $(z,x)\in U'\times V'$, then $z z_0 f(x)/|f(x)|\in U$. Fix any such $(z,x)$ and let $f'\coloneq  (\overline{f(x)}/|f(x)|)f$ and $z_{1} \coloneq  z z_0 f(x)/|f(x)|$. By choice of $U'\times V'$, $z_0$ is an element of $U$. We have
    \begin{align*}
        z_{1}nf'
        &=
        \left(z z_0 f(x)/|f(x)|\right) n \left(\overline{f(x)}/|f(x)|\right)f
        =
        (z z_0 ) n f
        \overset{\eqref{eq:z_0nf=mg}}{=}
        zmg.
    \end{align*}
    and since $g(x),f'(x)>0$, this implies
    \[
        \llbracket z m, x\rrbracket
        =
        \llbracket z_{1} n, x\rrbracket.
    \]
   As $z_{1}\in U$, this shows that $\llbracket z m, x\rrbracket\in W$. In summary, we have shown that
    \[
    \left\{\llbracket z m, x\rrbracket_{j} : z \in U', x\in V' \cap \dom(m) \right\}
    \subseteq \iota\inv(W).
    \]
    The set on the left-hand side is a basic open set of $\Sigma_{j}$, and
    since $U'$ is a neighborhood of $1\in \mathbb{T}$ and $V'$ is one for $x_0\in \widehat{D}$, this basic open set contains $\llbracket m,x_0\rrbracket_{j}$. Since
    $\llbracket m , x_0\rrbracket_{j}$ was arbitrary in $\iota\inv (W)$, this proves that $\iota\inv (W)$ is open in $\Sigma_{j}$, i.e., $\iota $ is continuous.

    To see that $\iota$ is an open map, just note that
    \begin{align*}
        \iota 
        \bigl(
        \left\{
            \llbracket zn,x\rrbracket_{j}
            : 
            z\in U, x\in\dom(n)
        \right\}
        \bigr)
        =
        \left\{
            \llbracket zn,x\rrbracket
            : 
            z\in U, x\in\dom(n)
        \right\},
    \end{align*}
    meaning that basic open sets are mapped to basic open sets.
    
    The same proofs go through for the map 
    \(
    [ n,x]_{j} \mapsto [ n,x]
    \) on the level of Weyl groupoids, so this map is likewise an injective, continuous, open groupoid homomorphism. (Alternatively, one can use that the surjective maps from the Weyl twists down to the Weyl groupoids are open maps, and then invoke the above.)

    For the last claim, take an arbitrary element $\llbracket n, x\rrbracket$ of $\pi\inv (\iota^{\cG}(\cG_{j}))$. Then there exists $m\in N_{j}$ and $f,g\in D$ with $f(x)\neq 0 \neq g(x)$ such that $nf=mg$. If we let $m'=g(x)/f(x) m  $, then 
    \begin{align}\label{eq:property of m'}
        m'\frac{g}{g(x)}
        =
        \left(
        \frac{g(x)}{f(x)} m
        \right)\frac{g}{g(x)}
        =
        m\frac{g}{f(x)}
        =
        n\frac{f}{f(x)}.
    \end{align}
    Since $N_{j}$ is closed under scalar multiplication, we have $m'\in N_{j}$,
    so that Equation~\eqref{eq:property of m'} implies $\llbracket n, x\rrbracket=\llbracket m', x\rrbracket \in \iota(\Sigma_{j})$
    which proves  $\pi\inv (\iota_{{j}}^{\cG}(\cG_{j})) \subseteq \iota_{{j}}^{\Sigma}(\Sigma_{j})$. The other inclusion is obvious.
\end{proof}

Since $\iota^{\cG}$ and  $\iota^{\Sigma}$ are embeddings of topological groupoids, we will from now on identify  $\cG_{j}$ and $\Sigma_{j}$ with their images in $\cG$ and $\Sigma$, respectively.
To see that $\cG$ is indeed the \ZS\ product of its subgroupoids $\cG_{1}$ and $\cG_{2}$, we first check that their intersection is trivial:

\begin{lemma}\label{lem:intersection}
    In the setting of Theorem~\ref{thm:from C*-blend}, the intersection of the subgroupoids $\cG_{1}$ and $\cG_{2}$ of $\cG$ is the unit space $\cU$. 
\end{lemma}

\begin{proof}
    Assume  that $\gamma\in \cG_{1}\cap \cG_{2}$, so we may write $\gamma=[n_{1},x]=[n_{2},x]$ with $n_{j}\in N_{j}$ and $x\in \dom(n_{1})\cap\dom(n_{2})$.  The equality says that there exists $f_{j}\in D$ with $f_{j}(x)\neq 0$ and $n_{1}f_{1}=n_{2}f_{2}$. Note that the left-hand side is an element of $A_{1}$ and the right-hand side is an element of $A_{2}$, so $n_{1}f_{1} \in A_{1}\cap A_{2}=D$, and since $x\in \dom(n_{1})$, this element of $D$ satisfies $|(n_{1}f_{1})(x)|^2=(n_{1}^*n_{1})(x)|f_{1}(x)|^2\neq 0$. To sum up, we have found an element $f_{1}\in D$ with $f_{1}(x)\neq 0$ and such that $n_{1}f_{1}\in D$ does not vanish at $x$, which means exactly that $\gamma=[n_{1},x]\in\cU$.
\end{proof}

\begin{proof}[Proof of Theorem~\ref{thm:from C*-blend}]
  Since the inclusion maps are open by Lemma~\ref{lem:open subgroupoids}, the subgroupoids $\cG_{j}$ of $\cG$ are open.  
    Since   $\pi\inv(\cG_{j}) = \Sigma_{j}$ by Lemma~\ref{lem:open subgroupoids}, we conclude that $\Sigma_{j}$ is likewise open.
   Once we have shown that $(\Sigma_{1},\Sigma_{2})$ is an internal \ZS\ structure for $\Sigma$, then the remaining claims of Theorem~\ref{thm:from C*-blend} follow from an application of Theorem~\ref{thm:IZS=MT}\ref{it:thm:from IZS to MT}.   
   
    Since $D$ is Cartan in $A_{j}$, we know that $N_{j}$ densely spans $A_{j}$. Combining this with Assumption~\ref{it:C*B:dense range} and a (finicky but easy) $\epsilon/k$-argument, we conclude that the set
    \[
        N\coloneq  \{ n_{1}n_{2} : m \in N_{j}\}
    \]
    densely spans $A$.  
    Since $N_{j}\subseteq N_{A}(D)$ and since $N_{A}(D)$ is closed under multiplication, we conclude that $N\subseteq N_{A}(D)$. Since $N_{j}$ is closed under scalar multiplication, it now follows from Lemma~\ref{lem:DGNRW-Prop4.1} that any element of $\Sigma$ can be represented as $\llbracket n_{1}n_{2},x \rrbracket$ for some $n_{j} \in N_{j}$.
    Fix one such element.

    Since $n_{1}$ and $n_{2}$ are both normalizers of $D$, the fact that $\llbracket n_{1}n_{2},x\rrbracket\in \Sigma$ in particular means that 
    \begin{equation}\label{eq:dom n_{1} n_{2}}
        x\in \dom(n_{1}n_{2})
        =
        \{ y \in \dom (n_{2}) : \alpha_{n_{2}}(y)\in \dom(n_{1})\},
    \end{equation}
    so we can write
    \[
        \llbracket n_{1}n_{2},x \rrbracket
        =
        \llbracket n_{1},\alpha_{n_{2}}(x)\rrbracket 
        \llbracket n_{2},x\rrbracket .
    \]
    This shows that we can decompose the arbitrary element $\llbracket n_{1}n_{2},x\rrbracket$ of $\Sigma$ into the product of elements $\llbracket n_{1},\alpha_{n_{2}}(x)\rrbracket \in \Sigma_{1}$ and $
        \llbracket n_{2},x\rrbracket \in \Sigma_{2}$, i.e., Condition~\ref{it:IZS:product} holds.

    Since $(\Sigma,\jmath,\pi)$ is a short exact sequence, Condition~\ref{it:IZS:cap} is equivalent to showing that $\pi(\Sigma_{1})\cap\pi(\Sigma_{2})\subseteq \cU$, which we have done in Lemma~\ref{lem:intersection}. 
    It remains to show that $\Sigma_i$ is closed in $\Sigma$; we will do so for $i=1$, as the other proof is analogous. In a nutshell, our proof uses Condition~\ref{it:IZS:cap} and the fact that $\jmath(\mathbb{T}\times\cU)=\pi\inv(\cU)$ is closed in $\Sigma$.
    
    Assume $\{\llbracket n_\lambda, x_\lambda\rrbracket\}_\lambda$ is a net of elements in $\Sigma_{1}$ that converges to an element $\llbracket n_{1}n_{2},x\rrbracket$ in $\Sigma$, where $n_i\in N_i$. Then for any fixed neighborhoods $U\subset \mathbb{T}$ of $1$ and $W\subset \dom(n_{1}n_{2})$ of $x$, we must have for large $\lambda$ that
    \begin{equation}
        \label{eq:net in nbhd}
    \llbracket n_\lambda, x_\lambda\rrbracket
    \in 
    \{\llbracket zn_{1}n_{2}, y\rrbracket : z \in U, y\in W \} 
    .
    \end{equation}
    In particular, $x_\lambda \to x \in \dom(n_{1}n_{2})$, so without loss of generality we have $x_\lambda\in W\subset \dom(n_{1}n_{2})$ for all $\lambda$; in particular by~\eqref{eq:dom n_{1} n_{2}},
    \begin{equation}\label{eq:x_lambda,x}
        x_\lambda,x\in\dom(n_{2}) \cap \alpha_{n_{2}}\inv(\dom(n_{1}))  .      
    \end{equation}

    By~\eqref{eq:net in nbhd},    
    there exist $f_\lambda,g_\lambda\in D$ with $f_\lambda(x_\lambda)>0, g_\lambda(x_\lambda)\in U$ and $n_\lambda f_\lambda = n_{1}n_{2} g_\lambda$.  
    Choose an open set $V$ around $\alpha_{n_{2}}(x)$ whose closure is contained entirely in $\dom(n_{2}^*)$; in particular, $\alpha_{n_{2}}\inv(V)$ is a neighborhood around $x$, so without loss of generality, the entire net $\{x_\lambda\}_\lambda$ lies in $\alpha_{n_{2}}\inv(V)$. Let $h\in D$ be a $[0,1]$-valued function such that $h|_{V}\equiv 1$ and $\suppo(h)\subseteq \dom(n_{2}^*)$. By choice of $f_\lambda$ and $g_\lambda$, we have
    \begin{equation}\label{eq:finicky}
        hn_{1}^*(n_\lambda f_\lambda) = hn_{1}^*(n_{1}n_{2} g_\lambda)= (hn_{1}^*n_{1})n_{2} g_\lambda.
    \end{equation}
    Since $hn_{1}^*n_{1} \in D$ has support in $V\subseteq\dom(n_{2}^*)$, it follows from \cite[Lemma 4.2]{DGNRW:Cartan} that  $k
    \coloneq (hn_{1}^*n_{1})\circ \alpha_{n_{2}}$ is a well-defined element of $D$ and that $(hn_{1}^*n_{1})n_{2} = n_{2} k
    $. Since $x_\lambda\in \alpha_{n_{2}}\inv(V \cap \dom(n_{1}))$, we have 
    \begin{equation}\label{eq:h and x_lambda}
    h(\alpha_{n_{2}}(x_\lambda)) = 1
    \end{equation}
    and $(n_{1}^*n_{1})(\alpha_{n_{2}}(x_\lambda)) > 1$, so that $k
    (x_\lambda)>0$.
    As $n_\lambda\in N_{1}$, the element $m_\lambda\coloneq hn_{1}^*n_\lambda$ is in $N_{1}$, and we have
    \[
        \dom (m_\lambda)
        =
        \{
            y\in \dom(n_\lambda):
            \alpha_{n_\lambda}(y)\in \dom (n_{1}^*)
            \cap \alpha_{n_{1}}(\suppo(h))
        \}.
    \]
    Note that $
    x_\lambda\in \dom(n_\lambda)$ satisfies by Equations~\eqref{eq:x_lambda,x} and~\eqref{eq:h and x_lambda}
    \[
        \alpha_{n_{1}n_{2}}(x_\lambda)
        =
        \alpha_{n_{1}}(\alpha_{n_{2}}(x_\lambda))
        \in 
        \dom (n_{1}^*)
            \cap \alpha_{n_{1}}(\suppo(h)).
    \]
    Furthermore, it follows from $n_\lambda f_\lambda = n_{1}n_{2}g_\lambda$ and $f_\lambda(x_\lambda)\neq 0 \neq g_\lambda(x_\lambda)$ that $\alpha_{n_\lambda}(x_\lambda)=\alpha_{n_{1} n_{2}}(x_\lambda)$. We thus
    conclude that $x_\lambda\in\dom(m_\lambda)$.
    By Equation~\eqref{eq:finicky} and choice of $k$, 
    \begin{align*}
        m_\lambda f_\lambda = hn_{1}^*n_\lambda f_\lambda = (hn_{1}^*n_{1})n_{2}g_\lambda = n_{2} k g_\lambda.
    \end{align*}
    Since $f_\lambda(x_\lambda)>0$ and $k(x_\lambda)>0$ and since $g_\lambda(x_\lambda)\in U$, this shows that
    \[
    \xi_\lambda\coloneq
    \llbracket m_\lambda,x_\lambda\rrbracket
    =
    \llbracket 
    g_\lambda(x_\lambda)n_{2},x_\lambda\rrbracket
    \in
    \{\llbracket zn_{2}, y\rrbracket : z \in U, y\in W \} 
    .
    \]
    Since $U$ and $W$ were arbitrary neighborhoods around $1$ and $x$, respectively,
    we have shown that the net $\{\xi_\lambda\}_\lambda$  converges to $\llbracket n_{2},x\rrbracket$, an element of $\Sigma_{2}$. Since we have shown that $\Sigma_{2}$ is open
    (see the first paragraph of this proof),
    we must eventually have that $\xi_\lambda\in\Sigma_{2}$. Since $\xi_\lambda$ also lies in $\Sigma_{1}$ by construction, it follows from \ref{it:IZS:cap} that $\xi_\lambda\in \jmath(\mathbb{T}\times\cU)$. Since $\jmath(\mathbb{T}\times\cU)=\pi\inv(\cU)$ is closed in $\Sigma$, we conclude that $\llbracket n_{2},x\rrbracket=\lim_\lambda \xi_\lambda $ lies in $\jmath(\mathbb{T}\times\cU)$. This means that $\alpha_{n_{2}}(x)=x$ and
    \[
    \lim_\lambda \llbracket n_\lambda,x_\lambda\rrbracket
    = \llbracket n_{1}n_{2},x\rrbracket
        =
        \llbracket n_{1},\alpha_{n_{2}}(x)\rrbracket 
        \llbracket n_{2},x\rrbracket=
    \llbracket n_{1},x\rrbracket\in \Sigma_{1},
    \]
    as claimed.

    Thus, $(\Sigma_{1},\Sigma_{2})$ is an internal \ZS\ structure for $\Sigma$. It now follows from \ref{it:IZS:ZS} that $\cG=\mathmbox{\cG_{1}\bowtie\cG_{2}}$ and from Theorem~\ref{thm:IZS=MT}\ref{it:thm:from IZS to MT} that the map $\Sigma_{1}\ZST \Sigma_{2}\to \Sigma$ given by $ e\astonetwo f \mapsto ef$ is an isomorphism of twists
    for a unique \wordforPhi\ $\Phi$.
    It remains to check that $\Phi$ is as claimed. 
    
    By Theorem~\ref{thm:IZS=MT}\ref{it:thm:from IZS to MT}, we have
    \(
    \Phi(f\asttwoone e) = e' \astonetwo f' \)
    whenever $fe=e'f'$ in $\Sigma$ for some $e'\in \Sigma_{1}$ and $f'\in \Sigma_{2}$.
    Now, let us write the elements explicitly, say $f=\llbracket m_{2},\alpha_{m_{1}}(x)
        \rrbracket
        $, $e=
        \llbracket m_{1},x
        \rrbracket$,  $
        e'=\llbracket n_{1},\alpha_{n_{2}}(x)
        \rrbracket
        $, and $f'=
        \llbracket n_{2},x
        \rrbracket$ for  some $m_j,n_j\in N_j$. Then the equality $fe=e'f'$ means that there exists $g,g'\in D$ such that $g(x),g'(x)>0$ and $m_{2}m_{1}g=n_{1}n_{2}g'$. By construction, $n_{2}'\coloneq n_{2}g'\in N_{2}$ satisfies  $f'=\llbracket n_{2}',x
        \rrbracket$, and so $\Phi$ is as claimed. 
\end{proof}

\begin{corollary}
    Suppose that $(A_{1}; A_{2}; A)$ is an \adjforBlend\ C*-blend and assume
    that $D\coloneq  A_{1}\cap A_{2}$ is a Cartan subalgebra of $A$. Then $D$ is Cartan in both $A_{1}$ and $A_{2}$ if and only if there exist conditional expectations $A\to A_{1}$ and $A\to A_{2}$.
\end{corollary}
\begin{proof}
    If there exist conditional expectations $A\to A_j$, then $D$ is Cartan in $A_{j}$ by \cite[Theorem 3.5]{BEFPR:2021:Intermediate}.

    Conversely, suppose $D$ is Cartan in $A_{1}$ and $A_{2}$. By Theorem~\ref{thm:from C*-blend}, the associated Weyl groupoids $\cG_{1},\cG_{2}$ are subgroupoids of the Weyl groupoid $\cG$ of the Cartan pair $(A,D)$, and the Weyl twists $\Sigma_{1},\Sigma_{2}$ are the restrictions $\Sigma|_{\cG_{1}}, \Sigma|_{\cG_{2}}$ of the Weyl twist $\Sigma$ of $(A,D)$. These subgroupoids are clopen by \etale ness of $\cG$; see Lemma~\ref{lem:clopen in ZS-product if r-discrete}. By  \cite[Lemma 3.4]{BEFPR:2021:Intermediate}, there therefore exist conditional expectations $C^*_{r}(\cG;\Sigma) \to C^*_{r}(\cG_{j};\Sigma_{j})$ which, by \cite[Theorem 5.9]{Renault:Cartan}, are conditional expectations $A\to A_{j}$.
\end{proof}

\begin{remark}
    Given a discrete (abelian) group $\Gamma$, we suspect that a $\Gamma$-graded version of Theorem~\ref{thm:from C*-blend} can be proved, where Cartan subalgebras are replaced with $\Gamma$-Cartans in the sense of \cite[Definition 3.9]{BFPR:GammaCartan}, and groupoids with $\Gamma$-graded groupoids. (See \cite[Prop.\ 5.3., Thm.\ 4.19, Thm.\ 6.2]{BFPR:GammaCartan} for the correspondence between $\Gamma$-graded twists and $\Gamma$-Cartan pairs of C*-algebras.)
    
    On the groupoid level, it is easy to check that a pair of $\Gamma$-gradings  $\vartheta_{i}\colon \cG_{i}\to\Gamma$ gives rise to a (necessarily unique) $\Gamma$-grading of $\cG_{1}\bowtie\cG_{2}$ that makes  the diagram
    \begin{equation}\label{diag:Gamma-graded bowtie}
    \begin{tikzcd}[column sep = large, row sep = large
    ]
        \cG_{1}
        \ar[r, hook]
        \ar[rd, "\vartheta_{1}"']
        &
        \cG_{1}\bowtie\cG_{2}
        \ar[d, "\vartheta", dashed]
        &
        \cG_{2}
        \ar[l, hook']
        \ar[ld, "\vartheta_{2}"]
        \\
        & \Gamma &
    \end{tikzcd} \end{equation}
    commute,
if and only if the pair of gradings satisfies
        \begin{equation}\label{eq:varthetas}
            \vartheta_{2}(g)\vartheta_{1}(y)
            =
            \vartheta_{1}(g\HleftG y)\vartheta_{2}(g\HrightG y)
        \end{equation}
    in $\Gamma$ for all $(g,y)\in \cG_{2}\bfpsr\cG_{1}$. On the level of a `$\Gamma$-graded C*-blend', the grading $\vartheta$ is known to correspond to a $\Gamma$-grading $\{A^{\gamma}\}_{\gamma\in \Gamma}$ of $A$, but we are unsure what Equation~\eqref{eq:varthetas} translates to in terms of the subalgebras $A_{1},A_{2}$ of $A$.
    Since the current setting of our theorem is rather technical as-is, we refrained from investigating this idea further.
\end{remark}

\appendix

\section{The \ZS\ product of a Fell line bundle and a groupoid}\label{appendix:FellBdl}

We still owe the reader a proof of the claims in
Examples~\ref{ex:MT if one is trivial} and~\ref{ex:MT if one is trivial,pt2}. Recall that we are in the following setting in both examples: $(\cG_{1},\cG_{2})$ is a pair of matched groupoids with unit space $\cU$, $\Sigma_{2}=\mathbb{T}\times \cG_{2}$ is the trivial twist, and $\Sigma_{1}$ is a twist over $\cG_{1}$. We let $L_{1}=\mathbb{C}\times_{\mathbb{T}}\Sigma_{1}$ be the associated line bundle.

\subsection{Regarding Example~\ref{ex:MT if one is trivial}} We claim that a \wordforPhi\ $\Phi\colon \Sigma_{2}\ast_{\mathbb{T}}\Sigma_{1}\to \Sigma_{1}\ast_{\mathbb{T}}\Sigma_{2}$
{that covers $\cG_{1}\bowtie\cG_{2}$} contains the same data as a $(\cG_{1},\cG_{2})$-compatible $\cG_{2}$-action on $L_{1}$ 
\[
    \mvisiblespace \HleftB \mvisiblespace \colon \cG_{2}\bfp{s}{\rho} L_{1} \to L_{1},
    \quad
    (g,[\lambda,e]) \mapsto g\HleftB [\lambda,e].
\]
So assume first that we are given a
\wordforPhi\
$\Phi\colon \Sigma_{2}\ast_{\mathbb{T}}\Sigma_{1}\to \Sigma_{1}\ast_{\mathbb{T}}\Sigma_{2}$.
Given an arbitrary $(g,[\lambda,e])\in \cG_{2}\bfp{s}{\rho} {L_{1}}$, the element $\Phi((1,g)\asttwoone e)$ of $\Sigma_{1}\ast_{\mathbb{T}} \Sigma_{2}$ can be written uniquely as $f\astonetwo (1,h)$ for some $f\in\Sigma_{1}$ and $h\in\cG_{2}$ thanks to the $\mathbb{T}$-balancing. 
We may thus let
\begin{equation}
\label{eq:def HleftB}
g\HleftB [\lambda,e] \coloneq  [\lambda ,f]
\quad\text{ where $f$ is such that } f\astonetwo (1,h) = \Phi((1,g)\asttwoone e)
.
\end{equation}
It follows from $\mathbb{T}$-equivariance of $\Phi$ that $\HleftB $ is well-defined: given another representative $(\overline{z}\lambda,z\cdot e)$ of $[\lambda,e]$ in ${L_{1}}$, we have
\[
    \Phi((1,g)\asttwoone (z\cdot e))
    \overset{\ref{it:MT:T equivariant}}{=}
    z\cdot (f\astonetwo (1,h)) =
    (z\cdot f)\astonetwo (1,h),
\]
so that the equality $[\lambda ,f] = [\overline{z}\lambda ,z\cdot f]$ in ${L_{1}}$ proves that indeed
\[
g \HleftB [\overline{z}\lambda,z\cdot e] =g \HleftB [\lambda,e].
\]
Continuity of $\Phi$ implies continuity of $\HleftB $.

\begin{claim}\label{claim:HleftB}
    $\HleftB$ satisfies Conditions~\textup{(A1)--(A5)} of \cite[Definition 3.1]{DuLi:ZS}, so that it is a $(\cG_{1},\cG_{2})$-compatible $\cG_{2}$-action on the Fell bundle ${L_{1}}$.
\end{claim}

\begin{proof}[Proof of the claim.]
Assume throughout that $\Phi((1,g)\asttwoone e)=f\astonetwo (1,h)$.
Roughly speaking, the reader can expect the dictionary to be as follows:
\begin{itemize}
    \item \ref{it:MT:T equivariant} makes $\HleftB$ well defined, as we have seen earlier;
    \item 
    {Equation~\eqref{eq:Pi and Phi}} yields (A1), i.e., $g\HleftB\mvisiblespace$ is a map from ${(L_{1})}_{x}$ to ${(L_{1})}_{g\HleftG x}$ (linearity comes for free);
    \item \ref{it:MT:associativity, left} yields (A2), i.e., $g'\HleftB (g\HleftB\mvisiblespace)=
    g'g\HleftB\mvisiblespace$;
    \item \ref{it:MT:diagram,top} yields (A3), i.e., units act trivially;
    \item \ref{it:MT:associativity, right} yields (A4), i.e., 
    $g$ acts as expected
    on a product in ${L_{1}}$;
    and lastly
    \item \ref{it:MT:diagram,top} and~\ref{it:MT:associativity, right} combined yield (A5), i.e.,
    the involution of $g\HleftB [\lambda,e]$ is as expected.   
\end{itemize}

(A1) Take $(g,e)\in \cG_{2}\bfpsr \Sigma_{1}$. Because of 
Equation~\eqref{eq:Pi and Phi}, we have for $f\astonetwo (1,h)=\mathmbox{\Phi ((1,g)\asttwoone e)}$ that
\[
    \pi_{1}(f) h\overset{{\eqref{eq:Pi and Phi}}}{=} g\pi_{1}(e)
    \overset{\eqref{eq:2 decompositions in ZS product}}{=}
    ( g\HleftG \pi_{1}(e))( g\HrightG \pi_{1}(e)).
\]
Since $\pi_{1}(f), g\HleftG \pi_{1}(e) \in \cG_{1}$ and $h, g\HrightG \pi_{1}(e)\in \cG_{2}$, the uniqueness of the decomposition of elements in $\cG_{1}\bowtie\cG_{2}$ implies that
\begin{equation}\label{eq:f and h}
    \pi_{1}(f) = g\HleftG \pi_{1}(e) 
    \quad\text{and}\quad
    h = g\HrightG \pi_{1}(e)
    \tag{$\clubsuit$}
\end{equation}
meaning that $g\HleftB [\lambda,e]$ is an element of ${(L_{1})}_{g\HleftG \pi_{1}(e)}$. In other words, $g\HleftB\mvisiblespace$ is indeed a map ${(L_{1})}_{\pi_{1}(e)}\to {(L_{1})}_{g\HleftG \pi_{1}(e)}$. Since $g\HleftB\mvisiblespace$ only multiplies the $\mathbb{C}$-component by a scalar, it is a linear map, proving Condition~(A1).

(A2) Suppose $(g',g)\in \cG_{2}^{(2)}$. Because of Condition~\ref{it:MT:associativity, left}, it follows from $\mathmbox{\Phi ((1,g)\asttwoone e)}=\mathmbox{f\astonetwo (1,h)}$ that
\[
    \Phi ((1,g'g)\asttwoone e)=
    \Phi((1,g')\asttwoone f) \bullet (1,h)
\]
Assume that $f'$ is such that $\Phi((1,g')\asttwoone f)=f'\astonetwo (1,h')$, so that
\[
    g' \HleftB (g\HleftB[\lambda,e])
    =
    g'\HleftB [\lambda, f]
    =
    [\lambda,f']
\]
and
\[
    \Phi ((1,g'g)\asttwoone e)
    =
    f'\astonetwo (1,h'h).
\]
The latter implies
\[
    (g' g)\HleftB[\lambda,e]
    =
    [\lambda,f']
    =
    g' \HleftB (g\HleftB[\lambda,e]),
\]
which proves that $g'\HleftB (g\HleftB\mvisiblespace)=(g'g)\HleftB\mvisiblespace$.

(A3) Assume $g=u\in\cU$. Then since $\Phi$ is compatible with the inclusion maps, we have
\[
    \Phi((1,u)\asttwoone u) \overset{\ref{it:MT:diagram,top}}{=} u\astonetwo (1,u),
\]
so that
\(
    u\HleftB [\lambda, u] = [\lambda,u],
\)
which proves that $u\HleftB\mvisiblespace$ is the identity map.

(A4) Suppose that $(e,e')\in\Sigma_{2}^{(2)}$. Because of Condition~\ref{it:MT:associativity, right}, it follows from $\Phi ((1,g)\asttwoone e)=f\astonetwo (1,h)$ that
\[
    \Phi ((1,g)\asttwoone ee')
    =
    f \cdot \Phi((1,h)\asttwoone e').
\]
Assume that $\Phi((1,h)\asttwoone e')=f'\astonetwo (1,h')$, so that
\[  
    (g\HrightG \pi_{1}(e))\HleftB [\lambda',e']
    \overset{\eqref{eq:f and h}}{=}
    h\HleftB [\lambda',e']
    =
    [\lambda', f']
\]
and
\[
    \Phi ((1,g)\asttwoone ee')
    =
    (ff')\astonetwo (1,h')
\]
The latter implies
\[
   g\HleftB [\lambda\lambda',ee']
    =
    [\lambda\lambda',ff']
    =
    [\lambda,f][\lambda',f'],
\]
so that we have shown
\[
    g\HleftB
   [\lambda\lambda',ee']
    = (g\HleftB [\lambda,e])
    \bigl(
    (g\HrightG \pi_{1}(e))\HleftB [\lambda',e']
    \bigr),
\]
as needed.

(A5) We have
\begin{align*}
    r(g)\astonetwo (1,g)
    &\overset{\ref{it:MT:diagram,top}}{=}
    \Phi\bigl( (1,g) \asttwoone ee\inv\bigr)
    \overset{\ref{it:MT:associativity, right}}{=}
    f\bullet \Phi((1,h)\asttwoone e\inv),
\end{align*} 
or in other words
\begin{align*}
    f\inv\astonetwo (1,g)
    &=
    \Phi((1,h)\asttwoone e\inv).
\end{align*}
This explains $(\dagger)$ in the following:
\begin{align*}
    (g \HrightG \pi_{1}(e))\HleftB [\lambda,e]^*
    &
    \overset{\eqref{eq:f and h}}{=}
    h\HleftB [\overline{\lambda},e\inv]
    \overset{(\dagger)}{=}
    [\overline{\lambda},f\inv]
    =
    [\lambda,f]^*
    =
    (g\HleftB [\lambda,e])^*.
\end{align*}
This concludes the proof of the claim.
\end{proof}

Conversely, suppose we are given a $(\cG_{1},\cG_{2})$-compatible $\cG_{2}$-action  on ${L_{1}}$,
and let $p\colon {L_{1}}^\times\to \Sigma_{1}$ be given by $p([\lambda,e])= \mathrm{Ph}(\lambda)\cdot e$. Note that, in particular, $p(\mu \xi) = \mathrm{Ph}(\mu)\cdot p(\xi) =  p\bigl( \mathrm{Ph}(\mu)\xi\bigr)$ for all $\xi\in {L_{1}}$ and all $\mu\in\mathbb{C}^\times$, so that linearity of $g\HleftB\mvisiblespace$ implies for $\xi=[\lambda,e]$ that
\begin{align}\label{eq:p twice}
    p( g\HleftB \xi)
    &=
    p\bigl( \lambda( g\HleftB [1, e])\bigr)
    =
    p\bigl( \mathrm{Ph}(\lambda) (g\HleftB [1, e])\bigr)
     =
    p\bigl( g\HleftB [1,p(\xi)]\bigr).
    \tag{$\spadesuit$} 
\end{align}
We will furthermore need that 
\begin{align}
    g\HleftG \pi_{1}(e)
    &=
    g\HleftG q_{1}([\lambda,e])
    \overset{\text{(A1)}}{=}
    q_{1}(g\HleftB [\lambda,e])
    =
    \pi_{1}\bigl( p(g\HleftB [\lambda,e])\bigr).
    \label{eq:p and pi}
    \tag{$\diamondsuit$}  
\end{align}
Define $\Phi_0\colon \Sigma_{2}\bfpsr \Sigma_{1}\to \Sigma_{1}\ast_{\mathbb{T}} \Sigma_{2}$ by
\[
    \Phi_0\bigl( (z,g), e\bigr)
    =
        p(g\HleftB [1,e]) \astonetwo
        (z, g\HrightG \pi_{1}(e))
    .
\]
Note that, since $s_{\cG_{2}}(g)=s_{\Sigma_{2}}(z,g)=r_{\Sigma_{1}}(e)= \rho([1,e])$, the element $g\HleftB [1,e]$ of ${L_{1}}$ is defined and lies in the fibre over $g\HleftG \pi_{1}(e)$. In particular,
\[
s_{\Sigma_{1}} (p(g\HleftB [1,e]))
=
s_{\cG_{1}} (g\HleftG \pi_{1}(e))
=
r_{\cG_{2}} (g\HrightG \pi_{1}(e))
=
r_{\Sigma_{2}} (z,g\HrightG \pi_{1}(e)),
\]
so that $\Phi_0$ indeed takes values in $\Sigma_{1}\ast_{\mathbb{T}} \Sigma_{2}$. Since $\HleftB$ is linear,  $(\dagger)$ in the following holds:
\begin{align*}
    w\cdot p(g\HleftB [1,e])
    &=
    p\bigl( w(g\HleftB [1,e])\bigr)\overset{(\dagger)}{=}
    p\bigl( g\HleftB [w,e]\bigr)
    =
    p\bigl( g\HleftB [1,w\cdot e]\bigr).
\end{align*}
This explains $(\ddagger)$ in
\begin{align*}
    \Phi_0\bigl( (wz,g), e\bigr)
    &=
    p(g\HleftB [1,e]) \astonetwo
    (wz, g\HrightG \pi_{1}(e))
    \\
    &\overset{(\ddagger)}{=}
    p\bigl( g\HleftB [1,w\cdot e]\bigr)\astonetwo
    (z, g\HrightG \pi_{1}(e))
    =
    \Phi_0\bigl( (z,g), w\cdot e\bigr).
\end{align*}
This proves that $\Phi_0$ is constant on $\mathbb{T}$-equivalence classes, i.e., it factors through a map 
\begin{align*}
\Phi\colon &&\Sigma_{2}\ast_{\mathbb{T}} \Sigma_{1}&\to \Sigma_{1}\ast_{\mathbb{T}} \Sigma_{2},
&
(z,g) \asttwoone e
&\mapsto
        p(g\HleftB [1,e]) \astonetwo
        (z, g\HrightG \pi_{1}(e)).
\end{align*}
Since the quotient map $\Sigma_{2}\bfpsr \Sigma_{1}\to \Sigma_{2}\ast_{\mathbb{T}} \Sigma_{1}$ is open, $\Phi$ is continuous. Moreover, its inverse is given by the unwieldy, but clearly continuous,
\begin{align*}
 f \astonetwo (w,h)
&\mapsto
\Bigl(w, 
h
\HrightG \bigl(h\inv \HleftG \pi_{1}(f\inv)\bigr)
\Bigr)
 \asttwoone
 p\Bigl(\bigl(h\inv \HrightG \pi_{1}(f\inv)\bigr)\HleftB [1, f] \Bigr)
\end{align*}

\begin{claim}\label{claim:Phi}
    $\Phi$ satisfies Conditions~\ref{it:MT:T equivariant}--\ref{it:MT:associativity, left}, so that $(\Sigma_{1},\Sigma_{2})$ is a matched pair of twists with \wordforPhi~$\Phi$.
\end{claim}

\begin{proof}[Proof of the claim.]
The dictionary is similar to that in the proof of Claim~\ref{claim:HleftB}.

   \ref{it:MT:T equivariant}         
         $\Phi$ is $\mathbb{T}$-equivariant because
\begin{align*}
    w\cdot \Phi((z,g) \asttwoone e)
    &=
    w\cdot \bigl(    p(g\HleftB [1,e]) \astonetwo
        (z, g\HrightG \pi_{1}(e))\bigr)
    \\&=
    p(g\HleftB [1,w\cdot e]) \astonetwo
        (z, g\HrightG \pi_{1}(e))
    =
    \Phi((z,g) \asttwoone (w\cdot  e)).
\end{align*}

\ref{it:MT:source and range} We compute
\begin{align*}
    r_{\Sigma_1}\bigl(p(g\HleftB [1,e])
    \bigr)
    &=
    (r_{\cG_1}\circ\pi_{1})\bigl(p(g\HleftB [1,e])
    \bigr)
    \overset{\eqref{eq:p and pi}}{=}
    r_{\cG_{1}} \bigl(g\HleftG \pi_{1}(e)\bigr)
    \overset{\ref{item:ZS2}}{=}
    r(g)=r(z,g)
\end{align*}
and
\begin{align*}
    s_{\Sigma_{2}}\bigl(z, g\HrightG \pi_{1}(e)\bigr)
    = s_{\cG_{2}}(g\HrightG \pi_{1}(e))
    \overset{\ref{item:ZS5}}{=}
    s_{\cG_{1}} (\pi_{1}(e))
    =
    s_{\Sigma_{1}}(e).
\end{align*}

  \ref{it:MT:diagram,top} 
  For $v\coloneq  s(g)=r(e)$, we have $g\HleftB [1,v] = [1,r(g)]$ by (A1) and by linearity, so that
    \begin{align*}
     \Phi\bigl( (z,g) \asttwoone v\bigr)  &=
     p(g\HleftB [1,v]) \astonetwo
        (z, g\HrightG \pi_{1}(v))
        \\&=
     p([1,r(g)]) \astonetwo
        (z, g\HrightG v)=
     r(g) \astonetwo
        (z, g),
    \end{align*}
    proving that $\iota^{2}_{1,2}=\Phi\circ \iota^{2}_{2,1}$. Likewise, $v\HrightG \pi_{1}(e)=s(e)$ by \ref{item:ZS11} and $v\HleftB [1,e]=[1,e]$ by (A3), so that 
    \begin{align*}
        \Phi((1,v) \asttwoone e)
        &=
        p(v\HleftB [1,e]) \astonetwo
        (z, v\HrightG \pi_{1}(e))
        \\&=
        p([1,e]) \astonetwo
        (z, s(e))
        =
        e\astonetwo
        (z, s(e)),
    \end{align*}
  proving that $\iota^{1}_{1,2}=\Phi\circ \iota^{1}_{2,1}$.

         \ref{it:MT:associativity, right}
         We have
         \begin{align*}
            \Phi \bigl( (z,g\HrightG \pi_{1}(e)) \asttwoone e'\bigr)
            &=
            p\bigl((g\HrightG \pi_{1}(e))\HleftB [1,e']\bigr) \astonetwo
            \bigl(z, (g\HrightG \pi_{1}(e))\HrightG \pi_{1}(e')\bigr).
         \end{align*}
         Since (A4) says that \[
            (g\HleftB [1,e])
            \bigl(
            (g\HrightG \pi_{1}(e))\HleftB [1,e']
            \bigr)
            =
            g\HleftB
           [1,ee'],
            \]
            we conclude        
         \begin{align*}
             p(g\HleftB [1,e])\bullet
            \Phi \bigl( (z,g\HrightG \pi_{1}(e)) \asttwoone e'\bigr)
            &=
            p\bigl(g\HleftB [1,ee']\bigr) \astonetwo
            \bigl(z, g\HrightG \pi_{1}(ee')\bigr)            
            \\
            &=
            \Phi \bigl( (z,g) \asttwoone ee'),
         \end{align*}
        as needed.

         \ref{it:MT:associativity, left}
         We have
         \begin{align*}
             \Phi\bigl(
             (w,g')(z,g)\asttwoone e
             \bigr)
             &=
             \Phi\bigl(
             (wz,g'g)\asttwoone e
             \bigr)
             \\
             &=
             p( (g'g) \HleftB [1,e])
             \astonetwo
             \bigl(wz,(g'g)\HrightG \pi_{1}(e)\bigr).
         \end{align*}
         On the other hand,       
         since $\pi_{1}(p(g\HleftB [1,e])) = g\HleftG \pi_{1}(e)$ by Equation~\eqref{eq:p and pi}, we get 
         \begin{align*}
             \Phi\bigl(
                (w,g')\asttwoone p(g\HleftB [1,e])
             \bigr)
             &
             =
             p\Bigl(
                g' \HleftB 
                \bigl[
                    1, p(g\HleftB[1,e])
                \bigr]
             \Bigr)
             \astonetwo
             \bigl(w, g' \HrightG (g\HleftG \pi_{1}(e))\bigr)
             \intertext{which by Equation~\eqref{eq:p twice} equals}
             &=
             p\Bigl(
                g' \HleftB (g\HleftB[1,e])
             \Bigr)
             \astonetwo
             \bigl(w, g' \HrightG (g\HleftG \pi_{1}(e))\bigr)
             .
         \end{align*}
         Thus, with \ref{item:ZS9},
         \begin{align*}
             &\Phi\bigl(
                (w,g')\asttwoone p(g\HleftB [1,e])
             \bigr)\bullet
             (z,g\HleftG\pi_{1}(e))
             \\
             &=
             p(
                (g'g) \HleftB [1,e]
             )
             \astonetwo
             \bigl(wz, (g'g) \HrightG \pi_{1}(e)\bigr)
            \\&=
             \Phi\bigl(
             (w,g')(z,g)\asttwoone e
             \bigr),
        \end{align*}
         as needed.
\end{proof}

This finishes the proof of the claims made in Example~\ref{ex:MT if one is trivial}. 

\subsection{Regarding Example~\ref{ex:MT if one is trivial,pt2}}
We next consider the Fell line bundle $L\coloneq \mathbb{C}\times_{\mathbb{T}} (\Sigma_{1}\ZST \Sigma_{2})$ associated to the \ZS\ twist $\Sigma_{1}\ZST \Sigma_{2}$, and the \ZS\ product   $L_{1}\bowtie \cG_{2} $ of the Fell bundle $L_{1}$ and the groupoid $\cG_{2}$. We claim that the map $
    \Omega
    $ given in Equation~\eqref{eq:Omega} exists and is the claimed isomorphism. 
    
    To construct $
    \Omega
    $, consider the map
\begin{align*}
    {\Omega_{0}}\colon 
    \mathbb{C}\times (\Sigma_{1}\bfpsr\Sigma_{2}) \to 
    L_{1}\bowtie \cG_{2} 
    ,
    \quad
    \bigl(\lambda, e , (z,g )\bigr)
    \mapsto 
    \bigl(
    [\lambda z,  e ], g 
    \bigr),
\end{align*}
clearly satisfies for $w\in\mathbb{T}$, 
\begin{align*}
    {\Omega_{0}}\bigl(\lambda, e , (wz,g )\bigr)   
    =
    {\Omega_{0}}\bigl(w\lambda, e , (z,g )\bigr)
    =
    {\Omega_{0}}\bigl(\lambda, w\cdot  e , (z,g )\bigr),
\end{align*}
so it factors through a map
\begin{align*}
    \Omega\colon 
    L=\mathbb{C}\times_{\mathbb{T}} (\Sigma_{1}\ZST \Sigma_{2}) \to 
    L_{1}\bowtie \cG_{2} 
    ,
    \quad
    \bigl[
        \lambda, e \astonetwo (z,g )
    \bigr]
    \mapsto 
    \bigl(
    [\lambda z,  e ], g 
    \bigr)
    .
\end{align*}
An application of \cite[Propositions A.7 and A.8]{DL:MJM2023} shows that $\Omega$ is an isomorphism of \usc\ Banach bundles. To see that $\Omega$ is multiplicative, let $\xi_{i}=[
        \lambda_{i}, e_{i} \astonetwo (z_{i},g_{i} )
    ]\in L $ be such that $s_{\cG_{2}}(g_{1})=r_{\Sigma_{1}}(e_{2})$, so that
\begin{align*}
    \xi_{1}\xi_{2}
    &=
    \bigl[
        \lambda_{1}\lambda_{2}, e_{1} \bullet\Phi((z_{1},g_{1} )\asttwoone e_{2})\bullet (z_{2},g_{2} )
    \bigr]
    .
\end{align*}
If $f\in\Sigma_{1}$ is such that $f\astonetwo (1,g_{1}\HrightG\pi_{1}(e_{2})) =
    \Phi((1,g_{1})\asttwoone e_{2})$, then
    the definition of $\HleftB$ at~\eqref{eq:def HleftB} gives
\begin{equation}\label{eq:our choice of f}
    [z_{1}\lambda_{1},  e_{1} ]
    \left(
        g_{1} \HleftB 
    [z_{2}\lambda_{2},  e_{2} ]
    \right)
    \overset{\eqref{eq:def HleftB}}{=}
    [z_{1}\lambda_{1},  e_{1} ]
    [z_{2}\lambda_{2},  f ]
    =[z_{1}z_{2}\lambda_{1}\lambda_{2}, e_{1}f].
\end{equation}
Furthermore,
\begin{align*}
    e_{1} \bullet\Phi((z_{1},g_{1} )\asttwoone e_{2})\bullet (z_{2},g_{2} )
    &=
    (e_{1}f)\astonetwo
    \bigl(
    (z_{1},g_{1}\HrightG\pi_{1}(e_{2}))
    (z_{2},g_{2} )
    \bigr)
    \\
    &=
    (e_{1}f)\astonetwo
    \bigl(z_{1}z_{2},
    (g_{1}\HrightG\pi_{1}(e_{2}))g_{2} 
    \bigr),
\end{align*}
so that
\begin{align*}
     \Omega(
    \xi_{1}\xi_{2}
     )
    &=
     \Omega\left(
    \bigl[
        \lambda_{1}\lambda_{2}, (e_{1}f)\astonetwo
        \bigl(z_{1}z_{2},
        (g_{1}\HrightG\pi_{1}(e_{2}))g_{2} 
        \bigr)
    \bigr]
     \right)
     \\
     &=
      \bigl(
        [z_{1}z_{2}\lambda_{1}\lambda_{2}, e_{1}f],
        (g_{1}\HrightG\pi_{1}(e_{2}))g_{2} 
    \bigr)
    \\
    &=
    \Bigl(
    [z_{1}\lambda_{1},  e_{1} ]
    \left(
        g_{1} \HleftB 
    [z_{2}\lambda_{2},  e_{2} ]
    \right)
    , 
        (g_{1}\HrightG\pi_{1}(e_{2}))g_{2} 
    \Bigr)
    &&\text{%
    by Eq.~\eqref{eq:our choice of f}
    }%
    \\
    &
    =
    \bigl(
    [z_{1}\lambda_{1},  e_{1} ], g_{1}
    \bigr)
    \bigl(
    [z_{2}\lambda_{2},  e_{2} ], g_{2}
    \bigr)
    &&
    \text{by def'n of $L_{1}\bowtie \cG_{2}$ (Eq.~\eqref{eq:Fbl multiplication})}
    \\
    &=\Omega(\xi_{1})\Omega(\xi_{2}).
\end{align*}
To see that $\Omega$ is $*$-preserving, we compute
\begin{align*}
   \bigl[
       \lambda, e \astonetwo (z,g )
   \bigr]^*
   &=
   \bigl[
       \overline{\lambda}, \left(e \astonetwo (z,g )\right)\inv
   \bigr]
   =
   \bigl[
       \overline{\lambda}, 
       \Phi\left((z,g )\inv \asttwoone e\inv \right)
   \bigr].
\end{align*}
Thus, if $h\in\Sigma_{1}$ is such that $h\astonetwo (1,g\inv \HrightG\pi_{1}(e\inv)) = \Phi((1,g\inv )\asttwoone e\inv)$, then
 again by   the definition of $\HleftB$ at~\eqref{eq:def HleftB}, we have
\begin{equation}\label{eq:our choice of h}
   g\inv \HleftB [z\lambda,e]^*
   =
   g\inv \HleftB [\overline{z\lambda},e\inv]
   =
   [\overline{z\lambda},h].
\end{equation}
Lastly,
\begin{align*}
    \Omega\left(
   \bigl[
       \lambda, e \astonetwo (z,g )
   \bigr]^*
   \right)
   &=
   \Omega\left(
   \bigl[
       \overline{\lambda}, 
       h\astonetwo (\overline{z},g\inv \HrightG\pi_{1}(e\inv))
   \bigr]
   \right)
   \\
   &=
   \bigl(
       [\overline{z\lambda}, 
       h],
       g\inv \HrightG\pi_{1}(e\inv)
   \bigr)\\
   &
   =
    \bigl(
    g\inv \HleftB [z\lambda,  e ]^*, g\inv \HrightG \pi_{1}(e)\inv 
    \bigr)    
    &&\text{by 
    Eq.~\eqref{eq:our choice of h}
    }
\\&=
\bigl(
    [z\lambda,  e ], g 
    \bigr)^*
    =
    \Omega\left(
    \bigl[
        \lambda, e \astonetwo (z,g )
    \bigr]
    \right)^*.
\end{align*}
This concludes our proof that $\Omega\colon L\to \cB$ is an isomorphism of Fell bundles.
    
\printbibliography
\end{document}